\documentclass[a4paper,10pt]{article}

\usepackage{geometry}
\usepackage{epsfig}
\usepackage{amsmath}
\usepackage{amssymb}
\usepackage{amsthm}
\usepackage{mathrsfs}
\usepackage{color}
\usepackage{amscd}

\usepackage{amsthm}
\newtheorem{theorem}{Theorem}[section]
\newtheorem{definition}[theorem]{Definition}
\newtheorem{lemma}[theorem]{Lemma}
\newtheorem{proposition}[theorem]{Proposition}
\newtheorem{corollary}[theorem]{Corollary}
\newtheorem{remark}[theorem]{Remark}
\newtheorem{example}[theorem]{Example}
\newtheorem{examples}[theorem]{Examples}
\newtheorem{assumption}[theorem]{Assumption}

\usepackage{amsfonts}

\newcommand{\hh}{{\mathbb{H}}}

\newcommand{\cc}{{\mathbb{C}}}
\newcommand{\rr}{{\mathbb{R}}}
\newcommand{\zz}{{\mathbb{Z}}}
\newcommand{\nn}{{\mathbb{N}}}

\newcommand{\s}{{\mathbb{S}}}

\newcommand{\sr}{\mathcal{SR}}

\newcommand{\I}{\mathcal{I}}

\newcommand{\scap}{\mathrm{Cap}}
\newcommand{\core}{\mathrm{core}}
\newcommand{\spine}{\mathrm{spine}}

\newcommand{\debar}{\overline{\partial}}

\newcommand\vs[1]{{#1}_s^\circ}

\newcommand\re{\operatorname{Re}}
\newcommand\im{\operatorname{Im}}

\newcommand{\ui}{\imath}

\newcommand{\OO}{\Omega}

\newcommand{\mscr}{\mathscr}
\newcommand{\mc}{\mathcal}
\newcommand{\hslashslash}{%
  \raisebox{.9ex}{%
    \scalebox{.7}{%
      \rotatebox[origin=c]{18}{$-$}%
    }%
  }%
}
\newcommand{\fslash}{%
  {%
   \vphantom{f}%
   \ooalign{\kern.05em\smash{\hslashslash}\hidewidth\cr$f$\cr}%
   \kern.05em
  }%
}

\title{\bf  Quaternionic slice regularity\\ beyond slice domains}

\author{Riccardo Ghiloni\\
\small Dipartimento di Matematica, Universit\`a di Trento\\ 
\small Via Sommarive 14, I-38123 Povo Trento, Italy\\
\small riccardo.ghiloni@unitn.it\\
\and
Caterina Stoppato
\\ 
\small Dipartimento di Matematica e Informatica ``U. Dini'', Universit\`a di Firenze \\
\small Viale Morgagni 67/A, I-50134 Firenze, Italy\\
\small caterina.stoppato@unifi.it}

\date{  }

\makeatletter

\def\@tocline#1#2#3#4#5#6#7{\relax

\ifnum #1>\c@tocdepth 

\else

\par \addpenalty\@secpenalty\addvspace{#2}%

\begingroup \hyphenpenalty\@M

\@ifempty{#4}{%

\@tempdima\csname r@tocindent\number#1\endcsname\relax

}{%

\@tempdima#4\relax

}%

\parindent\z@ \leftskip#3\relax \advance\leftskip\@tempdima\relax

\rightskip\@pnumwidth plus4em \parfillskip-\@pnumwidth

#5\leavevmode\hskip-\@tempdima

\ifcase #1

\or\or \hskip 1em \or \hskip 2em \else \hskip 3em \fi%

#6\nobreak\relax

\dotfill\hbox to\@pnumwidth{\@tocpagenum{#7}}\par

\nobreak

\endgroup

\fi}

\makeatother

\begin{document}

\maketitle

\vfill
\begin{abstract}
After Gentili and Struppa introduced in 2006 the theory of quaternionic slice regular function, the theory has focused on functions on the so-called slice domains. The present work defines the class of speared domains, which is a rather broad extension of the class of slice domains, and proves that the theory is extremely interesting on speared domains. A Semi-global Extension Theorem and a Semi-global Representation Formula are proven for slice regular functions on speared domains: they generalize and strengthen some known local properties of slice regular functions on slice domains. A proper subclass of speared domains, called hinged domains, is defined and studied in detail. For slice regular functions on a hinged domain, a Global Extension Theorem and a Global Representation Formula are proven. The new results are based on a novel approach: one can associate to each slice regular function $f:\OO\to\hh$ a family of holomorphic stem functions and a family of induced slice regular functions. As we tighten the hypotheses on $\OO$ (from an arbitrary quaternionic domain to a speared domain, to a hinged domain), these families represent $f$ better and better and allow to prove increasingly stronger results.
\end{abstract}
\vfill
\vfill

\newpage

\setcounter{tocdepth}{2}

\tableofcontents

\section{Introduction}\label{sec:introduction}

This work aims at significantly extending the theory of quaternionic slice regular functions introduced in~\cite{cras,advances} (see Definition~\ref{def:cullenregular}). The extension performed is manifold.

While the theory has focused thus far on the so-called slice domains (see Definition~\ref{def:slicedomain}), an important achievement here is the definition of a wider class of quaternionic domains, called \emph{speared domains} (see Definition~\ref{def:speared}). Slice regular functions on speared domains turn out to inherit many of the nice properties proven over slice domains in~\cite{geometricfunctiontheory,localrepresentation}. We point out that carefully choosing the domains of definition of slice regular functions is crucial, as pathological examples can be constructed on quaternionic domains that do not intersect the real axis. Slice regular functions on domains included in $\hh\setminus\rr$ need not be continuous, see~\cite[Example 1.11]{librospringer2ed}.

The new approach adopted to study slice regularity on a general speared domain $\OO$ is so powerful that it allows to strengthen the Local Extension Theorem and the Local Representation Formula of~\cite{localrepresentation} to what we call the Semi-global Extension Theorem~\ref{thm:localextension} and the Semi-global Representation Formula (Corollary~\ref{cor:localrepresentationformula}). These new results are stronger than the aforementioned local results even when $\OO$ happens to be a slice domain.

Another significant novelty concerns the Extension Theorem, which was proven in~\cite{localrepresentation} for the so-called simple slice domains (see Definition~\ref{def:simple}) and is known to be false for general slice domains, as a counterexample appeared in~\cite[Pages 4--5]{douren1}. To see explicitly that that specific example is not a simple slice domain, see~\cite[Example 4.4]{geometricfunctiontheory}. Another related article is~\cite{dourensabadini}. In the present work, we prove the Global Extension Theorem~\ref{thm:extension} on the class of \emph{hinged domains} (see Definition~\ref{def:hingeddomains}), which is considerably wider than the class of simple slice domains. The result is not obtained directly under the global hypothesis that the domain is hinged. Instead, several preliminary steps are taken on a general speared domain $\OO$. These include constructing an equivalence relation $\sim$ on $\OO$ and proving, in Theorem~\ref{thm:hingedpoints}, that if $f$ is a slice regular function on $\OO$ and $x,x'$ are equivalent points of $\OO$, then the local extensions of $f$ near $x$ and $x'$ are consistent. Only later we define hinged domains by means of the equivalence relation $\sim$ and prove that a slice regular function on a hinged domain $\OO$ automatically extends to a unique slice regular function on the so-called symmetric completion of $\OO$ (see Subsection~\ref{subsec:quaternions}).

The paper is organized as follows.

Section~\ref{sec:preliminaries} is devoted to preliminary material: Subsection~\ref{subsec:quaternions} sets up the complex and quaternionic algebras; Subsection~\ref{subsec:regular} recalls the definition and the first properties of slice regular functions; Subsection~\ref{subsec:slicefunctions} recalls the alternate approach to regularity proposed in~\cite{perotti} by introducing the function class $\mc{SR}$.

Section~\ref{sec:generalextensionformula} provides a novel interpretation of the General Extension Formula of~\cite{advancesrevised}. This makes it possible to associate to a slice regular function $f$ on an arbitrary domain $\OO\subseteq\hh$ a double-index family of elements of $\mc{SR}$ strictly related to $f$.

Section~\ref{sec:localrepresentationformula} defines the concept of speared domain. For a slice regular function $f$ on a speared domain $\OO\subseteq\hh$, the double-index family associated to $f$ reduces to a single-index family. This allows us to prove the aforementioned Semi-global Extension Theorem and Representation Formula, which open the path for a detailed analysis of slice regularity on speared domains.

Section~\ref{sec:hinged} is divided into two parts. Subsection~\ref{subsec:hingedpoints} constructs, on an arbitrary speared domain $\OO\subseteq\hh$, the aforementioned equivalence relation $\sim$. For two equivalent points $x,x'$, the corresponding elements of the single-index family are strictly related. This allows to prove that the extension of $f$ near $x$ is consistent with the extension of $f$ near $x'$, as we already mentioned. Subsection~\ref{subsec:extensiontheorem} defines hinged domains as those for which the equivalence relation $\sim$ becomes trivial. The Global Extension Theorem for hinged domains is therefore an immediate consequence of the work done in the previous subsection. A Global Representation Formula follows.

Section~\ref{sec:classesofdomains} studies speared domains and hinged domains in full detail. Subsection~\ref{subsec:speared} proves that the class of speared domains strictly includes the class of slice domains and provides a large family of examples. It also proves that every speared domain can be locally shrunk to a slice domain. The final Subsection~\ref{subsec:hingeddomains} presents several subclasses of the class of hinged domains: \emph{spear-simple} domains, including simple slice domains; \emph{$\s$-connected} speared domains; and speared domains having a \emph{main sail}. A wealth of examples is provided, to show that all these subclasses are distinct and that there exist hinged domains belonging to none of these subclasses.


\section{Preliminaries}\label{sec:preliminaries}

\subsection{Real quaternions and their complexification}\label{subsec:quaternions}

Let $\cc,\hh$ denote the real $^*$-algebras of complex numbers and quaternions, respectively. As explained in~\cite{ebbinghaus,libroward}, they can be built from the real field $\rr$ by means of the so-called Cayley-Dickson construction:

\begin{itemize}
\item $\cc=\rr+i\rr$, $(\alpha+i\beta)(\gamma+i\delta)=\alpha\gamma-\beta\delta+i(\alpha\delta+\beta\gamma)$, $(\alpha+i\beta)^c=\alpha-i\beta\ \forall\,\alpha,\beta,\gamma,\delta\in\rr$,
\item $\hh=\cc+j\cc$, $(\alpha+j\beta)(\gamma+j\delta)=\alpha\gamma-\beta^c\delta+j(\alpha^c\delta+\beta\gamma)$, $(\alpha+j\beta)^c=\alpha^c-j\beta\ \forall\,\alpha,\beta,\gamma,\delta\in\cc$.
\end{itemize}

On the one hand, this construction endows the two real vector spaces with a bilinear multiplicative operation, which makes each of them a (real) \emph{algebra}. By construction, each of them is \emph{unitary}, that is, it has a multiplicative neutral element $1$; and $\rr$ is identified with the subalgebra generated by $1$. For the purpose of the present work, every time we will speak of an algebra or of a subalgebra, we will automatically imply that it is unitary. It is well-known that $\hh$ is associative but not commutative. The \emph{center} $\{r\in\hh\, |\, rx=xr\ \forall\,x\in \hh\}$ of the algebra $\hh$ is $\rr$.

On the other hand, the Cayley-Dickson construction endows each of $\cc,\hh$ with a \emph{$^*$-involution}, i.e., a (real) linear transformation $x\mapsto x^c$ with the following properties: $(x^c)^c=x$ and $(xy)^c=y^cx^c$ for every $x,y$; $x^c=x$ for every $x \in \rr$. Thus, $\cc$ and $\hh$ are \emph{$^*$-algebras}. The \emph{trace} and \emph{norm} functions, defined by the formulas
\begin{equation}\label{traceandnorm}
t(x):=x+x^c \quad \text{and} \quad n(x):=xx^c,
\end{equation}
are real-valued. Moreover, in $\cc\simeq\rr^2$ and $\hh\simeq\rr^4$, the expression $\frac{1}{2}t(xy^c)$ coincides with the standard scalar product $\langle x,y\rangle$ and $n(x)$ coincides with the squared Euclidean norm $\Vert x\Vert^2$. In particular, these algebras are \emph{nonsingular}, i.e., $n(x)=xx^c=0$ implies $x=0$. 
The trace function $t$ vanishes on every commutator $[x,y]:=xy-yx$ (see~\cite[Lemma 5.6]{gpsalgebra}). It holds $n(xy) = n(x) n(y)$, or equivalently $\Vert xy\Vert = \Vert x\Vert\,\Vert y\Vert$. We point out that $(r+v)^c=r-v$ for all $r \in \rr$ and all $v$ in the Euclidean orthogonal complement of $\rr$.

Every nonzero element $x$ of $\cc$ or $\hh$ has a multiplicative inverse, namely $x^{-1} = n(x)^{-1} x^c = x^c\, n(x)^{-1}$. For all elements $x,y\neq 0$, the equality $(xy)^{-1} = y^{-1}x^{-1}$ holds. As a consequence, each of the algebras $\cc$ and $\hh$ is a division algebra and has no zero divisors. 
A famous result due to Frobenius states that $\rr,\cc,\hh$ are the only (finite-dimensional) associative division algebras.

Let us consider the $2$-sphere of quaternionic \emph{imaginary units}
\begin{equation} \label{eq:s_A}
\s:=\{x \in \hh\, | \, t(x)=0, n(x)=1\} = \{w \in \hh \, |\, w^2=-1\}\,.
\end{equation}
The $^*$-subalgebra generated by any $I \in \s$, i.e., $\cc_I=\mathrm{span}(1,I)$, is $^*$-isomorphic to the complex field $\cc$ (endowed with the standard multiplication and conjugation) through the $^*$-isomorphism 
\[\phi_I\ :\ \cc\to\cc_I,\quad \alpha+i\beta\mapsto\alpha+\beta I\,.\]
The union
\begin{equation} \label{eq:slice}
\text{$\bigcup_{I \in \s}\cc_I$}
\end{equation}
coincides with the entire algebra $\hh$. Moreover, $\cc_I \cap \cc_J=\rr$ for every $I,J \in \s$ with $I \neq \pm J$. As a consequence, every element $x$ of $\hh \setminus \rr$ can be expressed as follows: $x=\alpha+\beta I$, where $\alpha \in \rr$ is uniquely determined by $x$, while $\beta \in \rr$ and $I \in \s$ are uniquely determined by $x$, but only up to sign. If $x \in \rr$, then $\alpha=x$, $\beta=0$ and $I$ can be chosen arbitrarily in $\s$. Therefore, it makes sense to define the \emph{real part} $\re(x)$ and the \emph{imaginary part} $\im(x)$ by setting $\re(x):=t(x)/2=(x+x^c)/2 = \alpha$ and $\im(x):=x-\re(x)=(x-x^c)/2=\beta I$. It also makes sense to call the Euclidean norm $\Vert x\Vert=\sqrt{n(x)}=\sqrt{a^2+\beta^2}$ the \emph{modulus} of $x$ and to denote it as $|x|$. The algebra $\hh$ has the following useful property.
\begin{itemize}
\item {[Splitting property]} For each imaginary unit $I \in \s$, there exists $J \in \hh$ such that $\{1,I,J,IJ\}$ is a real vector basis of $\hh$, called a \emph{splitting basis} of $\hh$ associated to $I$. Moreover, $J$ can be chosen to be an imaginary unit and to make the basis orthonormal.
\end{itemize}
We consider on $\cc$ and $\hh$ the natural Euclidean topology and differential structure as a finite-dimensional real vector space. Similarly, we will adopt on $\cc$ and $\hh$ the natural structure of real analytic manifold. The relative topology on each $\cc_J$ with $J \in \s$ clearly agrees with the topology determined by the natural identification $\phi_J$ between $\cc$ and $\cc_J$. A subset $T$ of $\hh$ is called \emph{(axially) symmetric} (or \emph{circular}) if it is symmetric with respect to the real axis $\rr$; i.e., if for all $x=\alpha+\beta I\in T$ (with $\alpha,\beta\in\rr$ and $I\in\s$), the whole set
\[\s_x:=\alpha+\beta \, \s=\{\alpha+\beta I \in \hh \, | \, I \in \s\}\]
is included in $T$. We observe that $\s_x=\{x\}$ if $x \in \rr$. On the other hand, for $x \in \hh \setminus \rr$, the set $\s_x$ is obtained by real translation and dilation from the $2$-sphere $\s$. For each $T\subseteq\hh$, the \emph{symmetric completion} of $T$ is the set
\[\widetilde T:=\bigcup_{x\in T}\s_x\,.\]
For each $E\subseteq\cc$, the \emph{circularization} $\OO_E$ of $E$ is defined as the following subset of $\hh$:
\[\OO_E:=\left\{x \in \hh \, \big| \, \exists \alpha,\beta \in \rr, \exists I \in \s \mathrm{\ s.t.\ } x=\alpha+\beta I, \alpha+\beta i \in E\right\}\,.\]
For instance, given $x=\alpha+\beta I \in \hh$, we have that $\s_x$ is both the symmetric completion of the singleton $\{x\}\subset\hh$ and the circularization of the singleton $\{\alpha+i\beta\}\subset\cc$.

In this work, we will also use the complexified algebra $\hh_{\cc}=\hh \otimes_{\rr} \cc=\{x+\ui y \, | \, x,y \in \hh\}$, endowed with the following product:
\[
(x+\ui y)(x'+\ui y')=xx'-yy'+\ui (xy'+yx').
\]
In addition to the complex conjugation $\overline{x+\ui y}=x-\ui y$, $\hh_\cc$ is endowed with a $^*$-involution $x+\ui y \mapsto (x+\ui y)^c := x^c+\ui y^c$, which makes it a $^*$-algebra. This $^*$-algebra is still associative and it has center $\rr_\cc=\rr+\ui\rr$. If  we identify $\rr_\cc$ with $\cc$ and if $I\in\s$, then the previously defined $*$-isomorphism $\phi_I:\cc\to\cc_I$ extends to
\[\phi_I\ :\ \hh_\cc \to \hh\quad x+\ui y\mapsto x+Iy\,.\]
Throughout the paper, we will keep referring to $\rr_\cc=\rr+\ui\rr$, which is a distinct subset of $\hh_\cc=\hh+\ui\hh$ than $\cc=\rr+i\rr\subset\hh\subset\hh_\cc$. For instance: $\ui$ commutes with every element of $\hh_\cc$, while $i$ does not. We point out that the extended $\phi_I$ is still a surjective $^*$-algebra morphism, but is no longer injective. For future reference, we make the following remark. We recall that: a \emph{constant complex structure} on $\rr^{n_0}$ is an $\rr$-linear endomorphism $\mc{J}_0:\rr^{n_0}\to\rr^{n_0}$ such that $\mc{J}_0\circ\mc{J}_0=-id_{\rr^{n_0}}$; if $\mc{J}_0,\mc{J}_1$ are constant complex structures on $\rr^{n_0},\rr^{n_1}$, respectively, then a \emph{holomorphic map} $\psi:(\rr^{n_0},\mc{J}_0)\to(\rr^{n_1},\mc{J}_1)$ is a map $\psi:\rr^{n_0}\to\rr^{n_1}$ such that $\psi\circ\mc{J}_0=\mc{J}_1\circ\psi$.

\begin{remark}\label{rmk:phiholomorphic}
Left multiplication by $\ui$ is a constant complex structure on $\hh_\cc$, which preserves $\rr_\cc$. For each $I\in\s$, left multiplication by $I$ is a constant complex structure on $\hh$, which preserves $\cc_I$. The map $\phi_I$ is a holomorphic map from $(\hh_\cc,\ui)$ to $(\hh,I)$ and its restriction to $\rr_\cc$ is a biholomorphic map from $(\rr_\cc,\ui)$ to $(\cc_I,I)$.
\end{remark}

\subsection{Slice regular functions}\label{subsec:regular}

As usual, a nonempty open connected subset of $\cc$ is called a \emph{(complex) domain}. Similarly, a nonempty connected open subset of $\hh$ is called a \emph{(quaternionic) domain}. The following definition was given in~\cite{advancesrevised} and is a slight modification of~\cite{cras,advances}. It has been adopted as standard in the monograph~\cite{librospringer2ed}.

\begin{definition}\label{def:cullenregular}
Let $\OO\subseteq\hh$ be a domain and let $f:\OO\to\hh$ be a function. For each $I\in\s$, let us set $\OO_I:=\OO\cap\cc_I$ and $f_I:=f_{|_{\OO_I}}$. The restriction $f_I$ is termed \emph{holomorphic} if it is $\mscr{C}^1$ and if
\[\debar_If_I(\alpha+\beta I):=\frac12\left(\frac{\partial}{\partial\alpha}+I\frac{\partial}{\partial\beta}\right)f_I(\alpha+\beta I)\]
vanishes identically in $\OO_I$. The function $f$ is termed \emph{slice regular} if, for each $I\in\s$, $f_I$ is holomorphic.
\end{definition}

More precisely, the definition of $\debar_I$ should be stated as follows:

\begin{equation*} \label{def:debarI}
\debar_I f(x):=\frac{1}{2}\left(\left(\frac{\partial{}}{\partial\alpha}+I\frac{\partial{}}{\partial\beta}\right)(f\circ\widehat{\phi}_I)\right)((\widehat{\phi}_I)^{-1}(x))\,,
\end{equation*}
where $\widehat{\phi}_I$ denotes the restriction of $\phi_I$ to $\phi_I^{-1}(\OO_I)$ (see~\cite{slicebyslice}). In other words, $f_I$ is called holomorphic if it is a holomorphic map from $(\OO_I,I)$ to $(\hh,I)$. The so-called ``splitting lemma'' follows: after fixing a splitting basis $\{1,I,J,IJ\}$ of $\hh$ and letting $F,G$ be the functions $\OO_I \to \cc_I$ such that $f_I=F+GJ$, we have that $f_I$ is holomorphic if, and only if $F,G$ are holomorphic functions from $(\OO_I,I)$ to $(\cc_I,I)$.

For the purposes of this paper, we will also need the next definition.

\begin{definition}\label{def:cullenregularextended}
Let $U$ be an open subset of $\hh$: then every connected component of $U$ is a domain. We call a function $f:U\to\hh$  \emph{slice regular} if it is slice regular when restricted to each connected component of $U$.
\end{definition}

Polynomials and power series are examples of slice regular functions, see~\cite[Theorem 2.1]{advances}, while every slice regular function on a Euclidean ball centered at $0$ expands into power series, see~\cite[Theorem 2.7]{advances}. We can summarize these facts as follows. For any $x_0\in\hh$ and $R\in[0,+\infty]$, we adopt the notation
\[B(x_0, R) := \{x \in \hh\ |\ |x-x_0|<R\}\,.\]

\begin{theorem}
Every polynomial of the form $\sum_{m=0}^n x^ma_m = a_0+x a_1 + \ldots x^n a_n$ with coefficients $a_0, \ldots, a_n \in \hh$ is a slice regular function on $\hh$.

Every power series of the form $\sum_{n \in \nn} x^n a_n$ converges absolutely and uniformly on compact sets in $B(0, R)$ for some $R\in[0,+\infty]$, determined by the equality $1/R=\limsup_{n\to+\infty}|a_n|^{1/n}$. If $R>0$, then the sum of the series is a slice regular function on $B(0,R)$.

Conversely, if $R\in(0,+\infty]$ and if $f:B(0,R)\to\hh$ is a slice regular function, then there exists a sequence $\{a_n\}_{n\in\nn}\subset\hh$ such that
\[f(x)=\sum_{n \in \nn} x^n a_n\]
for all $x\in B(0,R)$.
\end{theorem}

A slice regular function $f$ is called \emph{slice preserving} if, for each $I\in\s$, it maps the ``slice'' $\OO_I$ into $\cc_I$. A polynomial or power series $\sum_{n \in \nn} x^n a_n$ is slice preserving if, and only if, $\{a_n\}_{n\in\nn}\subset\rr$. The following result derives from~\cite{zeros,open}.

\begin{lemma}\label{lem:representationball}
Let $R\in(0,+\infty]$ and let $f:B(0,R)\to\hh$ be a slice regular function. For each $\alpha,\beta\in\rr$ with $\alpha^2+\beta^2<R^2$, there exist quaternions $b,c$ such that
\[f(\alpha+\beta I)=b+Ic\]
for all $I\in\s$. Namely, if $f(x)=\sum_{n \in \nn} x^n a_n$ in $B(0,R)$ and if $s_n,t_n\in\rr$ are defined for each $n\in\nn$ to fulfill the equality $(\alpha+\beta I)^n=s_n+t_nI$, then $b=\sum_{n \in \nn} s_n a_n$ and $c=\sum_{n \in \nn} t_n a_n$.
\end{lemma}

As observed in~\cite{advancesrevised}, it holds
\begin{align*}
b&=(J-K)^{-1} \left[J f(\alpha+\beta J) - K f(\alpha+\beta K)\right]\\
c&=(J-K)^{-1} \left[f(\alpha+\beta J) - f(\alpha+\beta K)\right]
\end{align*}
for all $J,K\in\s$ with $J\neq K$. The same article~\cite{advancesrevised} proved a result called General Extension Formula, upon which we will elaborate in Section~\ref{sec:generalextensionformula}.

\subsection{Slice functions}\label{subsec:slicefunctions}

Let $D$ be a nonempty subset of $\rr_\cc$ that is invariant under the complex conjugation $z=\alpha+i\beta \mapsto \overline{z}=\alpha-i\beta$: then for each $J \in \s$ the map $\phi_J(\alpha+i\beta)=\alpha+\beta J$ naturally embeds $D$ into a ``slice'' $\phi_J(D)$ of $\OO_D=\bigcup_{J\in\s}\phi_J(D)\subseteq\hh$. The following definitions were given in~\cite{perotti}.

\begin{definition}\label{def:stemfunction}
Assume $D\subseteq\rr_\cc$ to be nonempty and preserved by complex conjugation. A function $F:D\to\hh_\cc$ is called a \emph{stem function} if $F(\overline{z})=\overline{F(z)}$ for every $z \in D$.
\end{definition}

If we consider a decomposition $F=F_1+\ui F_2$ into $\hh$-valued components $F_1$ and $F_2$, $F$ is a stem function if, and only if, $F_1(\overline{z})=F_1(z)$ and $F_2(\overline{z})=-F_2(z)$ for every $z \in D$.

\begin{definition} \label{def:slice-function}
Assume $D\subseteq\rr_\cc$ to be nonempty and preserved by complex conjugation. A function $f:\OO_D \to \hh$ is called a \emph{(left) slice function} if there exists a stem function  $F:D \to \hh_\cc$ such that the diagram
\begin{equation}\label{lifting}
\begin{CD}
D @>F> >\hh_\cc\\ 
@V V \phi_J V 
@V V \phi_J V\\
\OO_D @>f> >\hh 
\end{CD}
\end{equation}
commutes for each $J\in\s$. In this situation, we say that $f$ is induced by $F$ and we write $f=\I(F)$. If $F$ is $\rr_\cc$-valued, then we say that the slice function $f$ is \emph{slice preserving}. The real vector space of slice functions on $\OO_D$ is denoted as 
$\mc{S}(\OO_D)$.
\end{definition}
It is easy to check that each slice function $f$ is induced by a unique stem function $F$.

For each slice function $f:\OO_D \to \hh$, it is useful to define the slice function $\vs f:\OO_D \to \hh$, called \emph{spherical value} of $f$, and the slice function $f'_s:\OO_D \setminus \rr \to \hh$, called \emph{spherical derivative} of $f$, by setting
\[
\vs f(x):=\frac{1}{2}(f(x)+f(x^c))
\quad \text{and} \quad
f'_s(x):=\frac{1}{2}\im(x)^{-1}(f(x)-f(x^c)).
\] 
The works~\cite{advancesrevised,perotti} showed that these functions are constant on each sphere $\s_x\subseteq\OO_D \setminus \rr$. For all $x \in \OO_D \setminus \rr$, it holds
\[
f(x) = \vs  f(x)+ \im(x)  f'_s(x)\,,
\]
where the function $\im$ is a slice preserving element of $\mc{S}(\hh)$. The function $f$ is slice preserving if, and only if, $\vs f$ and $f'_s$ are real-valued.

Some special subclasses of the algebra $\mc{S}(\OO_D)$ of slice functions have been singled out in~\cite{perotti}: the nested subspaces $\mc{S}^0(\OO_D), \mc{S}^1(\OO_D), \mc{S}^\omega(\OO_D), \sr(\OO_D)$ of slice functions induced by continuous, continuously differentiable, real analytic and holomorphic stem functions $F:D\to\hh_\cc$, respectively. By holomorphic, here, we mean that $F$ is a holomorphic map from $(D,\ui)$ to $(\hh_\cc,\ui)$. For all of these definitions except the first one, $\OO_D$ is assumed to be an open subset of $\hh$ (whence $D$ is open in $\rr_\cc$). If $D$ is a connected open subset of $\rr_\cc$ and it intersects the real line $\rr$, then $\OO_D$ is called a \emph{symmetric slice domain}. If an open $D$ does not intersect $\rr$ and has two connected components switched by complex conjugation, then $\OO_D$ is called a \emph{product domain}. A general open $\OO_D$ is a disjoint union of symmetric slice domains and product domains. 

\begin{remark}
The elements of $\sr(\OO_D)$ are automatically slice regular functions according to Definition~\ref{def:cullenregularextended}. Indeed, if $f=\I(F)$ with $F:(D,\ui)\to(\hh_\cc,\ui)$ holomorphic, then, for each $J\in\s$, the composition $\phi_J\circ F=f\circ\phi_J=f_J\circ\phi_J$ is a holomorphic map $(D,\ui)\to(\hh,J)$. This, in turn, implies that $f_J$ is a holomorphic map $((\OO_D)_J,J)\to(\hh,J)$, because the restriction of $\phi_J$ to $D$ is a biholomorphic map $(D,\ui)\to((\OO_D)_J,J)$.
\end{remark}

It has been proven in~\cite{perotti} that, when $\OO=\OO_D$ is a symmetric slice domain, a function $f:\OO_D\to\hh$ belongs to $\sr(\OO_D)$ if, and only if, it is slice regular. We will provide a new proof of this fact in the forthcoming Section~\ref{sec:generalextensionformula}. On the other hand, when $\OO_D$ is a product domain, the class $\sr(\OO_D)$ (studied, for instance, in~\cite{altavillawithoutreal}) is strictly included in the class slice regular functions $\OO_D\to\hh$. To see this, it suffices to restrict to $\OO_D$ the slice regular function appearing in the next example.

\begin{example}
Fix $I_0 \in \s$ and let $f: \hh \setminus \rr \to \hh$ be defined as follows:
\[f(q) = \left\{ 
\begin{array}{ll}
0 \ \mathrm{if} \ q \in \hh \setminus \cc_{I_0}\\
1 \ \mathrm{if} \ q \in  \cc_{I_0} \setminus \rr
\end{array}
\right.\]
The function $f$ is slice regular because each restriction $f_I$ is constant, whence holomorphic. It belongs to neither $\sr(\hh \setminus \rr)$ nor $\mc{S}(\hh \setminus \rr)$.
\end{example}

More work on the relation between $\sr(\OO_D)$ and the class of slice regular functions on $\OO_D$ includes~\cite{slicebyslice}. The recent paper~\cite{localsliceanalysis} generalizes the class $\sr$, dropping the symmetry hypothesis on $\OO_D$ and imposing the existence of local zonal decompositions. This provides an extension of the theory of slice regular functions beyond the classical setup of slice domains, distinct from the one undertaken in the present work.


\section{General extension formula}\label{sec:generalextensionformula}

\subsection{General extension formula, revisited}

This subsection elaborates the ideas behind the General Extension Formula, namely~\cite[Theorem 4.2]{advancesrevised}. This will produce new powerful results in the forthcoming Subsection~\ref{subsec:double-index}.

For all $I\in\s$ and all $T\subseteq\hh$, we will use the notations
\begin{align*}
&\cc_I^\geqslant:=\{\alpha+\beta I\,|\, \alpha,\beta\in\rr, \beta\geq0\}\,,\\
&\cc_I^>:=\{\alpha+\beta I\,|\, \alpha,\beta\in\rr, \beta>0\}\,,
\end{align*}
as well as $T_I^\geqslant:=T\cap\cc_I^{\geqslant}$ and $T_I^>:=T\cap\cc_I^>$. In the sequel, we will speak of a \emph{$\mscr{C}^1$ function} $r:T_I^\geqslant\to\hh$ according to the usual convention: there exists a $\mscr{C}^1$ function $\widetilde r$ from an open neighborhood of $T_I^\geqslant$ in $\cc_I$ to $\hh$, whose restriction to $T_I^\geqslant$ equals $r$. Similarly, a $\mscr{C}^1$ function $r:T_I^\geqslant\to\hh$ will be called \emph{holomorphic} if there exist an open neighborhood $U$ of $T_I^\geqslant$ in $\cc_I$ and a holomorphic map $\widetilde r:(U,I)\to(\hh,I)$ whose restriction to $T_I^\geqslant$ equals $r$. We point out that, if $T_I^\geqslant$ is an open subset of $\cc_I^\geqslant$ and if $r:T_I^\geqslant\to\hh$ is a $\mscr{C}^1$ function, then the real partial derivative of $r$ in an arbitrary direction within $\cc_I$ is well-defined and continuous at each point of $T_I^\geqslant$, including the real points.

\begin{proposition}\label{prop:holomorphicfamily}
Let $\OO=\OO_D$ be a symmetric open subset of $\hh$ and take distinct $J,K\in\s$. Consider $\mscr{C}^1$ functions $r:\OO_J^\geqslant\to\hh$ and $s:\OO_K^\geqslant\to\hh$. For each $I\in\s$, we define $p^I:\OO_I^\geqslant\to\hh$ by setting, for each $\alpha+\beta I\in\OO$ with $\alpha,\beta\in\rr$ and $\beta\geq0$,
\[p^I(\alpha+\beta I) := (J-K)^{-1} \left[J r(\alpha+\beta J) - K s(\alpha+\beta K)\right] + I (J-K)^{-1} \left[r(\alpha+\beta J) - s(\alpha+\beta K)\right]\,.\]
The following properties hold:
\begin{itemize}
\item For each $I\in\s$, the function $p^I$ is $\mscr{C}^1$ in $\OO_I^\geqslant$.
\item $p^J=r$ and $p^K=s$.
\item For each $I\in\s$, it holds
\begin{align*}
\debar_I p^I(\alpha+\beta I)=\,& [(J-K)^{-1}J+I(J-K)^{-1}]\,\debar_J r(\alpha+\beta J)\\
&-[(J-K)^{-1}K+I(J-K)^{-1}]\,\debar_K s(\alpha+\beta K)\,.
\end{align*}
As a consequence, if $r$ is holomorphic in $\OO_J^\geqslant$ and if $s$ is holomorphic in $\OO_K^\geqslant$, then $p^I$ is holomorphic in $\OO_I^\geqslant$.
\end{itemize}
\end{proposition}

\begin{proof}
By direct inspection in the definition of $p^I$, we find that $p^I$ is $\mscr{C}^1$ in $\OO_I^\geqslant$. For all $\alpha+\beta J\in\OO_J^\geqslant$, it holds
\begin{align*}
p^J(\alpha+\beta J)=\,& [(J-K)^{-1}J+J(J-K)^{-1}]\,r(\alpha+\beta J)\\
&-[(J-K)^{-1}K+J(J-K)^{-1}]\,s(\alpha+\beta K)\,\\
=\,&|J-K|^{-2}[(-J+K)J+J(-J+K)]\,r(\alpha+\beta J)\\
&-|J-K|^{-2}[(-J+K)K+J(-J+K)]\,s(\alpha+\beta K)\,\\
=\,&[(J-K)(K-J)]^{-1} (2 + KJ + JK)]\,r(\alpha+\beta J)\\
&-|J-K|^{-2} (-JK - 1+ 1 + JK)\,s(\alpha+\beta K)\\
=\,&r(\alpha+\beta J)\,.
\end{align*}
Analogous computations prove that $p^K$ coincides with $s$ in $\OO_K^\geqslant$. Another direct inspection in the definition of $p^I$ yields that
\begin{align*}
\frac{\partial p^I}{\partial\alpha}(\alpha+\beta I)=\,& [(J-K)^{-1}J+I(J-K)^{-1}]\,\frac{\partial r}{\partial\alpha}(\alpha+\beta J)\\
&-[(J-K)^{-1}K+I(J-K)^{-1}]\,\frac{\partial s}{\partial\alpha}(\alpha+\beta K)\,.
\end{align*}
Moreover,
\begin{align*}
I\frac{\partial p^I}{\partial\beta}(\alpha+\beta I)=\,& [I(J-K)^{-1}J-(J-K)^{-1}]\,\frac{\partial r}{\partial\beta}(\alpha+\beta J)\\
&-[I(J-K)^{-1}K-(J-K)^{-1}]\,\frac{\partial s}{\partial\beta}(\alpha+\beta K)\\
=\,&[I(J-K)^{-1}+(J-K)^{-1}J]\,J\frac{\partial r}{\partial\beta}(\alpha+\beta J)\\
&-[I(J-K)^{-1}+(J-K)^{-1}K]\,K\frac{\partial s}{\partial\beta}(\alpha+\beta K)\,.
\end{align*}
The definition $\debar_I p^I(\alpha+\beta I):=\frac12\left(\frac{\partial}{\partial\alpha}+I\frac{\partial}{\partial\beta}\right)p^I(\alpha+\beta I)$ now yields that
\[\debar_I p^I(\alpha+\beta I)=[(J-K)^{-1}J+I(J-K)^{-1}]\,\debar_J r(\alpha+\beta J)-[(J-K)^{-1}K+I(J-K)^{-1}]\,\debar_K s(\alpha+\beta K)\,.\]
It immediately follows that, if $r$ is holomorphic in $\OO_J^\geqslant$ and if $s$ is holomorphic in $\OO_K^\geqslant$, then $p^I$ is holomorphic in $\OO_I^\geqslant$.
\end{proof}

The next result will be extremely useful in the sequel. It follows from the classical Schwarz Reflection Principle if we identify $\rr_\cc$ with $\cc$ and $\hh_\cc$ with $\cc^4$. We include a direct proof for the sake of completeness. We recall that $\overline{x+\ui y}=x-\ui y$ for all $x,y\in\hh$ and we set, for all $D\subseteq\rr_\cc$, the notations
\begin{align*}
D^\geqslant&:=\{\alpha+\ui\beta\in D : \alpha,\beta\in\rr,\beta\geq0\}\,,\\
D^\leqslant&:=\{\alpha+\ui\beta\in D : \alpha,\beta\in\rr,\beta\leq0\}\,,\\
D^>&:=\{\alpha+\ui\beta\in D : \alpha,\beta\in\rr,\beta>0\}\,,\\
D^<&:=\{\alpha+\ui\beta\in D : \alpha,\beta\in\rr,\beta<0\}\,.
\end{align*}

\begin{lemma}[Schwarz reflection principle for stem functions]\label{lem:schwarzreflection}
Let $D$ be an open subset of $\rr_\cc$, symmetric with respect to conjugation $z\mapsto\bar z$. Assume $F:D^\geqslant\to\hh_\cc$ to be a holomorphic function with $\overline{F(\alpha)}=F(\alpha)$ for all $\alpha\in D\cap\rr$ (if any). Setting $F(\bar z):=\overline{F(z)}$ for all $z\in D^\geqslant$ defines a holomorphic stem function $D\to\hh_\cc$.
\end{lemma}

\begin{proof}
Since $F$ is holomorphic in $D^\geqslant$, it holds
\[\lim_{D^\geqslant\ni z\to z_0}\frac{F(z)-F(z_0)}{z-z_0}=\frac{\partial F}{\partial\alpha}(z_0)\]
for each $z_0\in D^\geqslant$. For $z_0\in D^\leqslant$, we get
\begin{align*}
\lim_{D^\leqslant\ni z\to z_0}\frac{F(z)-F(z_0)}{z-z_0}&=\lim_{D^\leqslant\ni z\to z_0}\frac{\overline{F(\bar z)}-\overline{F(\bar z_0)}}{z-z_0}\\
&=\lim_{D^\leqslant\ni z\to z_0}\overline{\left(\frac{F(\bar z)-F(\bar z_0)}{\bar z-\bar z_0}\right)}\\
&=\lim_{D^\geqslant\ni w\to\bar z_0}\overline{\left(\frac{F(w)-F(\bar z_0)}{w-\bar z_0}\right)}\\
&=\overline{\left(\frac{\partial F}{\partial\alpha}(\bar z_0)\right)}=\frac{\partial}{\partial\alpha} \overline{F(\bar z)}_{|_{z=z_0}}\\
&=\frac{\partial F}{\partial\alpha}(z_0)\,.
\end{align*}
For all $\alpha_0\in D\cap\rr$, it holds
\[\lim_{D^\leqslant\ni z\to \alpha_0}\frac{F(z)-F(\alpha_0)}{z-\alpha_0}=\frac{\partial F}{\partial\alpha}(\alpha_0)=\lim_{D^\geqslant\ni z\to \alpha_0}\frac{F(z)-F(\alpha_0)}{z-\alpha_0}\,.\]
Overall, we have proven that, for each $z_0\in D$,
\[\lim_{D\ni z\to z_0}\frac{F(z)-F(z_0)}{z-z_0}=\frac{\partial F}{\partial\alpha}(z_0)\,,\]
whence $F$ is globally a holomorphic map $(D,\ui)\to(\hh_\cc,\ui)$.
\end{proof}

We are now ready to associate to each couple of holomorphic functions $r:\OO_J^\geqslant\to\hh$ and $s:\OO_K^\geqslant\to\hh$ a holomorphic stem function.

\begin{theorem}[General extension formula]\label{thm:generalextensionformula}
Let $\OO$ be a symmetric open subset of $\hh$ and take distinct $J,K\in\s$. Consider holomorphic functions $r:\OO_J^\geqslant\to\hh$ and $s:\OO_K^\geqslant\to\hh$ and assume the equality $r(\alpha)=s(\alpha)$ to hold for all $\alpha\in\OO\cap\rr$ (if any). Let us set $D:=\phi_J^{-1}(\OO_J)$ (so that $\OO=\OO_D$). Let us define $F:D\to\hh_\cc$ by setting, for all $z=\alpha+\ui\beta\in D^\geqslant$,
\[F(\alpha+\ui\beta):=(J-K)^{-1} \left[J r(\alpha+\beta J) - K s(\alpha+\beta K)\right] + \ui (J-K)^{-1} \left[r(\alpha+\beta J) - s(\alpha+\beta K)\right]\]
as well as $F(\bar z):=\overline{F(z)}$. Then $F$ is a well-defined holomorphic stem function. It induces an element of $\sr(\OO)$ (whence a slice regular function), namely $f:=\I(F):\OO\to\hh$, such that $r=f_{|_{\OO_J^\geqslant}}$ and $s=f_{|_{\OO_K^\geqslant}}$.
\end{theorem}

\begin{proof}
The function $F$ is clearly well-defined in $D^\geqslant$. It is also holomorphic in $D^\geqslant$ because
\begin{align*}
\debar F(\alpha+\ui\beta)=\,&[(J-K)^{-1}J+\ui(J-K)^{-1}]\,\debar_J r(\alpha+\beta J)\\
&-[(J-K)^{-1}K+\ui(J-K)^{-1}]\,\debar_K s(\alpha+\beta K)\equiv0\,.
\end{align*}
Moreover, for all $\alpha\in\OO\cap\rr$, it holds
\begin{align*}
F(\alpha)=\,& (J-K)^{-1} \left[J r(\alpha) - K s(\alpha)\right] + \ui (J-K)^{-1} \left[r(\alpha) - s(\alpha)\right]\\
=\,& (J-K)^{-1}(J-K)r(\alpha)+\ui(J-K)^{-1}[r(\alpha)-r(\alpha)]\\
=\,& r(\alpha)\,.
\end{align*}
Now, $\overline{r(\alpha)}=r(\alpha)$ because the map $\overline{x+\ui y}=x-\ui y$ preserves all $x\in\hh$. Thus, $\overline{F(\alpha)}=F(\alpha)$ for all $\alpha\in D\cap\rr$. By applying Lemma~\ref{lem:schwarzreflection}, we get that the function $F:D\to\hh_\cc$ is a well-defined holomorphic stem function, as desired.

If we set $f:=\I(F)$, then for all $I\in\s$ the restriction $f_{|_{\OO_I^\geqslant}}$ is the holomorphic map $p^I:\OO_I^\geqslant\to\hh$ of Proposition~\ref{prop:holomorphicfamily}. In particular, $r=p^J=f_{|_{\OO_J^\geqslant}}$ and $s=p^K=f_{|_{\OO_K^\geqslant}}$, as desired.
\end{proof}

\subsection{Double-index slice regular family}\label{subsec:double-index}

The work done in the previous subsection allows us to associate to each slice regular function a family of holomorphic stem functions depending on two indices $J,K\in\s$. We recall that, for every $E\subseteq\rr_\cc$, the symbol $\overline{E}$ denotes the image of $E$ through conjugation $z\mapsto\bar z$.

\begin{figure}[htbp]
\centering
\includegraphics[height=4cm]{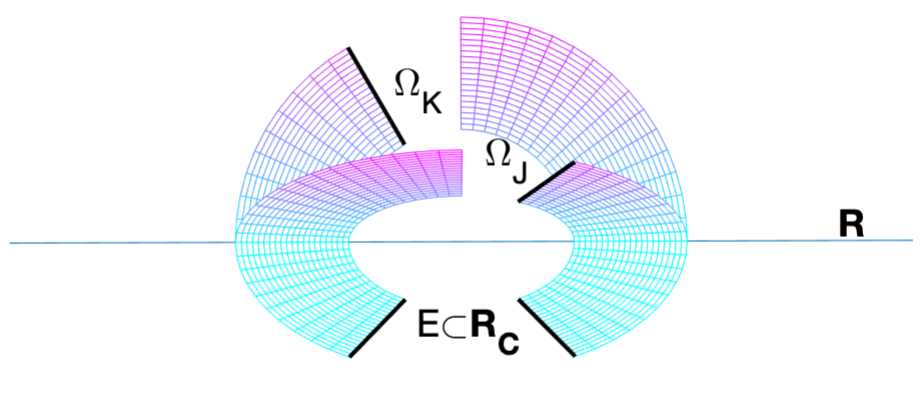}
\caption{An example of $E:=\phi_J^{-1}(\OO_J^\geqslant)\cap\phi_K^{-1}(\OO_K^\geqslant)$.}\label{fig:setE}
\end{figure}

\begin{definition}\label{def:double-index}
Let $\OO\subseteq\hh$ be a domain and let $f:\OO\to\hh$ be a slice regular function. For each choice of distinct $J,K\in\s$, we set
\[E:=\phi_J^{-1}(\OO_J^\geqslant)\cap\phi_K^{-1}(\OO_K^\geqslant),\quad D^{J,K}:=E\cup\overline{E},\quad \OO^{J,K}:=\OO_{D^{J,K}}\,.\]
If $E$ (whence $D^{J,K}$ and $\OO^{J,K}$) is not empty, we can apply Theorem~\ref{thm:generalextensionformula} to $r:=f_{|_{\phi_J(E)}}$ and to $s:=f_{|_{\phi_K(E)}}$ to get a holomorphic stem function $F^{J,K}:D^{J,K}\to\hh_\cc$, defined by setting, for all $z=\alpha+\ui\beta\in E$,
\[F^{J,K}(\alpha+\ui\beta):=(J-K)^{-1} \left[J f(\alpha+\beta J) - K f(\alpha+\beta K)\right] + \ui (J-K)^{-1} \left[f(\alpha+\beta J) - f(\alpha+\beta K)\right]\]
as well as $F^{J,K}(\bar z):=\overline{F^{J,K}(z)}$. The family $\{F^{J,K}\}_{J,K\in\s, J\neq K, D^{J,K}\neq\emptyset}$ is called the \emph{double-index holomorphic stem family associated to $f$}. The \emph{double-index slice regular family associated to $f$} is the family $\{f^{J,K}\}_{J,K\in\s, J\neq K,\OO^{J,K}\neq\emptyset}$ of the elements $f^{J,K}=\I(F^{J,K})$ of $\sr(\OO^{J,K})$ induced by the $F^{J,K}$'s.
\end{definition}

We will soon study the double-index families associated to $f$. Before doing so, we need the next definition and result.

\begin{definition}\label{def:spine}
Let $\OO$ be an open subset of $\hh$. The \emph{spine} of $\OO$ is the open subset $\spine(\OO)$ of $\hh$ defined as the union of all quaternionic open balls $B(\alpha_0,R)$ with $\alpha_0\in\OO\cap\rr$ (if any) and $R>0$ such that $B(\alpha_0,R)\subseteq\OO$. We define $\spine_{\rr_\cc}(\OO)$ as the subset $D$ of $\rr_\cc$, preserved by conjugation $z\mapsto\bar z$, whose circularization $\OO_D$ equals $\spine(\OO)$.
\end{definition}

\begin{remark}
If we adopt the notations of Definition~\ref{def:double-index}, then for all distinct $J,K\in\s$: if $D^{J,K}$ is not empty, then it includes the set $\spine_{\rr_\cc}(\OO)$. It follows that $\spine(\OO)$ is contained in the circularization $\OO^{J,K}$ of $D^{J,K}$, which is the domain of $f^{J,K}$.
\end{remark}

We are now ready for the announced study.

\begin{proposition}\label{prop:doubleindexstemfamily}
Let $\OO\subseteq\hh$ be a domain and let $f:\OO\to\hh$ be a slice regular function. Let $\{F^{J,K}:D^{J,K}\to\hh_\cc\}_{J,K\in\s, J\neq K,D^{J,K}\neq\emptyset}$ and $\{f^{J,K}:\OO^{J,K}\to\hh\}_{J,K\in\s, J\neq K, \OO^{J,K}\neq\emptyset}$ denote the double-index families associated to $f$. Fix $J,K,J',K'\in\s$ with $J\neq K,J'\neq K'$ and $\OO^{J,K}\neq\emptyset\neq\OO^{J',K'}$.
\begin{enumerate}
\item The function $f^{J,K}$ coincides with $f$ in $(\OO^{J,K})_J^\geqslant$ and in $(\OO^{J,K})_K^\geqslant$.
\item The function $f^{J,K}$ coincides with $f$ in $\spine(\OO)$.  Moreover, $f$ and $f^{J,K}$ coincide with $F^{J,K}$ in $\OO\cap\rr=(\OO^{J,K})\cap\rr=D^{J,K}\cap\rr$.
\item The function $F^{J',K'}$ coincides with $F^{J,K}$ in $\spine_{\rr_\cc}(\OO)$.
\item For each connected open subset $C$ of $D^{J',K'}\cap D^{J,K}$ that includes a real point (if any), the function $F^{J',K'}$ coincides with $F^{J,K}$ in $C$ and the function $f^{J',K'}$ coincides with $f^{J,K}$ in $\OO_C$.
\end{enumerate}
\end{proposition}

\begin{proof}
\begin{enumerate}
\item By Theorem~\ref{thm:generalextensionformula}, the induced function $f^{J,K}$ coincides with $f$ in $(\OO^{J,K})_J^\geqslant$ and in $(\OO^{J,K})_K^\geqslant$.
\item Take any $\alpha_0\in D^{J,K}\cap\rr$. The equality $F^{J,K}(\alpha_0)=f(\alpha_0)$ follows by direct inspection in the definition of $F^{J,K}$. The equality $f(\alpha_0)=f^{J,K}(\alpha_0)$ is a special case of property {\it 1.} Now let $R>0$ be such that $B(\alpha_0,R)\subseteq\OO$ (whence $B(\alpha_0,R)\subseteq\OO^{J,K}$). Up to precomposing the function $f$ with the real translation $x\mapsto x+\alpha_0$, we may suppose $\alpha_0=0$. Lemma~\ref{lem:representationball} implies that
\begin{align*}
&f(\alpha+\beta I)=\\
&=(J-K)^{-1} \left[J f(\alpha+\beta J) - K f(\alpha+\beta K)\right] + I (J-K)^{-1} \left[f(\alpha+\beta J) - f(\alpha+\beta K)\right]\\
&=f^{J,K}(\alpha+\beta I)
\end{align*}
for all $\alpha+\beta I\in B(0,R)$. The thesis now follows from the definition of $\spine(\OO)$.
\item By property {\it 2.}, $f^{J,K}=\I(F^{J,K})$ and $f^{J',K'}=\I(F^{J',K'})$ both coincide with $f$ in $\spine(\OO)$. It follows that $F^{J,K}$ coincides with $F^{J',K'}$ in $\spine_{\rr_\cc}(\OO)$.
\item Let $C$ be any connected open subset of $D^{J,K}\cap D^{J',K'}$ with $C\cap\rr\neq\emptyset$ (whence $C\cap\spine_{\rr_\cc}(\OO)$ is a nonempty open subset of $\rr_\cc$). Now, $F^{J,K}_{|_C}$ and $F^{J',K'}_{|_C}$ are holomorphic maps $(C,\ui)\to(\hh_\cc,\ui)$ coinciding in $C\cap\spine_{\rr_\cc}(\OO)$ by property {\it 3.} Since $C$ is connected, the holomorphic maps $F^{J,K}_{|_C}$ and $F^{J',K'}_{|_C}$ must coincide throughout $C$. Since $F^{J,K}$ and $F^{J',K'}$ are stem functions, it follows that they coincide in $C\cup\overline{C}\subseteq D^{J,K}\cap D^{J',K'}$. As a consequence, $f^{J,K}=\I(F^{J,K})$ and $f^{J',K'}=\I(F^{J',K'})$ coincide in the circularization $\OO_C$ of $C$.
\qedhere
\end{enumerate}
\end{proof}

As a byproduct, we obtain a new proof of two well-known results: the General Representation Formula for symmetric slice domains, proven in~\cite{advancesrevised}, and the fact, proven in~\cite{perotti}, that the class of slice regular functions on a symmetric slice domain $\OO_D$ coincides with $\sr(\OO_D)$ (whence is included in the class of real analytic functions). We merge all of these results into a single statement.

\begin{theorem}[Global representation formula for symmetric slice domains]\label{thm:generalrepresentationformula}
Let $\OO_D$ be a symmetric slice domain in $\hh$ and let $f:\OO_D\to\hh$ be a slice regular function. Then $f\in\sr(\OO_D)$ and, in particular, $f$ is real analytic. Moreover, for all distinct $J,K\in\s$, it holds $f=f^{J,K}$ and
\[f(\alpha+\beta I)=(J-K)^{-1} \left[J f(\alpha+\beta J) - K f(\alpha+\beta K)\right] + I (J-K)^{-1} \left[f(\alpha+\beta J) - f(\alpha+\beta K)\right]\]
for all $I\in\s$ and all $\alpha,\beta\in\rr$ with $\beta\geq0$ such that $\alpha+\beta I\in\OO_D$.
\end{theorem}

\begin{proof}
We assume, without loss of generality, $D$ to be preserved by conjugation $z\mapsto\bar z$. Let $\{F^{J,K}\}_{J,K\in\s, J\neq K}$ denote the double-index holomorphic stem family associated to $f:\OO_D\to\hh$. Since $\OO_D$ is a symmetric slice domain, all domains $D^{J,K}$ (whence all intersections $D^{J,K}\cap D^{J',K'}$) coincide with $D$. By hypothesis, the latter set is connected and intersects the real axis. By property {\it 4.} in Proposition~\ref{prop:doubleindexstemfamily}, all $f^{J,K}:\OO_D\to\hh$ coincide and form a single element $g$ of $\sr(\OO_D)$. Moreover, property {\it 1.} in the same theorem guarantees that $g$ coincides with $f$ in $(\OO_D)_J^\geqslant$ for all $J\in\s$, whence $g=f$. As a consequence, the formula
\[f(\alpha+\beta I)=(J-K)^{-1} \left[J f(\alpha+\beta J) - K f(\alpha+\beta K)\right] + I (J-K)^{-1} \left[f(\alpha+\beta J) - f(\alpha+\beta K)\right]\]
holds for all $\alpha+\beta I\in\OO_D$ with $\alpha,\beta\in\rr,\beta\geq0$ and $I\in\s$.
\end{proof}


\section{Speared domains and local representation formula}\label{sec:localrepresentationformula}

\subsection{Speared domains and their cores}

Let us introduce some new terminology.

\begin{definition}\label{def:speared}
A \emph{speared} open subset of $\rr_\cc$ is an open $T\subseteq\rr_\cc$ whose connected components all intersect the real axis.
A \emph{speared} open subset of $\hh$ is an open $U\subseteq\hh$ such that, for all $I\in\s$, each connected component of $U^{\geqslant}_I$ includes a real point. A \emph{speared domain} is a domain $\OO$ of $\hh$ that is a speared open subset of $\hh$.
\end{definition}

We remark that, if $\OO\subseteq\hh$ is a speared domain, then $\OO\cap\rr\neq\emptyset$. In particular, the sets $D^{J,K}$ and $\OO^{J,K}$ of Definition~\ref {def:double-index} are nonempty for each choice of distinct $J,K\in\s$.

As we will see in the forthcoming Section~\ref{sec:classesofdomains}, the class of speared open sets is a rather broad extension of the class of slice domains, whose definition we recall from~\cite[Definition 1.12]{librospringer2ed}.

\begin{definition}\label{def:slicedomain}
Let $\OO$ be a domain in $\hh$. Then $\OO$ is called a \emph{slice domain} if it intersects the real axis $\rr$ and if, for all $J\in\s$, the ``slice'' $\OO_J$ is connected.
\end{definition}

However, if a speared open subset $U$ of $\hh$ happens to be symmetric, we are only slightly generalizing the notion of symmetric slice domain: indeed, like every other symmetric open subset of $\hh$, $U$ is a disjoint union of symmetric slice domains or product domains. Now, every symmetric slice domain is speared and every product domain is not. It follows that our $U$ is a disjoint union of symmetric slice domains. By applying Theorem~\ref{thm:generalrepresentationformula} in each connected component of $U$, we can make the next remark.

\begin{remark}\label{rmk:union}
Let $\OO_D$ be a symmetric speared open subset of $\hh$ and let $f:\OO_D\to\hh$ be a slice regular function. Then $f\in\sr(\OO_D)$. Moreover, for all distinct $J,K\in\s$, it holds $f=f^{J,K}$ and
\[f(\alpha+\beta I)=(J-K)^{-1} \left[J f(\alpha+\beta J) - K f(\alpha+\beta K)\right] + I (J-K)^{-1} \left[f(\alpha+\beta J) - f(\alpha+\beta K)\right]\]
for all $I\in\s$ and all $\alpha,\beta\in\rr$ with $\beta\geq0$ such that $\alpha+\beta I\in\OO_D$.
\end{remark}

If $U$ is a general open subset of $\hh$, we will still be able to apply Remark~\ref{rmk:union} to an appropriately chosen symmetric speared open subset of $U$. This choice will be possible after some preliminary steps. We begin with the next remark, which uses the fact that every open subset of $\rr_\cc$ or $\rr_\cc^\geqslant$ is locally path-connected.

\begin{remark}\label{rmk:path}
An open subset $U$ of $\hh$ is speared if, and only if, for all $I\in\s$ and for all $x\in U_I^{\geqslant}$, there exists a path $\gamma:[0,1]\to\phi_I^{-1}(U_I^{\geqslant})$ such that $\gamma(0)\in\rr$ and $\phi_I(\gamma(1))=x$. The latter property is equivalent to assuming the relative $0$-homology group $H_0(U_I^{\geqslant},U\cap\rr;\zz)$ to be trivial.

An open subset $T$ of $\rr_\cc$ is speared if, and only if, for all $z\in T$, there exists a path $\gamma:[0,1]\to T$ such that $\gamma(0)\in\rr$ and $\gamma(1)=z$. Up to restricting $\gamma$, we can assume the support of $\gamma$ to be entirely contained in $\rr_\cc^\geqslant$ or in $\rr_\cc^\leqslant$.
\end{remark}

For every open subset $U$ of $\hh$, by Definition~\ref{def:spine}, the symmetric set $\spine(U)$ is a speared open subset of $\hh$ (nonempty if $U\cap\rr\neq\emptyset$). Remark~\ref{rmk:path} implies that any union of speared open sets of $\hh$ is speared. This justifies the next definition.

\begin{definition}
Let $U$ be an open subset of $\hh$. The \emph{core} of $U$, denoted as $\core(U)$, is the largest symmetric speared open subset of $U$. We define $\core_{\rr_\cc}(U)$ to be the open subset $D\subseteq\rr_\cc$, preserved by conjugation $z\mapsto\bar z$, such that $\core(U)=\OO_D$.
\end{definition}

Clearly, $\core(\emptyset)=\emptyset$. The core of a nonempty open subset $U\subseteq\hh$ can be obtained as follows. Let us first set $D:=\bigcap_{J\in\s}\phi_J^{-1}(U_J^\geqslant)$: then the circularization $\OO_D$ is the largest symmetric open subset of $U$. This is a consequence of the following facts:
\begin{itemize}
\item each $2$-sphere (or singleton) $x+y\s$ is contained in $\OO_D$ if, and only if, it is contained in $U$;
\item the open set $U$ contains $x+y\s$ (which is compact) if, and only if, it contains a symmetric neighborhood of $x+y\s$.
\end{itemize}
Now, $\OO_D$ is a disjoint union of symmetric slice domains or product domains. The set $\core(U)$ is the union of all connected components of $\OO_D$ that are symmetric slice domains. Clearly, $\spine(U)$ and $\core(U)$ are nested symmetric speared open subsets of $U$. They are nonempty if, and only if, $U\cap\rr\neq\emptyset$ (which is true if $U$ is a nonempty speared open subset of $\hh$). In the forthcoming Example~\ref{ex:speared}, we will construct a speared domain $\underline{\OO}$ with $\emptyset\neq\spine(\underline{\OO})\subsetneq\core(\underline{\OO})\subsetneq\underline{\OO}$: see Figure~\ref{fig:spinecore}.

We can now apply Remark~\ref{rmk:union} to $\OO_D=\core(U)$ to prove the next lemma.

\begin{lemma}\label{lem:representationincore}
Let $U$ be an open subset of $\hh$ with $U\cap\rr\neq\emptyset$ and let $f:U\to\hh$ be a slice regular function. Then $f_{|_{\core(U)}}\in\sr(\core(U))$. Moreover, for all distinct $J,K\in\s$, it holds $f_{|_{\core(U)}}=f^{J,K}_{|_{\core(U)}}$ and
\[f(\alpha+\beta I)=(J-K)^{-1} \left[J f(\alpha+\beta J) - K f(\alpha+\beta K)\right] + I (J-K)^{-1} \left[f(\alpha+\beta J) - f(\alpha+\beta K)\right]\]
for all $I\in\s$ and all $\alpha,\beta\in\rr$ with $\beta\geq0$ such that $\alpha+\beta I\in\core(U)$. As a consequence, for all $J',K'\in\s$ with $J'\neq K'$, the stem function $F^{J',K'}$ coincides with $F^{J,K}$ in $\core_{\rr_\cc}(U)$.
\end{lemma}

Besides its independent interest, we will use Lemma~\ref{lem:representationincore} to prove the forthcoming Theorem~\ref{thm:singleindex}.

\subsection{Slice regular families on speared domains}

The aim of this subsection is proving that the double-index families associated to a slice regular function $f:\OO\to\hh$ reduce to single-index families when $\OO$ is a speared domain. The proof is based on Lemma~\ref{lem:representationincore} and on a second lemma, which is a variant of~\cite[Lemma 3.1]{localrepresentation}. Let us use the notation $\scap(J,\varepsilon):=\{K\in\s : |K-J|<\varepsilon\}$ for the spherical cap of radius $\varepsilon$ centered at $J$ in $\s$ (or the whole sphere $\s$ if $\varepsilon>2$). Let $\scap^*(J,\varepsilon):=\scap(J,\varepsilon)\setminus\{J\}$ denote the same spherical cap, punctured at $J$.

\begin{lemma}\label{lem:neighborhoodcompact}
Let $Y$ be a nonempty open subset of $\hh$. Let $C$ be a nonempty, compact and path-connected subset of $\rr_\cc\simeq\cc$. For each positive real number $\varepsilon$, the set
\[C_\varepsilon:=\{z\in\rr_\cc \ |\  \mathrm{dist}(z,C)<\varepsilon\}\]
is a path-connected open neighborhood of $C$. If $J\in\s$ is such that $\phi_J(C)\subset Y$, then there exists $\varepsilon>0$ such that $\phi_K(C_\varepsilon)\subset Y$ for all $K\in\scap(J,\varepsilon)$.
\end{lemma}

\begin{proof}
If $Y=\hh$, it suffices to set $\varepsilon:=3$. Let us assume henceforth $\hh\setminus Y\neq\emptyset$.

Since $\phi_J(C)$ is a compact subset of the open set $Y$, the distance between $\hh\setminus Y$ and $\phi_J(C)$ is strictly positive. Set $\ell:=\max_{z\in C}|\im(z)|$ and
\[\varepsilon:=\frac{\mathrm{dist}(\hh\setminus Y,\phi_J(C))}{\ell+1}>0\,.\]

We observe that $C_\varepsilon$ is the union of open disks of radii $\varepsilon$ centered at arbitrary points $z_0\in C$, whence $C_\varepsilon$ is an open neighborhood of $C$. Since $C$ is path-connected and since any point $z\in C_\varepsilon$ can be joined to a point $z_0\in C$ through a line segment included in $C_\varepsilon$, it follows that $C_\varepsilon$ is path-connected.

We finally show that $\phi_K(C_\varepsilon)\subset Y$ for all $K\in\scap(J,\varepsilon)$. If $z\in C_\varepsilon$, then there exists $z_0\in C$ such that $|z-z_0|<\varepsilon$. It follows that
\begin{align*}
\mathrm{dist}(\phi_K(z),\phi_J(C))&\leq|\phi_K(z)-\phi_J(z_0)|\\
&\leq|\phi_K(z)-\phi_K(z_0)|+|\phi_K(z_0)-\phi_J(z_0)|\\
&=|z-z_0|+|K-J||\im(z_0)|<\varepsilon+\varepsilon\ell\\
&=\mathrm{dist}(\hh\setminus Y,\phi_J(C)).
\end{align*}
As a consequence, $\phi_K(z)\in Y$, as desired.
\end{proof}

We are now ready to prove that the double index reduces to a single index if we are dealing with a speared domain. We recall that, for every $E\subseteq\rr_\cc$, the symbol $\overline{E}$ denotes the image of $E$ through conjugation $z\mapsto\bar z$.

\begin{theorem}\label{thm:singleindex}
Let $\OO\subseteq\hh$ be a speared domain, let $f:\OO\to\hh$ be a slice regular function and let $\{F^{J,K}:D^{J,K}\to\hh_\cc\}_{J,K\in\s, J\neq K}$ denote the double-index holomorphic stem family associated to $f$. Let us fix $J\in\s$ and consider the speared open set
\[D^J:=\phi_J^{-1}(\OO_J^\geqslant)\cup\overline{\phi_J^{-1}(\OO_J^\geqslant)}\,.\]
Then there exists a holomorphic stem function $F^J:D^J\to\hh_\cc$ with the following property: for each $z\in D^J$, there exist  a real number $\varepsilon_z>0$ and a speared open neighborhood $U_z$ of $z$ in $D^J$ such that
\begin{enumerate}
\item it holds $\core_{\rr_\cc}(\OO)\subseteq U_z$ and $U_z\setminus\core_{\rr_\cc}(\OO)$ is contained in either $\rr_\cc^>$ or $\rr_\cc^<$;
\item for any choice of distinct $J',K'\in\scap(J,\varepsilon_z)$, it holds $U_z\subseteq D^{J',K'}$ and $F^{J',K'}_{|_{U_z}}=F^J_{|_{U_z}}$.
\end{enumerate}
In case $z\in\core_{\rr_\cc}(\OO)$, we can assume $U_z$ to be $\core_{\rr_\cc}(\OO)$ and $\varepsilon_z=3$ (whence $\scap(J,\varepsilon_z)=\s$).
\end{theorem}

\begin{proof}
Let us begin by choosing, for each $z=\alpha+\beta J\in (D^J)^\geqslant$, appropriate $\varepsilon_z>0$ and $U_z$.
\begin{itemize}
\item If $z\in\core_{\rr_\cc}(\OO)$, we can define $U_z:=\core_{\rr_\cc}(\OO)\subseteq D^J$ and $\varepsilon_z:=3$: Lemma~\ref{lem:representationincore} guarantees that $F^{J,K}_{|_{U_z}}=F^{J',K'}_{|_{U_z}}$ for all $J',K,K'\in\scap(J,\varepsilon_z)=\s$ with $J\neq K$ and $J'\neq K'$.
\item Assume $z\not\in\core_{\rr_\cc}(\OO)$ (whence $\beta>0$). Remark~\ref{rmk:path} implies that there exists a path $\gamma:[0,1]\to\phi_J^{-1}(\OO_J^{\geqslant})$ with $\gamma(0)\in\rr$ and $\gamma(1)=z$. Let us set
\[C:=\gamma([t_0,1]),\quad t_0:=\sup\{t\in[0,1]\,|\,\gamma(t)\in\core_{\rr_\cc}(\OO)\}\in(0,1]\,.\]
We observe, for future reference, that $\gamma(t_0)$ belongs to the boundary of $\core_{\rr_\cc}(\OO)$ and that
\[C\subset \phi_J^{-1}(\OO_J^{\geqslant})\setminus\core_{\rr_\cc}(\OO)\subseteq\rr_\cc^>\,.\]
Now, $C$ is compact and path-connected; moreover, $\phi_J(C)\subset\OO_J^>\subset\OO\setminus\rr$. If we apply Lemma~\ref{lem:neighborhoodcompact} to $C$ and to $Y:=\OO\setminus\rr$, we conclude that there exists a real number $\varepsilon>0$ such that the path-connected open neighborhood $C_\varepsilon$ of $C$ fulfills the condition $\phi_K(C_\varepsilon)\subset\OO\setminus\rr$ (whence $\phi_K(C_\varepsilon)\subset\OO_K^>$) for all $K\in\scap(J,\varepsilon)$. We define $U_z:=\core_{\rr_\cc}(\OO)\cup C_\varepsilon$ and observe that $z\in U_z\subseteq D^{J',K'}$ for all distinct $J',K'\in\scap(J,\varepsilon)$. We also remark that, since $C$ intersects the boundary of $\core_{\rr_\cc}(\OO)$ at $\gamma(t_0)$, its open neighborhood $C_\varepsilon$ intersects $\core_{\rr_\cc}(\OO)$ at some point $p_0$. We can prove that the open set $U_z$ is speared by applying Remark~\ref{rmk:path}. Indeed, we can exhibit for every $p\in U_z$ a path $\gamma_p:[0,1]\to U_z$ with $\gamma_p(0)\in\rr$ and $\gamma_p(1)=p$:
\begin{itemize}
\item if $p\in \core_{\rr_\cc}(\OO)$, the existence of a suitable path $\gamma_p:[0,1]\to \core_{\rr_\cc}(\OO)\subset U_z$ is guaranteed because $\core_{\rr_\cc}(\OO)$ is speared;
\item if $p\in C_\varepsilon$, we can construct $\gamma_p:[0,1]\to U_z$ by joining a path $\gamma_{p_0}$ from a real point to $p_0$ in $\core_{\rr_\cc}(\OO)$ with a path from $p_0$ to $p$ in $C_\varepsilon$ (which exists because $C_\varepsilon$ is path-connected).
\end{itemize}
By applying property {\it 4.} of Proposition~\ref{prop:doubleindexstemfamily} to each connected component of $U_z$, we obtain that $F^{J,K}_{|_{U_z}}=F^{J',K'}_{|_{U_z}}$ for all $J',K,K'\in\scap(J,\varepsilon)$ with $J\neq K$ and $J'\neq K'$. We conclude by setting $\varepsilon_z:=\varepsilon$.
\end{itemize}
We are now ready to define $F^J(z)$ as the constant value of the map $\scap^*(J,\varepsilon_z)\to\hh_\cc,\ K\mapsto F^{J,K}(z)$. In particular,
\[F^J(z)=\lim_{\scap^*(J,\varepsilon_z)\ni K\to J}F^{J,K}(z)\,.\]
We notice that the value $F^J(z)$ does not depend on the choice of $\varepsilon_z$ nor on the choice of $U_z$. Furthermore, $F^J$ is holomorphic near $z$ because, for $K\in\scap^*(J,\varepsilon_z)$, it coincides in $U_z$ with the holomorphic stem function $F^{J,K}$.

We have therefore defined a holomorphic $F^J:(D^J)^\geqslant\to\hh_\cc$ such that, for all $\alpha\in D^J\cap\rr$ it holds $F^J(\alpha)=f(\alpha)\in\hh$ (whence $\overline{F^J(\alpha)}=F^J(\alpha)$). Lemma~\ref{lem:schwarzreflection} implies that $F^J$ extends to a holomorphic stem function $D^J\to\hh_\cc$ by setting $F^J(\bar z):=\overline{F^J(z)}$ for each $z\in(D^J)^\geqslant$.

We conclude the proof by setting $U_{\bar z}:=\overline{U_z}$ and $\varepsilon_{\bar z}:=\varepsilon_z$ for all $z\in(D^J)^\geqslant$. We can do so because $F^{J',K'}_{|_{\overline{{U_z}}}}=F^{J}_{|_{\overline{{U_z}}}}$ for all distinct $J',K'\in\scap(J,\varepsilon_z)$, as a consequence of the fact that $F^{J,'K'}$ and $F^J$ are stem functions.
\end{proof}

\begin{definition}
In the situation described in Theorem~\ref{thm:singleindex}, $\{F^J:D^J\to\hh_\cc\}_{J\in\s}$ is called the \emph{holomorphic stem family associated to $f$}. For each $J\in\s$, let us consider the symmetric speared open set $\OO^J:=\OO_{D^J}$: the \emph{slice regular family associated to $f$} is the family $\{f^J:\OO^J\to\hh\}_{J\in\s}$ of the elements $f^J=\I(F^J)$ of $\sr(\OO^J)$ induced by the $F^J$'s.
\end{definition}

For future use, we prove the next lemma.

\begin{lemma}\label{lem:connectedcomponent}
Let $\OO\subseteq\hh$ be a speared domain, let $f:\OO\to\hh$ be a slice regular function and let $\{F^J:D^J\to\hh_\cc\}_{J\in\s}$ be the holomorphic stem family associated to $f$. Fix $J_0,J_1\in\s$. The holomorphic stem functions $F^{J_0},F^{J_1}$ coincide in $\core_{\rr_\cc}(\OO)\supseteq\OO\cap\rr$. Now, let $\alpha,\beta\in\rr$ with $\beta\geq0$ be such that $x_0=\alpha+\beta J_0,x_1=\alpha+\beta J_1$ both belong to $\OO$. If $x_0,x_1$ belong to the same connected component of $(\alpha+\beta\s)\cap\OO$, then there exists an open neighborhood $U$ of $\alpha+\beta\ui$ in $D^{J_0,J_1}$ such that $F^{J_0},F^{J_1}$ coincide in $U$.
\end{lemma}

\begin{proof}
The first statement concerning $\core_{\rr_\cc}(\OO)$ immediately follows from Theorem~\ref{thm:singleindex} and Lemma~\ref{lem:representationincore}. Now let us prove the second statement about $x_0$ and $x_1$ belonging to the same connected component $\mathcal{C}$ of $(\alpha+\beta\s)\cap\OO$. It suffices to pick any $\mc{J}:[0,1]\to\s$ such that $\alpha+\beta\mc{J}$ is a path within $\mathcal{C}$ and to show that there exists an open neighborhood $U$ of $\alpha+\beta\ui$ in $D^{\mc{J}(0),\mc{J}(1)}$ such that $F^{\mc{J}(0)},F^{\mc{J}(1)}$ coincide in $U$. 

For every $J\in\s$ such that  $\alpha+\beta J\in\OO$, Theorem~\ref{thm:singleindex} guarantees that there exist an an open neighborhood $U^J$ of the point $\alpha+\beta\ui$ in $D^{J}$ and an $\varepsilon^J>0$ with the following property: for any choice of distinct $J',K'\in\scap(J,\varepsilon^J)$, it holds $U^J\subseteq D^{J',K'}$ and $F^{J',K'}_{|_{U^J}}=F^J_{|_{U^J}}$. It follows that $F^{J'}_{|_{U^J}}=F^{K'}_{|_{U^J}}$ for all $J',K'\in\scap(J,\varepsilon^J)$. Consider the open cover $\{\mc{J}^{-1}(\scap(J,\varepsilon^J))\}_{J\in\s}$ of the compact interval $[0,1]$. Take $N\in\nn$ large enough for $1/N$ to be a Lebesgue number for this cover (see~\cite[Lemma 7.2]{libromunkres} for details). It follows that, for each $s\in\{0,\ldots,N-1\}$, there exists $J_s\in\s$ such that $\mc{J}([s/N,(s+1)/N])\subset\scap(J^s,\varepsilon^{J^s})$. As a consequence, $U^{J^s}\subseteq D^{\mc{J}(s/N),\mc{J}((s+1)/N)}$ and $F^{\mc{J}(s/N)}_{|_{U^{J^s}}}=F^{\mc{J}((s+1)/N)}_{|_{U^{J^s}}}$. The intersection $U:=\bigcap_{s=0}^NU^{J^s}$ is an open neighborhood of $\alpha+\beta\ui$ in $D^{\mc{J}(s/N)}$ for all $s\in\{0,\ldots,N\}$ (whence in $D^{\mc{J}(0)}\cap D^{\mc{J}(1)}=D^{\mc{J}(0),\mc{J}(1)}$) and we have
\[F^{\mc{J}(0)}_{|_U}=F^{\mc{J}(1/N)}_{|_U}=\ldots=F^{\mc{J}((N-1)/N)}_{|_U}=F^{\mc{J}(1)}_{|_U}\,,\]
whence the thesis follows.
\end{proof}

\subsection{Semi-global extension and representation for speared domains}

Theorem~\ref{thm:singleindex} allows us to extend and strengthen~\cite[Theorem 3.2]{localrepresentation} to our first main theorem.

\begin{theorem}[Semi-global extension for speared domains]\label{thm:localextension}
Let $\OO\subseteq\hh$ be a speared domain, let $f:\OO\to\hh$ be a slice regular function and let $\{f^J\}_{J\in\s}$ be the slice regular family associated to $f$. Fix $J_0\in\s$. There exists a speared open set $\Lambda$ with
\[\core(\OO)\cup\OO_{J_0}^\geqslant\subset\Lambda\subseteq\OO\cap\OO^{J_0}\]
such that $f$ coincides with $f^{J_0}$ in $\Lambda$. In particular, for each $x_0\in\OO_{J_0}^\geqslant$, $f$ coincides with $f^{J_0}$ in a speared domain including $x_0$: namely, the connected component of $\Lambda$ including $x_0$.
\end{theorem}

\begin{proof}
Let $\{F^J:D^J\to\hh_\cc\}_{J\in\s}$ and $\{F^{J,K}:D^{J,K}\to\hh_\cc\}_{J,K\in\s, J\neq K}$ denote, respectively, the holomorphic stem family and the double-index holomorphic stem family associated to $f$. In particular, $f^{J_0}=\I(F^{J_0})$.

Let us choose $x_0=\alpha_0+\beta_0 J_0\in\OO_{J_0}^\geqslant$ (with $\alpha_0,\beta_0\in\rr$). If we set $z_0:=\alpha_0+\ui\beta_0$ (so that $x_0=\phi_{J_0}(z_0)$), then Theorem~\ref{thm:singleindex} guarantees that there exist $\varepsilon_{z_0}>0$ and a speared open neighborhood $U_{z_0}$ of $z_0$ in $D^{J_0}$ such that:
\begin{enumerate}
\item it holds $\core_{\rr_\cc}(\OO)\subseteq U_{z_0}$ and $U_{z_0}':=U_{z_0}\setminus\core_{\rr_\cc}(\OO)$ is contained in either $\rr_\cc^>$ or $\rr_\cc^<$;
\item for all $K\in\scap^*({J_0},\varepsilon_{z_0})$, it holds $U_{z_0}\subseteq D^{{J_0},K}$ and $F^{{J_0},K}_{|_{U_{z_0}}}=F^{J_0}_{|_{U_{z_0}}}$.
\end{enumerate}
We can produce a speared open neighborhood $V_{x_0}$ of $x_0$ in $\OO\cap\OO_{U_{z_0}}\subseteq\OO\cap\OO_{D^{J_0}}$, as follows. We set $V_{x_0}:=\core(\OO)\cup V_{x_0}'$, where
\[V_{x_0}':=\bigcup_{K\in\scap({J_0},\varepsilon_{z_0})}\phi_K(U_{z_0}')=\{\alpha+\beta K : \alpha,\beta\in\rr,\alpha+\ui\beta\in U_{z_0}', K\in\scap({J_0},\varepsilon_{z_0})\}\,.\]
Now, $f^{J_0}=\I(F^{J_0})$ coincides in $V_{x_0}$ with $f^{{J_0},K}=\I(F^{{J_0},K})$ for all $K\in\scap^*({J_0},\varepsilon_{z_0})$. Let us apply property {\it 1.} in Proposition~\ref{prop:doubleindexstemfamily}: namely, for all $K\in\scap^*({J_0},\varepsilon_{z_0})$, the function $f^{{J_0},K}$ coincides with $f$ in $(\OO_{U_{z_0}})_{J_0}^\geqslant=(V_{x_0})_{J_0}^\geqslant$ and in $(\OO_{U_{z_0}})_{K}^\geqslant=(V_{x_0})_{K}^\geqslant$. It follows that $f^{J_0}$ coincides with $f$ in $V_{x_0}'$. Lemma~\ref{lem:representationincore} allows us to conclude that $f^{J_0}$ coincides with $f$ in $V_{x_0}=\core(\OO)\cup V_{x_0}'$.

The thesis now follows setting $\Lambda:=\bigcup_{x_0\in\OO_{J_0}^\geqslant}V_{x_0}$.
\end{proof}

Let us add the next definition, which extends~\cite[Definition 2.14]{geometricfunctiontheory}. We have recalled the notion of slice domain in Definition~\ref{def:slicedomain}.

\begin{definition}
Let $\OO$ be a domain in $\hh$ and let $f:\OO\to\hh$ be a slice regular function. Whenever $N$ is a symmetric speared open set, $\widetilde f:N\to\hh$ is a slice regular function and $\Lambda$ is a speared open set contained in $\OO\cap N$ such that $f_{|_\Lambda}=\widetilde f_{|_\Lambda}$, then $(\widetilde f, N, \Lambda)$ is called a \emph{speared extension triplet} for $f$. A speared extension triplet $(\widetilde f, N, \Lambda)$ for $f$ where $N$ is a symmetric slice domain and $\Lambda$ is a slice domain is called a \emph{slice extension triplet} for $f$.
\end{definition}

Using the last definition, Theorem~\ref{thm:localextension} can be reformulated as follows.

\setcounter{theorem}{10}
\begin{theorem}[Semi-global extension for speared domains, rephrased]
If $f$ is a slice regular function on a speared domain $\OO$ then, for all $J_0\in\s$, the couple $f^{J_0},\OO^{J_0}$ can be completed to a speared extension triplet $(f^{J_0},\OO^{J_0}, \Lambda)$ with $\Lambda\supseteq\core(\OO)\cup\OO_{J_0}^\geqslant$. In particular, if we define $\Lambda_0$ to be the connected component of $\Lambda$ including $x_0$, then $\Lambda_0$ is a speared domain and $(f^{J_0},\OO^{J_0}, \Lambda_0)$ is still a speared extension triplet for $f$.
\end{theorem}

\setcounter{theorem}{12}

Theorem~\ref{thm:localextension} implies that every speared domain has the Local Extension Property, as defined in~\cite[Definition 11.14]{librospringer2ed}. We prove this fact as the next Corollary.

\begin{corollary}[Local extension for speared domains]\label{cor:localextension}
If $f$ is a slice regular function on a speared domain $\OO$ and if $x_0\in\OO$, then there exists a slice extension triplet $(\widetilde f, N_{x_0}, \Lambda_{x_0})$ for $f$ with $\Lambda_{x_0}\ni x_0$. As a consequence, $f$ is real analytic.
\end{corollary}

\begin{proof}
We apply Theorem~\ref{thm:localextension} to find a speared extension triplet $(f^{J_0},\OO^{J_0}, \Lambda)$ for $f$ with $\Lambda\supseteq\core(\OO)\cup\OO_{J_0}^\geqslant$. We define $N_{x_0}$ to be the connected component of $\OO^{J_0}$ that includes $x_0$: since $\OO^{J_0}$ is a symmetric speared open set, it follows immediately that $N_{x_0}$ is a symmetric slice domain. The connected component $\Lambda_0$ of $\Lambda$ that includes $x_0$ is a speared domain contained in $N_{x_0}$. Moreover, $\Lambda_0$ is contained in $\OO$ because $\Lambda$ was. By the forthcoming Proposition~\ref{prop:localslicedomain}, there exists a slice domain $\Lambda_{x_0}\ni x_0$ contained in $\Lambda_0\subseteq\OO\cap N_{x_0}$. If we define $\widetilde f$ as the restriction of $f^{J_0}$ to $N_{x_0}$, then $(\widetilde f, N_{x_0}, \Lambda_{x_0})$ is a slice extension triplet for $f$.

Our final statement concerning real analyticity now follows from Theorem~\ref{thm:generalrepresentationformula}.
\end{proof}

The original result~\cite[Theorem 3.2]{localrepresentation} was exactly Corollary~\ref{cor:localextension}, under the additional hypothesis that $\OO$ be a slice domain. Based on the same result, the article~\cite{geometricfunctiontheory} developed a rather complete theory of regular functions on general slice domains. Later on,~\cite[\S11]{librospringer2ed} defined the class of domains \emph{having the Local Extension Property} and proved several properties of slice regular functions on such domains: for instance, the existence of a $*$-algebra structure~\cite[Corollary 11.16]{librospringer2ed}, the Strong Identity Principle~\cite[Corollary 11.19]{librospringer2ed}, the existence of the regular reciprocal $f^{-*}$ of a slice regular function $f$~\cite[Proposition 11.28]{librospringer2ed}, the Maximum Modulus Principle~\cite[Theorem 11.53]{librospringer2ed}, the Minimum Modulus Principle~\cite[Theorem 11.54]{librospringer2ed}, the Open Mapping Theorem~\cite[Theorem 11.57]{librospringer2ed}, the Local Cauchy Formula~\cite[Proposition 11.58]{librospringer2ed} and its volumetric analog~\cite[Proposition 11.59]{librospringer2ed}, as well as the existence of local spherical series expansions~\cite[Theorem 11.63]{librospringer2ed}. Our Corollary~\ref{cor:localextension} implies that the properties just listed are true on general speared domains.

By applying Remark~\ref{rmk:union} to $f^{J_0}$, we obtain another corollary of Theorem~\ref{thm:localextension}. The corollary strengthens~\cite[Theorem 3.4]{localrepresentation} and extends it to speared domains.

\begin{corollary}[Semi-global representation formula for speared domains]\label{cor:localrepresentationformula}
Let $\OO\subseteq\hh$ be a speared domain, let $f:\OO\to\hh$ be a slice regular function and let $J_0\in\s$. There exists a speared open subset $\Lambda\subseteq\OO$, including both $\core(\OO)$ and $\OO_{J_0}^\geqslant$, such that the formula
\[f(\alpha+\beta I)=(J-K)^{-1} \left[J f(\alpha+\beta J) - K f(\alpha+\beta K)\right] + I (J-K)^{-1} \left[f(\alpha+\beta J) - f(\alpha+\beta K)\right]\]
holds for all $I,J,K\in\s$ and all $\alpha,\beta\in\rr$ with $J\neq K, \beta\geq0$ and $\alpha+\beta I,\alpha+\beta J,\alpha+\beta K\in \Lambda$.
\end{corollary}

The last corollary allows us to extend~\cite[Definition 6]{perotti} and~\cite[Definition 3.1]{geometricfunctiontheory} to all speared domains.

\begin{definition}\label{def:sphericalvalueanderivative}
In the situation described in Corollary~\ref{cor:localrepresentationformula}, for all $\alpha,\beta\in\rr$ with $\beta\geq0$ such that $x_0:=\alpha+\beta J_0\in\OO$, we define the \emph{spherical value} of $f$ at $x_0$ as the quaternion
\[f^\circ_s(x_0):=(J-K)^{-1} \left[J f(\alpha+\beta J) - K f(\alpha+\beta K)\right]\,,\]
which does not depend on the specific choice of $J\neq K$ such that $\alpha+\beta J,\alpha+\beta K\in \Lambda$. If, moreover, $\beta>0$, we define the \emph{spherical derivative} of $f$ at $x_0$ as the quaternion
\[f'_s(x_0):=\beta^{-1}(J-K)^{-1} \left[f(\alpha+\beta J) - f(\alpha+\beta K)\right]\,,\]
which does not depend on the specific choice of $J\neq K$ such that $\alpha+\beta J,\alpha+\beta K\in \Lambda$.
\end{definition}

\begin{remark}\label{rmk:sphericalvalueandderivative}
Let $\OO\subseteq\hh$ be a speared domain, let $f:\OO\to\hh$ be a slice regular function and let $\{F^J:D^J\to\hh_\cc\}_{J\in\s}$ denote the holomorphic stem family associated to $f$. Fix $J_0\in\s$ and decompose the $\hh_\cc$-valued function $F^{J_0}$ as $F^{J_0}_1+\ui F^{J_0}_2$, where $F^{J_0}_1,F^{J_0}_2$ are $\hh$-valued. Then, for all $\alpha,\beta\in\rr$ with $\beta\geq0$ such that $\alpha+\beta J_0\in\OO$, we have $f^\circ_s(\alpha+\beta J_0)=F^{J_0}_1(\alpha+\beta\ui)$ and, if $\beta>0$, $f'_s(\alpha+\beta J_0)=\beta^{-1}F^{J_0}_2(\alpha+\beta\ui)$. As a consequence, the functions $f^\circ_s:\OO\to\hh$ and $f'_s:\OO\setminus\rr\to\hh$ are real analytic.
\end{remark}

Corollary~\ref{cor:localrepresentationformula} implies the following properties. We recall that the \emph{slice (or complex) derivative} of any slice regular function $f:\OO\to\hh$ is the slice regular function $f'_c:\OO\to\hh$ defined as
\[f'_c(\alpha+\beta I):=\frac12\left(\frac{\partial}{\partial\alpha}+I\frac{\partial}{\partial\beta}\right)f_I(\alpha+\beta I)\]
at any $\alpha+\beta I\in\OO$ (see~\cite[Definition 1.8]{librospringer2ed}). If $f$ happens to be $C^1$, let us denote the real differential of $f$ at a point $x_0\in\OO$ as $df_{x_0}$.

\begin{proposition}\label{prop:sphericalvalueandderivative}
Let $\OO\subseteq\hh$ be a speared domain and let $f:\OO\to\hh$ be a slice regular function.
\begin{enumerate}
\item For all $x\in\OO\setminus\rr$, the equality $f(x)=f^\circ_s(x)+\im(x)f'_s(x)$ holds, while $f(x)=f^\circ_s(x)$ at all $x_0\in\OO\cap\rr$.
\item If we fix $\alpha,\beta\in\rr$ with $\beta>0$ and set $S:=\alpha+\beta\s$, then the maps 
\begin{align*}
&S\cap\OO\to\hh\quad J\mapsto f^\circ_s(\alpha+\beta J)\,,\\
&S\cap\OO\to\hh\quad J\mapsto f'_s(\alpha+\beta J)
\end{align*}
are locally constant, whence constant on each connected component of $S\cap\OO$.
\item $f$ is a locally slice function according to~\cite[Definition 3.6]{geometricfunctiontheory} (see also~\cite[Definition 11.6]{librospringer2ed}).
\item At each $x_0\in\OO\cap\rr$, for any $v\in T_{x_0}\OO\simeq\hh$, we have $df_{x_0} v = v f'_c(x_0)$. As a consequence, $df_{x_0}$ is singular if, and only if, $f'_c(x_0)=0$. At each $x_0\in\OO\setminus\rr$, say $x_0\in\OO_{J}^>$, if we decompose $T_{x_0}\OO\simeq\hh$ as $\cc_J\oplus\cc_J^\perp$, then for any $v\in\cc_J$ and any $w\in\cc_J^\perp$ we have $df_{x_0} (v+w) = v f'_c(x_0)+w f'_s(x_0)$. As a consequence, $df_{x_0}$ is singular if, and only if, $f'_c(x_0)f'_s(x_0)^c\in\cc_J^\perp$.
\end{enumerate}
\end{proposition}

\begin{proof}
Points {\it 1.} and {\it 2.} are immediate consequences of Corollary~\ref{cor:localrepresentationformula}. In turn, they imply point {\it 3.} Point {\it 4.} now follows from~\cite[Proposition 3.5]{geometricfunctiontheory} (see also~\cite[\S8.4 and \S8.5]{librospringer2ed} and~\cite{altavilladifferential,gporientation,conformality} for the case when $\OO$ is a symmetric slice domain).
\end{proof}

In view of property {\it 3.} of Proposition~\ref{prop:sphericalvalueandderivative}, we point out that our Definition~\ref{def:sphericalvalueanderivative} is consistent with~\cite[Definition 11.6]{librospringer2ed}.


\section{Hinged points and global extension theorem for hinged domains}\label{sec:hinged}

\subsection{Hinged points}\label{subsec:hingedpoints}

\begin{assumption}
Throughout this subsection, we assume $\OO\subseteq\hh$ to be a speared domain and we adopt the notations
\[
D^J:=\phi_J^{-1}(\OO_J^\geqslant)\cup\overline{\phi_J^{-1}(\OO_J^\geqslant)}, \quad E^{J,K}:=\phi_J^{-1}(\OO_J^\geqslant)\cap\phi_K^{-1}(\OO_K^\geqslant),\quad D^{J,K}:=E^{J,K}\cup\overline{E^{J,K}}\]
for all distinct $J,K\in\s$. Moreover, we assume $f:\OO\to\hh$ to be a slice regular function and denote by $\{F^{J,K}:D^{J,K}\to\hh_\cc\}_{J,K\in\s, J\neq K}$ the double-index holomorphic stem family associated to $f$ and by $\{F^J:D^J\to\hh_\cc\}_{J\in\s}$ the holomorphic stem family associated to $f$.
\end{assumption}

As customary, in the previous assumption the symbol $\overline{E}$ denotes the image of a set $E\subseteq\rr_\cc$ through conjugation $z\mapsto\bar z$. We now define and study several relations on $\OO$, which will be useful to prove our Global Extension Theorem.

\begin{figure}[htbp]
\centering
\includegraphics[height=4cm]{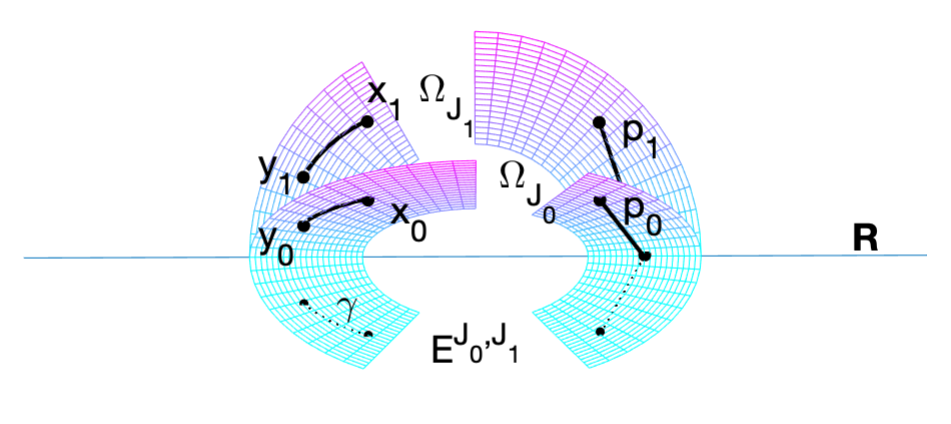}
\caption{An example where $(x_1,y_1)$ shadows $(x_0,y_0)$ in $\Omega$ and $p_1$ is strongly hinged to $p_0$ in $\Omega$.}\label{fig:shadowing}
\end{figure}

\begin{definition}\label{def:stronglyhinged}
Let us choose $\alpha,\alpha',\beta,\beta'\in\rr,J_0,J_1\in\s$ (with $\beta,\beta'\geq0$) such that $x_0=\alpha+\beta J_0,x_1=\alpha+\beta J_1,y_0=\alpha'+\beta' J_0,y_1=\alpha'+\beta' J_1$ all belong to $\OO$.

We say that $(x_1,y_1)$ \emph{shadows} $(x_0,y_0)$ in $\OO$ if there exists a path $\gamma:[0,1]\to E^{J_0,J_1}$ such that $\gamma(0)=\alpha+\beta\ui$ and $\gamma(1)=\alpha'+\beta'\ui$.

We say that $x_1$ is \emph{strongly hinged to $x_0$ in $\OO$} if there exists $y \in \OO \cap \rr$ such that $(x_1,y)$ shadows $(x_0,y)$.
\end{definition}

The expression ``strongly hinged'', used here for the first time, subsumes ideas already used in the articles~\cite{douren1,dourensabadini,geometricfunctiontheory,localrepresentation}. Figure~\ref{fig:shadowing} portrays an example of shadowing and an example of point that is strongly hinged to another point.

\begin{remark}\label{rmk:stronglyhinged}
Shadowing is a partial equivalence relation on $\bigcup_{J\in\s}(\OO_J^\geqslant\times\OO_J^\geqslant)$, i.e., it is symmetric and transitive. If $\OO_J^\geqslant$ is path-connected, then the relation is also reflexive, whence an equivalence relation, on $\OO_J^\geqslant\times\OO_J^\geqslant$. By construction, if $(x_1,y_1)$ shadows $(x_0,y_0)$, then: $\s_{x_1}=\s_{x_0}$; $\s_{y_1}=\s_{y_0}$; and $(y_1,x_1)$ shadows $(y_0,x_0)$.

Being strongly hinged is a reflexive and symmetric relation on $\OO$. By construction, if $x_1$ is strongly hinged to $x_0$, then $\s_{x_1}=\s_{x_0}$.
\end{remark}

For future use, we establish the following lemma.

\begin{lemma}\label{lem:stronglyhinged}
Let us choose $\alpha,\alpha',\beta,\beta'\in\rr,J_0,J_1\in\s$ (with $\beta,\beta'\geq0$) such that $x_0=\alpha+\beta J_0,x_1=\alpha+\beta J_1,y_0=\alpha'+\beta' J_0,y_1=\alpha'+\beta' J_1$ all belong to $\OO$.
\begin{enumerate}
\item If $(x_1,y_1)$ shadows $(x_0,y_0)$ and if $F^{J_0},F^{J_1}$ coincide in a neighborhood of $\alpha'+\beta'\ui$, then there exists a path-connected open subset of $D^{J_0,J_1}$ including both $\alpha+\beta\ui$ and $\alpha'+\beta'\ui$ where $F^{J_0},F^{J_1}$ coincide.
\item If $x_1$ is strongly hinged to $x_0$, then there exists speared open subset of $D^{J_0,J_1}$ including $\alpha+\beta\ui$ where $F^{J_0},F^{J_1}$ coincide. In particular, the connected component of $D^{J_0,J_1}$ including $\alpha+\beta\ui$ intersects the real axis.
\end{enumerate}
\end{lemma}

\begin{proof}
\begin{enumerate}
\item There exists a path $\gamma:[0,1]\to E^{J_0,J_1}$ such that $\gamma(0)=\alpha+\beta\ui$ and $\gamma(1)=\alpha'+\beta'\ui$. By Lemma~\ref{lem:neighborhoodcompact}, the path-connected open neighborhood $C_\varepsilon$ of $C:=\gamma([0,1])$ in $\rr_\cc$ is included in $D^{J_0}$ and in $D^{J_1}$. As a consequence, $C_\varepsilon$ is contained in $D^{J_0}\cap D^{J_1}=D^{J_0,J_1}$. Since the holomorphic maps $F^{J_0},F^{J_1}$ coincide near the point $\alpha'+\beta'\ui \in C_\varepsilon$, it follows that they coincide throughout the connected open set $C_\varepsilon$ (which includes $\alpha+\beta\ui$, too).
\item There exists $y=\alpha' \in \OO \cap \rr$ such that $(x_1,y)$ shadows $(x_0,y)$. Moreover, Lemma~\ref{lem:connectedcomponent} guarantees that $F^{J_0},F^{J_1}$ coincide in the neighborhood $\core_{\rr_\cc}(\OO)$ of $\alpha'$. If we repeat the construction made in point {1.}, we find a speared open subset $C_\varepsilon$ of $D^{J_0,J_1}$ including $\alpha+\beta\ui$ where $F^{J_0},F^{J_1}$ coincide.\qedhere
\end{enumerate}
\end{proof}

In preparation for the main result in this section, we will need one more relation.

\begin{definition}\label{def:hinged}
Take $t\in\nn^*:=\nn\setminus\{0\}$ and let $T:=\{0,\ldots,t\}$. Let us consider a finite (ordered) sequence $\{x_s\}_{s=0}^t$ in $\OO$. For each $s,s'\in T$ with $|s-s'|=1$, we say that there is a \emph{simple step} at $s_0:=\min\{s,s'\}$ if one of the following conditions is fulfilled:
\begin{enumerate}
\item $x_{s'}$ belongs to the connected component of $\s_{x_s}\cap\OO$ that includes $x_s$;
\item $x_{s'}$ is strongly hinged to $x_s$.
\end{enumerate}
If conditions {1.} and {2.} are false, but
\begin{enumerate}
\item[3.] there exist $s'',s'''\in T\setminus\{s,s'\}$ with $s''-s'''=s'-s$ such that $(x_{s''},x_{s'''})$ shadows $(x_s,x_{s'})$,
\end{enumerate}
and if we set $\underline{s}:=\min\{s,s',s'',s'''\},\overline{s}:=\max\{s,s',s'',s'''\}-1$, then we say that there is a \emph{double step} at $(\underline{s},\overline{s})$.

The sequence $\{x_s\}_{s=0}^t$ is called a \emph{chain of length $t$} connecting $x_0$ and $x_t$ if, after erasing from $T$ the index $t$ and all indices $s_0$ where a simple step happens, there remain in $T$ an even number $2\ell$ of indices, which can be listed in a form $\underline{s}_1,\overline{s}_1,\ldots,\underline{s}_\ell,\overline{s}_\ell$ such that, for all $m,n \in \{1,\ldots,\ell\}$:
\begin{itemize}
\item a double step happens at $(\underline{s}_m,\overline{s}_m)$ (whence $\underline{s}_m<\overline{s}_m$);
\item $\underline{s}_1<\ldots<\underline{s}_\ell$; in other words, the double steps are listed in the same order they start; and
\item if $m<n$, then either $\overline{s}_m<\underline{s}_n$ or $\overline{s}_m>\overline{s}_n$; in other words, if the $n$th double step starts before the $m$th double step is concluded, then the $n$th double step must also end before the $m$th double step does; equivalently, the double step at $(\underline{s}_n,\overline{s}_n)$ may be nested within the double step at $(\underline{s}_m,\overline{s}_m)$, but not intertwined with it.
\end{itemize}

For all points $x,x'\in\OO$, we say that $x'$ is \emph{hinged} to $x$ in $\OO$, and we write $x\sim x'$, if there exists a chain connecting $x$ and $x'$.
\end{definition}

An example of chain of length $4$ is portrayed in the forthcoming Figure~\ref{fig:chain}.

We will soon prove that $\sim$ is an equivalence relation. Before doing so, we establish a few useful technical results.

\begin{remark}\label{rmk:spheres}
Let us consider a finite sequence $\{x_s\}_{s=0}^t\subset\OO$. If there is a simple step at $s$, then $\s_{x_s}=\s_{x_{s+1}}$. If there is a double step at $(\underline{s},\overline{s})$, then $\s_{x_{\underline{s}}}=\s_{x_{\overline{s}+1}}$ and $\s_{x_{\underline{s}+1}}=\s_{x_{\overline{s}}}$. This is a consequence of the definitions of simple and double step, as well as of Remark~\ref{rmk:stronglyhinged}.
\end{remark}

\begin{lemma}\label{lem:chain}
If $\{x_s\}_{s=0}^t$ is a chain of length $t\geq2$ connecting $x_0$ and $x_t$, then one of the following facts is true.
\begin{enumerate}
\item There exists $u$ with $0<u<t$ such that $\{x_s\}_{s=0}^{u}$ is a chain of length $u$ connecting $x_0$ and $x_{u}$ (whence $x_0\sim x_{u}$) and $\{x_s\}_{s=u}^{t}$ is a chain of length $t-u$ connecting $x_{u}$ and $x_t$ (whence $x_{u}\sim x_t$).
\item There is a double step at $(0,t-1)$ and $\{x_s\}_{s=1}^{t-1}$ is a chain of length $t-2$ connecting $x_1$ and $x_{t-1}$ (whence $x_1\sim x_{t-1}$).
\end{enumerate}
The second case is excluded when $t=2$.
\end{lemma}

\begin{proof}
If there is a simple step at $0$, then the sequence $\{x_s\}_{s=0}^{1}$ is a chain of length $1$ connecting $x_0$ and $x_1$ and there exists no double step of the form $(0,s)$. Thus, $\{x_s\}_{s=1}^{t}$ is automatically a chain of length $t-1$ connecting $x_1$ and $x_{t}$. Overall, we are in case {\it 1.}

Suppose, instead, that there is not a simple step at $0$. Then, if we list all double steps as in Definition~\ref{def:hinged}, we find $(0,\overline{s}_1),\ldots,(\underline{s}_\ell,\overline{s}_\ell)$ for some natural number $\ell$. We separate two cases.

\begin{itemize}
\item If $\overline{s}_1<t-1$ and we set $u:=\overline{s}_1+1$, then we are in case {\it 1.} because $\{x_s\}_{s=0}^u$ is a chain of length $u$ connecting $x_0$ and $x_u$ and $\{x_s\}^{t}_{s=u}$ is a chain of length $t-u$ connecting $x_u$ and $x_{t}$. Indeed, every simple step in the original chain will still be a simple step in one of the two subsequences. Every double step $(\underline{s}_m,\overline{s}_m)$ with $m>1$ will be a double step in the first subsequence if it is nested within $(0,\overline{s}_1)$ (i.e., if $0<\underline{s}_m<\overline{s}_m<\overline{s}_1$) or a double step in the second subsequence if it starts after the double step at $(0,\overline{s}_1)$ is concluded (i.e., if $\overline{s}_1+1=u\leq\underline{s}_m$).

\item If $\overline{s}_1=t-1$, then there is a double step at $(0,t-1)$. We are in case {\it 2.} because $\{x_s\}_{s=1}^{t-1}$ is a chain of length $t-2$ connecting $x_1$ and $x_{t-1}$. Indeed, every simple step in the original chain will still be a simple step in the subsequence. Every double step $(\underline{s}_m,\overline{s}_m)$ with $m>1$ will be a double step in the subsequence because it is nested within the double step at $(0,t-1)$ (i.e., $0<\underline{s}_m<\overline{s}_m<t-1$).\qedhere
\end{itemize}
\end{proof}

\begin{lemma}
The relation $\sim$ is an equivalence relation on $\OO$. Moreover, $x\sim x'$ implies $\s_x=\s_{x'}$.
\end{lemma}

\begin{proof}
The relation $\sim$ is reflexive: for each $x\in\OO$, it holds $x\sim x$ because we can set $t=1, x_0=x, x_1=x$ and observe that there is a simple step at $0$ because $x_0,x_1$ fulfill condition {1.} of Definition \ref{def:hinged}.

The relation $\sim$ is symmetric: if $\{x_s\}_{s=0}^t$ is a chain connecting $x$ and $x'$, then $\{x_{t-s}\}_{s=0}^t$ is a chain connecting $x'$ and $x$. Indeed, if there is a simple step at $s_0$ in the original sequence, then there is a simple step at $t-s_0-1$ in the new sequence because conditions {1.} or {2.} are unchanged when we swap $s$ and $s'$. If there is a double step at $(\underline{s},\overline{s})$ in the original sequence, then there is a double step at $(t-\overline{s}-1,t-\underline{s}-1)$ in the new sequence because condition {3.} stays equivalent if we swap $s$ and $s'$, while also swapping $s''$ and $s'''$.

The relation $\sim$ is transitive. Assume $x\sim x'$ and $x'\sim x''$. Let $\{x_s\}_{s=0}^t$ be a chain connecting $x$ and $x'$ and let $\{y_u\}_{u=0}^{t'}$ be a chain connecting $x'$ and $x''$. We form a longer sequence $x_0=x,x_1,\ldots,x_t=x',x_{t+1},\ldots,x_{t+t'}=x''$ with $x_{t+u}:=y_u$ for all $u \in \{1,\ldots,t'\}$ and argue that $\{x_s\}_{s=0}^{t+t'}$ is a chain of length $t+t'$ connecting $x$ and $x''$, thus proving that $x\sim x''$. Indeed, any simple step happening at $s_0$ in the first short sequence will still be a simple step at $s_0$ in the long sequence, while any simple step happening at $u_0$ in the second short sequence corresponds to a simple step at $t+u_0$ in the long sequence. If, according to Definition~\ref{def:hinged}, the double steps in the first short sequence can be listed as $(\underline{s}_1,\overline{s}_1),\ldots,(\underline{s}_\ell,\overline{s}_\ell)$ and the double steps in the second short sequence can be listed as $(\underline{u}_1,\overline{u}_1),\ldots,(\underline{u}_{\ell'},\overline{u}_{\ell'})$, then the double steps in the long sequence can be listed as $(\underline{s}_1,\overline{s}_1),\ldots,(\underline{s}_\ell,\overline{s}_\ell),(t+\underline{u}_1,t+\overline{u}_1),\ldots,(t+\underline{u}_{\ell'},t+\overline{u}_{\ell'})$. Here, $\underline{s}_1<\ldots<\underline{s}_\ell<t+\underline{u}_1<\ldots<t+\underline{u}_{\ell'}$ and no double steps are intertwined because they were not in either short list and because $\max\{\overline{s}_1,\ldots,\overline{s}_\ell\}<t\leq t+\underline{u}_1=\min\{ t+\underline{u}_1,\ldots, t+\underline{u}_{\ell'}\}$.

We have therefore proven the first statement.

The second statement can be proven by induction on the length $t$ of the chain $\{x_s\}_{s=0}^t$ connecting $x$ and $x'$. If $t=1$, then there is a simple step at $0$: Remark~\ref{rmk:spheres} guarantees that $\s_{x}=\s_{x_0}=\s_{x_1}=\s_{x'}$. Now suppose the thesis true for all $t<t_0$ and let us prove it for $t=t_0$. If there exists $u$ such that $\{x_s\}_{s=0}^{t_0}$ breaks at $s=u$ into two chains of lengths less than $t_0$, then the inductive hypothesis yields that $\s_{x}=\s_{x_u}=\s_{x'}$. If there exists no such $u$, then Lemma~\ref{lem:chain} guarantees that there is a double step at $(0,t_0-1)$. In such a case, Remark~\ref{rmk:spheres} guarantees that $\s_{x}=\s_{x_0}=\s_{x_{t_0}}=\s_{x'}$.
\end{proof}

Hinged points have an extremely relevant property: if $x\sim x'$, then the local representation of $f$ near $x'$ is consistent with the local representation of $f$ near $x$. Equivalently, the corresponding holomorphic stem functions locally coincide. 

\begin{theorem}\label{thm:hingedpoints}
Let $\OO\subseteq\hh$ be a speared domain, let $f:\OO\to\hh$ be a slice regular function and let $\{F^J:D^J\to\hh_\cc\}_{J\in\s}$ be the holomorphic stem family associated to $f$. Let $\alpha_0,\beta_0\in\rr$ (with $\beta_0\geq0$) and $I,I'\in\s$ be such that $x=\alpha_0+\beta_0 I,x'=\alpha_0+\beta_0 I' \in \OO$. If $x\sim x'$, then there exists an open neighborhood $U$ of $\alpha_0+\beta_0\ui$ in $D^{I,I'}$ where $F^{I},F^{I'}$ coincide. As a consequence, for all ($\alpha,\beta\in\rr$ with $\beta\geq0$ such that) $\alpha+\beta\ui\in U$, the equality $f^\circ_s(\alpha+\beta I)=f^\circ_s(\alpha+\beta I')$ holds and the equality $f'_s(\alpha+\beta I)=f'_s(\alpha+\beta I')$ holds provided $\beta>0$.
\end{theorem}

\begin{proof}
Since $x\sim x'$, there exist a natural number $t\geq1$ and a chain $\{x_s\}_{s=0}^t$ connecting $x$ and $x'$. We prove the thesis by induction on the length $t$ of the chain.

If $t=1$, then there is a simple step at $0$. If $x'=x_1$ belongs to the connected component of $\s_{x}\cap\OO$ that includes $x=x_0$, then the thesis follows immediately from Lemma~\ref{lem:connectedcomponent}. If, instead, $x'=x_1$ is strongly hinged to $x=x_0$, then case {\it 2.} in Lemma~\ref{lem:stronglyhinged} yields the thesis.

Now choose $t_0\geq2$, suppose the thesis true for all $t<t_0$ and let us prove it for $t=t_0$.
\begin{itemize}
\item Let us first assume that there exists $u$ with $0<u<t_0$ such that $\{x_s\}_{s=0}^{t_0}$ breaks at $s=u$ into two chains of lengths $u$ and $t_0-u$. In particular, $x\sim x_u$ and $x_u\sim x'$, whence $x_u=\alpha_0+\beta_0 I''$ for some $I''\in\s$. Since $u$ and $t_0-u$ are both strictly less than $t_0$, the inductive hypothesis yields that there exist open neighborhoods $U,U'$ of $\alpha_0+\beta_0\ui$ in $D^{I,I''},D^{I'',I'}$ (respectively) such that $F^{I},F^{I''}$ coincide in $U$ and $F^{I''},F^{I'}$ coincide in $U'$. It follows that $U\cap U'$ is an open neighborhood of $\alpha_0+\beta_0\ui$ in $D^{I,I''}\cap D^{I'',I'}\subseteq D^{I,I'}$ where $F^{I},F^{I'}$ coincide, as desired.
\item Let us now assume, instead, that there exists no $u$ with $0<u<t_0$ such that $\{x_s\}_{s=0}^{t_0}$ breaks at $s=u$ into two chains. Lemma~\ref{lem:chain} guarantees that $t_0\geq3$ and that there is a double step at $(0,t_0-1)$ and that $\{x_s\}_{s=1}^{t_0-1}$ is a chain of length $t_0-2$ from $x_1$ to $x_{t_0-1}$. In particular, by the definition of double step and by Remark~\ref{rmk:spheres}, there exist $\alpha_1,\beta_1\in\rr$ (with $\beta_1\geq0$) such that $x_1=\alpha_1+\beta_1 I,x_{t_0-1}=\alpha_1+\beta_1I'$. Since $t_0-2<t_0$, the inductive hypothesis guarantees that there exists an open neighborhood of $\alpha_1+\beta_1\ui$ in $D^{I,I'}$ where $F^{I},F^{I'}$ coincide. Since $(x_{t_0},x_{t_0-1})$ shadows $(x_0,x_1)$, point {\it 1.} in Lemma~\ref{lem:stronglyhinged} guarantees that there exists a path-connected open subset of $D^{I,I'}$ including both $\alpha_0+\beta_0\ui$ and $\alpha_1+\beta_1\ui$ where $F^{I},F^{I'}$ coincide.
\end{itemize}
To prove the last statement, it suffices to perform the following computation, where we apply Remark~\ref{rmk:sphericalvalueandderivative}:
\begin{align*}
&f^\circ_s(\alpha+\beta I)=F^{I}_1(\alpha+\beta\ui)=F^{I'}_1(\alpha+\beta\ui)=f^\circ_s(\alpha+\beta I')\,,\\
&f'_s(\alpha+\beta I)=\beta^{-1}F^{I}_2(\alpha+\beta\ui)=\beta^{-1}F^{I'}_2(\alpha+\beta\ui)=f'_s(\alpha+\beta I')
\end{align*}
for $\alpha+\beta\ui\in U$.
\end{proof}

\subsection{Hinged domains and extension theorem}\label{subsec:extensiontheorem}

This subsection is devoted to the class of domains described in the next definition and proves for them the Global Extension Theorem and the Representation Formula.

\begin{definition}\label{def:hingeddomains}
A speared domain $\OO\subseteq\hh$ is termed \emph{hinged} if, for each $x\in\OO$, it holds $x\sim x'$ for all $x'\in\s_x\cap\OO$; or, equivalently, if $\s_x\cap\OO$ is an equivalence class under the equivalence relation $\sim$.
\end{definition}

Over a hinged domain, we can draw the following consequence from Theorem~\ref{thm:hingedpoints}.

\begin{corollary}\label{cor:hingeddomain}
Let $\OO\subseteq\hh$ be a hinged domain, let $f:\OO\to\hh$ be a slice regular function and let $\{F^J:D^J\to\hh_\cc\}_{J\in\s}$ be the holomorphic stem family associated to $f$. If we set $D:=\bigcup_{J\in\s}D^J$, then there exists a holomorphic stem function $F:D\to\hh_\cc$ such that, for each $J\in\s$, the function $F^J$ is the restriction $F_{|_{D^J}}$.
\end{corollary}

\begin{proof}
For each $J\in\s$, we define $F$ to equal $F^J$ in $D^J$. The function $F$ is well defined: by construction, for each $\alpha,\beta\in\rr$ (with $\beta\geq0$) such that $\alpha+\beta\ui\in D^J\cap D^{J'}=D^{J,J'}$, the equivalence $\alpha+\beta J\sim \alpha+\beta J'$ holds, whence $F^J$ and $F^{J'}$ agree near $\alpha+\beta\ui$ by Theorem~\ref{thm:hingedpoints}. Since each $F^J$ is a holomorphic stem function, it follows that $F$ is a holomorphic stem function, as desired.
\end{proof}

We are now in a position to prove the announced Global Extension Theorem. It had been proven in~\cite[Theorem 4.5]{localrepresentation} only for the so-called simple slice domains. In the forthcoming Subsection~\ref{subsec:hingeddomains}, we will recall the definition of simple slice domain, prove that all simple slice domains are hinged and provide many examples of hinged domains that are not simple.

\begin{theorem}[Global extension for hinged domains]\label{thm:extension}
Let $\OO\subseteq\hh$ be a hinged domain and let $f:\OO\to\hh$ be a slice regular function. Then there exists a unique slice regular function $\widetilde{f}:\widetilde{\OO}\to\hh$ that extends $f$ to the symmetric completion $\widetilde{\OO}$ of its domain $\OO$.
\end{theorem}

\begin{proof}
We recall that $\{F^J:D^J\to\hh_\cc\}_{J\in\s}$ denotes the holomorphic stem family associated to $f$ and $\{f^J:\OO^J\to\hh\}_{J\in\s}$ denotes the slice regular family associated to $f$.
Corollary~\ref{cor:hingeddomain} guarantees that there there exists a holomorphic stem function $F$ from the union $D:=\bigcup_{J\in\s}D^J$ to $\hh_\cc$ such that, for each $J\in\s$, the function $F^J$ is the restriction $F_{|_{D^J}}$. We remark that $\OO_D=\widetilde{\OO}$ and define $\widetilde{f}$ to be the element $\I(F)$ of $\sr(\widetilde{\OO})$ induced by $F$. For each $J\in\s$, the definition $f^J:=\I(F^J)$ yields that $f^J$ is the restriction of $\widetilde{f}$ to $\OO^J:=\OO_{D^J}\subseteq\widetilde{\OO}$. We can now apply Theorem~\ref{thm:localextension} to conclude that $f$ coincides with $\widetilde{f}$ in $\bigcup_{J\in\s}\OO_J^\geqslant=\OO$.
\end{proof}

Theorem~\ref{thm:extension} immediately implies the Global Representation Formula for hinged domains.

\begin{corollary}[Global representation formula for hinged domains]
Let $\OO\subseteq\hh$ be a hinged domain and let $f:\OO\to\hh$ be a slice regular function. For all $\alpha,\beta\in\rr$ with $\beta\geq0$ and all $I,J,K\in\s$ with $J\neq K$, such that $\alpha+\beta I,\alpha+\beta J,\alpha+\beta K\in\OO$, the equality
\[f(\alpha+\beta I)=(J-K)^{-1} \left[J f(\alpha+\beta J) - K f(\alpha+\beta K)\right] + I (J-K)^{-1} \left[f(\alpha+\beta J) - f(\alpha+\beta K)\right]\]
holds. As a consequence, for each $x\in\OO$, the spherical value function $f^\circ_s:\OO\to\hh$ is constant on $\s_x\cap\OO$ and, if $x\not\in\rr$, spherical derivative $f'_s:\OO\setminus\rr\to\hh$ is constant on $\s_x\cap\OO$.
\end{corollary}

An equivalent statement is: if $\{f^J:\OO^J\to\hh\}_{J\in\s}$ is the slice regular family associated to a slice regular function $f$ on a hinged domain $\OO$, then $f$ coincides with $f^J$ in $\OO\cap\OO^J$.


\section{Study of speared domains and hinged domains}\label{sec:classesofdomains}

This section is devoted to the study of speared domains and hinged domains.

\subsection{Speared domains}\label{subsec:speared}

We begin by proving that the class of speared domains (see Definition~\ref{def:speared}) includes the class of slice domains (see Definition~\ref{def:slicedomain}).

\begin{proposition}
Every quaternionic slice domain is a speared domain.
\end{proposition}

\begin{proof}
Let $\OO\subseteq\hh$ be a domain. If $\OO$ is not speared, then it is not a slice domain. Indeed: assume there exists $J\in\s$ such that $\OO^\geqslant_J$ has a connected component $\mathcal{C}$ with $\mathcal{C}\cap\rr=\emptyset$. Necessarily, $\mathcal{C}$ is also a connected component of $\OO_J$. In this situation, either $\OO_J$ is not connected or $\OO_J=\mathcal{C}$ does not intersect $\rr$, whence $\OO\cap\rr=\emptyset$. In either case, $\OO$ is not a slice domain.
\end{proof}

We now prove that every speared open subset can be locally shrunk to a slice domain.

\begin{proposition}\label{prop:localslicedomain}
Let $U$ be a speared open subset of $\hh$. For every $x_0\in U$, there exists a slice domain $\Lambda_{x_0}\subseteq U$ including $x_0$.
\end{proposition}

\begin{proof}
If $x_0$ is a real point $\alpha_0\in\rr$, we can take $\Lambda_{x_0}$ to be any open ball $B(\alpha_0,\delta)$ contained in $U$. We therefore assume $x_0\in U\setminus\rr$. Our proof borrows some techniques already used in the proofs of Theorems~\ref{thm:singleindex} and~\ref{thm:localextension}.

Let $I\in\s$ be such that $x_0\in U_I^>$. By Remark~\ref{rmk:path}, there exists a path $\gamma:[0,1]\to\phi_I^{-1}(U_I^{\geqslant})$ such that $\gamma(0)$ is a real number $\alpha$ and $\phi_I(\gamma(1))=x_0$. We can pick $\delta>0$ such that $B=B(\alpha,\delta)\subseteq U$; let us denote by $\Delta$ the disk of radius $\delta$ centered at $\alpha$ in $\rr_\cc$. We set
\[C:=\gamma([t_0,1]),\quad t_0:=\sup\{t\in[0,1]\,|\,\gamma(t)\in\Delta\}\in(0,1]\,.\]
If we apply Lemma~\ref{lem:neighborhoodcompact} to $C$ and to $Y:=U\setminus\rr$, we conclude that there exists a real number $\varepsilon>0$ with the following property: for the path-connected open neighborhood
\[C_\varepsilon:=\{z\in\rr_\cc \ |\  \mathrm{dist}(z,C)<\varepsilon\}\]
of $C$, it holds $\phi_K(C_\varepsilon)\subset U\setminus\rr$ (whence $\phi_K(C_\varepsilon)\subset U_K^>$) for all $K\in\scap(I,\varepsilon)$. We remark that, since $C$ intersects the boundary of $\Delta$ at $\gamma(t_0)$, its open neighborhood $C_\varepsilon$ intersects $\Delta$ at some point $\alpha_0+\beta_0\ui$. If we define
\[\Lambda_{x_0}:=B\cup\bigcup_{K\in\scap(I,\varepsilon)}\phi_K(C_\varepsilon)\,,\]
then, by construction, $\Lambda_{x_0}$ is an open neighborhood of $x_0$ in $U$ that intersects the real axis in the interval $(\alpha-\delta,\alpha+\delta)$. Moreover, for each $J\in\s$, its slice $\big(\Lambda_{x_0}\big)_{J}$ is connected because it is the union of the following open connected subsets of $\cc_J$: the disk $B_J=B_{-J}$ of radius $\delta$ centered at $\alpha$ in $\cc_J$; if $J\in\scap(I,\varepsilon)$, the path-connected set $\phi_J(C_\varepsilon)$ (which intersects the disk at $\alpha_0+\beta_0J$); and, if $-J\in\scap(I,\varepsilon)$, the path-connected set $\phi_{-J}(C_\varepsilon)$  (which intersects the disk at $\alpha_0-\beta_0J$). The same argument yields that $\Lambda_{x_0}$ is connected, whence a domain.
\end{proof}

We conclude this subsection by proving that the inclusion of the class of slice domains within the class of speared domains is proper. Indeed, we construct a family of examples of speared domains that are not slice domains.

\begin{figure}[htbp]
\centering
\includegraphics[height=7cm]{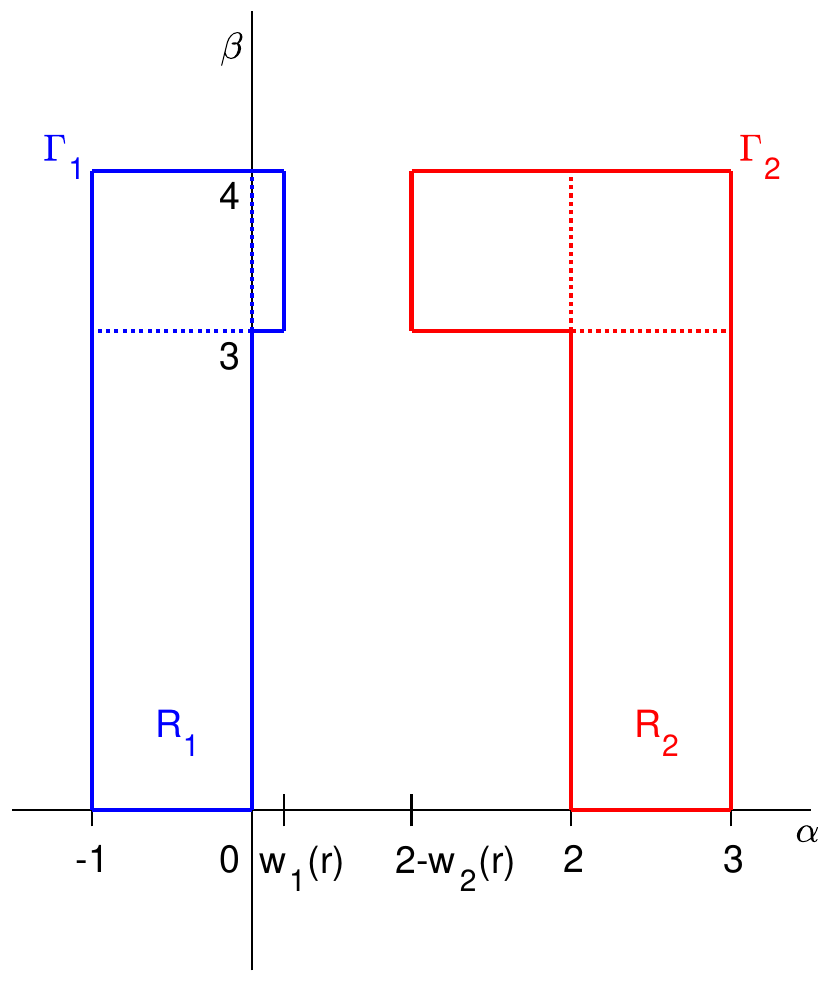}
\caption{The planar domain $D_r=\Gamma_1\cup\Gamma_2$.}\label{fig:Dr}
\end{figure}

\begin{examples}\label{ex:speared}
In $\rr_\cc\simeq\rr^2$, we consider the rectangles $R_1:=(-1,0)\times[0,4)$ and $R_2:=(2,3)\times[0,4)$ and their disjoint union $R:=R_1\cup R_2$. Let us fix two lower semicontinuous width functions $w_1,w_2:[-1,1]\to[0,2]$ such that $w_1(\pm1)=0=w_2(\pm1)$ and such that there exists $r_0\in(-1,1)$ with $w_1(r_0)+w_2(r_0)>2$. We define, for each $r\in[-1,1]$, the $\Gamma$-shaped sets
\[\Gamma_1:=R_1\cup\Big((-1,w_1(r))\times(3,4)\Big),\quad\Gamma_2:=R_2\cup\Big((2-w_2(r),3)\times(3,4)\Big)\,,\]
portrayed in Figure~\ref{fig:Dr}, and their union
\[D_r:=\Gamma_1\cup\Gamma_2\,.\]
Whenever $w_1(r)+w_2(r)\leq2$, as in Figure~\ref{fig:Dr}, then $\Gamma_1$ and $\Gamma_2$ are separate connected components of $D_r$. This is true, in particular for $D_{\pm1}=R$. When, instead, $w_1(r)+w_2(r)>2$, as it happens at $r=r_0$, then $D_r$ is the (connected) $C$-shaped set $C:=R\cup\Big((-1,3)\times(3,4)\Big)$. We always have $D_r\subseteq C$ and $D_r\cap\rr=R\cap\rr=(-1,0)\cup(2,3)$, regardless of $r$. We remark that $R,D_r,C$ (for all $r\in[-1,1]$) are all open subsets of $\rr_\cc^\geqslant$. We define $\underline{\OO}\subseteq\hh$ by choosing $\underline{\OO}_J^\geqslant$ for each $J\in\s$, as follows: for all $x_1,x_2,x_3\in\rr$ such that $x_1^2+x_2^2+x_3^2=1$ (so that $x_1i+x_2j+x_3k\in\s$), we set
\[\underline{\OO}_{x_1i+x_2j+x_3k}^\geqslant:=\phi_{x_1i+x_2j+x_3k}(D_{x_3})\,.\]
We can prove that $\underline{\OO}$ is a speared domain but not a slice domain, as follows.
\begin{itemize}
\item We first prove that $\underline{\OO}$ is an open subset of $\hh$ by picking $\alpha_0+\beta_0 I\in\underline{\OO}$ (with $\alpha_0,\beta_0\in\rr, \beta_0\geq0$ and $I=y_1i+y_2j+y_3k\in\s$) and finding an open neighborhood $U$ of $\alpha_0+\beta_0 I$ included in $\underline{\OO}$. If $\alpha_0+\beta_0\ui\in R$, then we can choose $U$ to be the circularization $\OO_R$ of $R$. If $\alpha_0+\beta_0\ui\in (-1,w_1(y_3))\times(3,4)$, then we can pick $\alpha_1$ such that $w_1(y_3)>\alpha_1>\alpha_0$ and set
\[U:=\{\alpha+\beta(x_1i+x_2j+x_3k)\,|\,\alpha+\beta\ui\in(-1,\alpha_1)\times(3,4),x_1i+x_2j+x_3k\in\s,w_1(x_3)>\alpha_1\}\,.\]
Indeed: $U$ includes $\alpha_0+\beta_0 I$ by construction; $U$ is open because it is the ``product'' of the open rectangle $(-1,\alpha_1)\times(3,4)$ and of the open subset of $\s$ defined by the inequality $w_1(x_3)>\alpha_1$ (which is open because the superlevel set $w_1^{-1}((\alpha_1,2])$ is open). Similarly, if $\alpha_0+\beta_0\ui\in(2-w_2(y_3),3)\times(3,4)$, then we can pick $\alpha_2$ such that $2-w_2(y_3)<\alpha_2<\alpha_0$ and set
\[U:=\{\alpha+\beta(x_1i+x_2j+x_3k)\,|\,\alpha+\beta\ui\in(\alpha_2,3)\times(3,4),x_1i+x_2j+x_3k\in\s,2-w_2(x_3)<\alpha_2\}\,.\]
Indeed: $U$ includes $\alpha_0+\beta_0 I$ by construction; $U$ is open because it is the ``product'' of the open rectangle $(\alpha_2,3)\times(3,4)$ and of the open subset of $\s$ defined by the inequality $2-w_2(x_3)<\alpha_2$ (which is open because the superlevel set $w_2^{-1}((2-\alpha_2,2])$ is open).
\item $\underline{\OO}$ is connected, whence a domain, because, for $J_0:=\sqrt{1-r_0^2}j+r_0k$, the $C$-shaped half-slice
\[\underline{\OO}_{J_0}^\geqslant=\phi_{J_0}(D_{r_0})=\phi_{J_0}(C)\]
is connected and because, for each $J\in\s$, every connected component of $\underline{\OO}_J^\geqslant$ intersects $\underline{\OO}_{J_0}^\geqslant$ in the real interval $(-1,0)$ or in the real interval $(2,3)$.
\item The previous argument also proves that $\underline{\OO}$ is a speared domain.
\item $\underline{\OO}$ is not a slice domain because its slice
\[\underline{\OO}_k=\phi_k(D_1)\cup\phi_{-k}(D_{-1})=\phi_k(R)\cup\phi_{-k}(R)\]
is the disjoint union between the rectangle $\phi_k((-1,0)\times(-4,4))$ and the rectangle $\phi_k((2,3)\times(-4,4))$.
\end{itemize}
We remark that $\spine(\underline{\OO})=B(-1/2,1/2)\cup B(5/2,1/2)$ and $\core(\underline{\OO})$ is the circularization $\OO_R$ of $R=R_1\cup R_2$. The planar open sets $\spine_{\rr_\cc}(\underline{\OO})$ and $\core_{\rr_\cc}(\underline{\OO})$ are portrayed in Figure~\ref{fig:spinecore}.
\end{examples}

\begin{figure}[htbp]
\centering
\includegraphics[height=6.8cm]{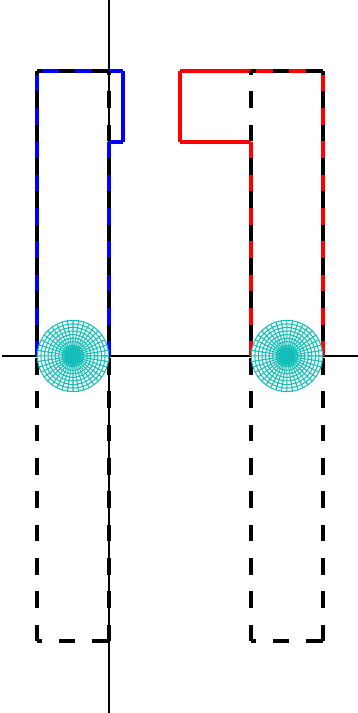}
\caption{The planar set $\spine_{\rr_\cc}(\underline{\OO})$ is the union of the open disks of radius $1/2$ centered at $-1/2$ and at $5/2$ within $\rr_\cc$ (meshed). The planar set $\core_{\rr_\cc}(\underline{\OO})$ is the union of the open rectangles $(-1,0)\times(-4,4)$ and $(2,3)\times(-4,4)$ within $\rr_\cc$ (with dashed boundaries).}\label{fig:spinecore}
\end{figure}

It is an open question whether an interesting theory of slice regular functions on such a large class of domains is also available in several quaternionic variables. Indeed, the theory developed in~\cite{gpseveral} is restricted to (the multidimensional analogs of) symmetric slice domains and product domains. A positive answer to this open question may widen the class of quaternionic manifolds to which the direct approach of~\cite{ggsdirectapproach}, based on the use of slice regular functions in several quaternionic variables, can be applied.

\subsection{Hinged domains}\label{subsec:hingeddomains}

We now prove that the class of hinged domains strictly includes the class of simple slice domains defined in~\cite{localrepresentation}. We rephrase the definition of simple slice domain according to the present setup, as follows.

\begin{definition}\label{def:simple}
A slice domain $\OO\subseteq\hh$ is called \emph{simple} if, for all distinct $J,K\in\s$, the intersection $\phi_J^{-1}(\OO_J^>)\cap\phi_K^{-1}(\OO_K^>)$ is connected.
\end{definition}

Before proving the announced inclusion, we propose a new notion valid for speared domains $\OO$, which is weaker than the property of being simple in the special case when $\OO$ is a slice domain.

\begin{definition}
Let $\OO\subseteq\hh$ be a speared domain. Then $\OO$ is \emph{spear-simple} if, for all distinct $J,K\in\s$, every connected component of the intersection $\phi_J^{-1}(\OO_J^\geqslant)\cap\phi_K^{-1}(\OO_K^\geqslant)$ intersects the real axis.
\end{definition}

We are now ready for the announced result.

\begin{proposition}
\begin{enumerate}
\item Every simple slice domain is spear-simple.
\item Every spear-simple domain is a hinged domain.
\end{enumerate}
\end{proposition}

\begin{proof}
Let $\OO\subseteq\hh$ be a speared domain.
\begin{enumerate}
\item First assume $\OO$ to be a simple slice domain. In particular, $\OO\cap\rr\neq\emptyset$. Moreover, for all distinct $J,K\in\s$, the intersection $\phi_J^{-1}(\OO_J^>)\cap\phi_K^{-1}(\OO_K^>)$ is connected. Since
\[A:=\phi_J^{-1}(\OO_J^\geqslant)\cap\phi_K^{-1}(\OO_K^\geqslant)=(\phi_J^{-1}(\OO_J^>)\cap\phi_K^{-1}(\OO_K^>))\cup(\OO\cap\rr)\,,\]
is included in the closure of $\phi_J^{-1}(\OO_J^>)\cap\phi_K^{-1}(\OO_K^>)$ in $\rr_\cc$, it follows that $A$ is connected. Moreover, $A$ intersects the real axis. As a consequence, $\OO$ is spear-simple.
\item Now assume $\OO$ to be spear-simple and let us prove that $\OO$ is a hinged domain, i.e., that for all $\alpha,\beta\in\rr$ with $\beta\geq0$ and for every two points $x=\alpha+\beta J,y=\alpha+\beta K$ belonging to the intersection between the sphere $S:=\alpha+\beta\s$ and $\OO$, the equivalence $x\sim y$ holds. This is clearly true when $J=K$, i.e., $x=y$. Now consider the case when $J\neq K$: the connected component $\phi_J^{-1}(\OO_J^\geqslant)\cap\phi_K^{-1}(\OO_K^\geqslant)$ that includes $\alpha+\beta\ui$ intersects the real axis at some point $\alpha'$. In particular, there exists a path from $\alpha+\beta\ui$ to $\alpha'$ in $\phi_J^{-1}(\OO_J^\geqslant)\cap\phi_K^{-1}(\OO_K^\geqslant)$. Thus: $(y,\alpha')$ shadows $(x,\alpha')$; the point $y$ is strongly hinged to the point $x$; and $x\sim y$, as desired.\qedhere
\end{enumerate}
\end{proof}

\begin{example}
It has been proven in~\cite[Proposition 2.18]{geometricfunctiontheory} that every open subset of $\hh$ that is starlike with respect to a real point is a simple slice domain, whence a spear-simple and hinged domain.
\end{example}

In the forthcoming Examples~\ref{ex:mainsail} we will exhibit a family of examples of hinged domains that are not simple slice domains. Examples~\ref{ex:spearsimple} will present a subfamily comprising spear-simple domains that are not simple slice domains. Examples~\ref{ex:sconnected} will provide one example of hinged domain that is not spear-simple.

An example of speared domain which is not a hinged domain can be found in~\cite[Pages 4--5]{douren1}: in that work, the authors present an example of slice domain where the Global Extension Theorem does not hold. To see explicitly that that specific slice domain is not simple (nor spear-simple), see~\cite[Example 4.4]{geometricfunctiontheory}. Another related article is~\cite{dourensabadini}.

It is useful to describe two more classes of speared domains, distinct from the class of spear-simple domains, both included in the class of hinged domains.

\begin{definition}
Let $U$ be an open subset of $\hh$. We say that $U$ is \emph{$\s$-connected} if, for each $x\in U$, the intersection $\s_x\cap U$ is connected. We say that $U$ \emph{has a main sail} if there exists $J_0\in\s$ such that $\phi_J^{-1}(U_J^\geqslant)\subseteq\phi_{J_0}^{-1}(U_{J_0}^\geqslant)$ for all $J\in\s$; if this is the case, then $U_{J_0}^\geqslant$ is called a \emph{main sail} for $U$.
\end{definition}

\begin{proposition}
\begin{enumerate}
\item Every $\s$-connected speared domain is a hinged domain.
\item Every speared domain having a main sail is a hinged domain.
\end{enumerate}
\end{proposition}

\begin{proof}
Let $\OO\subseteq\hh$ be a speared domain.
\begin{enumerate}
\item Assume $\OO$ to be $\s$-connected. Take any $x\in\OO$: since $\s_x\cap\OO$ is connected, there is a simple step from $x$ to any $y\in\s_x\cap\OO$, whence $x\sim y$. As a consequence, $\OO$ is a hinged domain, as desired. 
\item Assume $\OO$ to have a main sail $\OO_{J_0}^\geqslant$. We claim that, for all $\alpha,\beta\in\rr$ with $\beta\geq0$ and for every $J$ such that the point $x=\alpha+\beta J$ belongs to $\OO$, the point $x_0=\alpha+\beta J_0$ is strongly hinged to $x$, whence $x\sim x_0$. Now, for each point $y=\alpha+\beta K$ belonging to $\OO$ it holds $y\sim x_0$ and $x\sim x_0$, whence $y\sim x$ by symmetry and transitivity. Thus, $\OO$ is a hinged domain, as desired.\\
We are left with proving our claim. Since $\OO$ is a speared domain, the connected component of $\OO_J^\geqslant$ that includes $x$ also includes a real point $\alpha'$. Take any path within $\phi_J^{-1}(\OO_J^\geqslant)$ from $\alpha+\beta\ui$ to $\alpha'$. Since $\phi_J^{-1}(\OO_J^\geqslant)\subseteq\phi_{J_0}^{-1}(\OO_{J_0}^\geqslant)$, the support of the path is also included in $\phi_{J_0}^{-1}(\OO_{J_0}^\geqslant)$. Thus, $(x,\alpha')$ shadows $(x_0,\alpha')$ and $x_0$ is strongly hinged to $x$ in $\OO$, as claimed.\qedhere
\end{enumerate}
\end{proof}

We are now ready for the announced examples of hinged domains that are not simple slice domains.

\begin{examples}\label{ex:mainsail}
Each speared domain $\underline{\OO}$ constructed in Examples~\ref{ex:speared} is a hinged domain because it has a main sail: namely, its $C$-shaped half-slice $\underline{\OO}_{J_0}^\geqslant=\phi_{J_0}(C)$. We have already proven that it is not a slice domain.
\end{examples}

By carefully choosing the lower semicontinuous width functions $w_1,w_2:[-1,1]\to[0,2]$ used in the construction of $\underline{\OO}$ in Examples~\ref{ex:speared}, we can illustrate all classes of domains defined in the present subsection with examples (not belonging to the class of simple slice domains). Eventually, our examples will show that each class is distinct from all others.

\begin{examples}\label{ex:spearsimple}
Consider again the (non simple) speared domain $\underline{\OO}$ constructed in Examples~\ref{ex:speared}. Assume $w_1$ and $w_2$ to coincide throughout $[-1,1]$. Then, for each choice of distinct $J=x_1i+x_2j+x_3k,K=y_1i+y_2j+y_3k\in\s$, it holds
\[\phi_J^{-1}(\underline{\OO}_J^\geqslant)\cap\phi_K^{-1}(\underline{\OO}_K^\geqslant)=D_{x_3}\cap D_{y_3}\in\{D_{x_3},D_{y_3}\}\,.\]
Since (both for $r=x_3$ and for $r=y_3$) every connected component of $D_r$ includes the real interval $(-1,0)$ or the real interval $(2,3)$, we conclude that $\underline{\OO}$ is spear-simple.

For instance, let us pick the width functions
\begin{align*}
&w_1(r):=\left\{
\begin{array}{lll}
4r+4 & \mathrm{\ if\ } & r\in[-1,-1/2]\\
-4r & \mathrm{\ if\ } & r\in[-1/2,0]\\
4r & \mathrm{\ if\ } & r\in[0,1/2]\\
-4r+4 & \mathrm{\ if\ } & r\in[1/2,1]
\end{array}
\right.\,,\\
&w_2(r):=w_1(r) \mathrm{\quad for\ all\quad} r\in[-1,1]\,,
\end{align*}
portrayed in Figure~\ref{fig:omega1}: they are continuous, vanish at $\pm1$ and have sum greater than $2$ exactly in $(-3/4,-1/4)\cup(1/4,3/4)$ (corresponding to the dashed part of the figure). Let us denote the resulting $\underline{\OO}$ as $\underline{\OO}^1$. By our previous discussion, $\underline{\OO}^1$ is spear-simple. Additionally, $\underline{\OO}^1$ is not $\s$-connected because the intersection $\big(1+\frac72\s\big)\cap\underline{\OO}^1$ has two connected components. These can be visualized cutting Figure~\ref{fig:omega1} with a horizontal line at level $1=\re\big(1+\frac72\ui\big)$ and are
\begin{align*}
&\{1+7/2(x_1i+x_2j+x_3k)\,|\,x_1i+x_2j+x_3k\in\s, x_3\in(-3/4,-1/4)\}\,,\\
&\{1+7/2(x_1i+x_2j+x_3k)\,|\,x_1i+x_2j+x_3k\in\s, x_3\in(1/4,3/4)\}\,.
\end{align*}
\end{examples}

\begin{figure}[htbp]
\centering
\includegraphics[height=6.5cm]{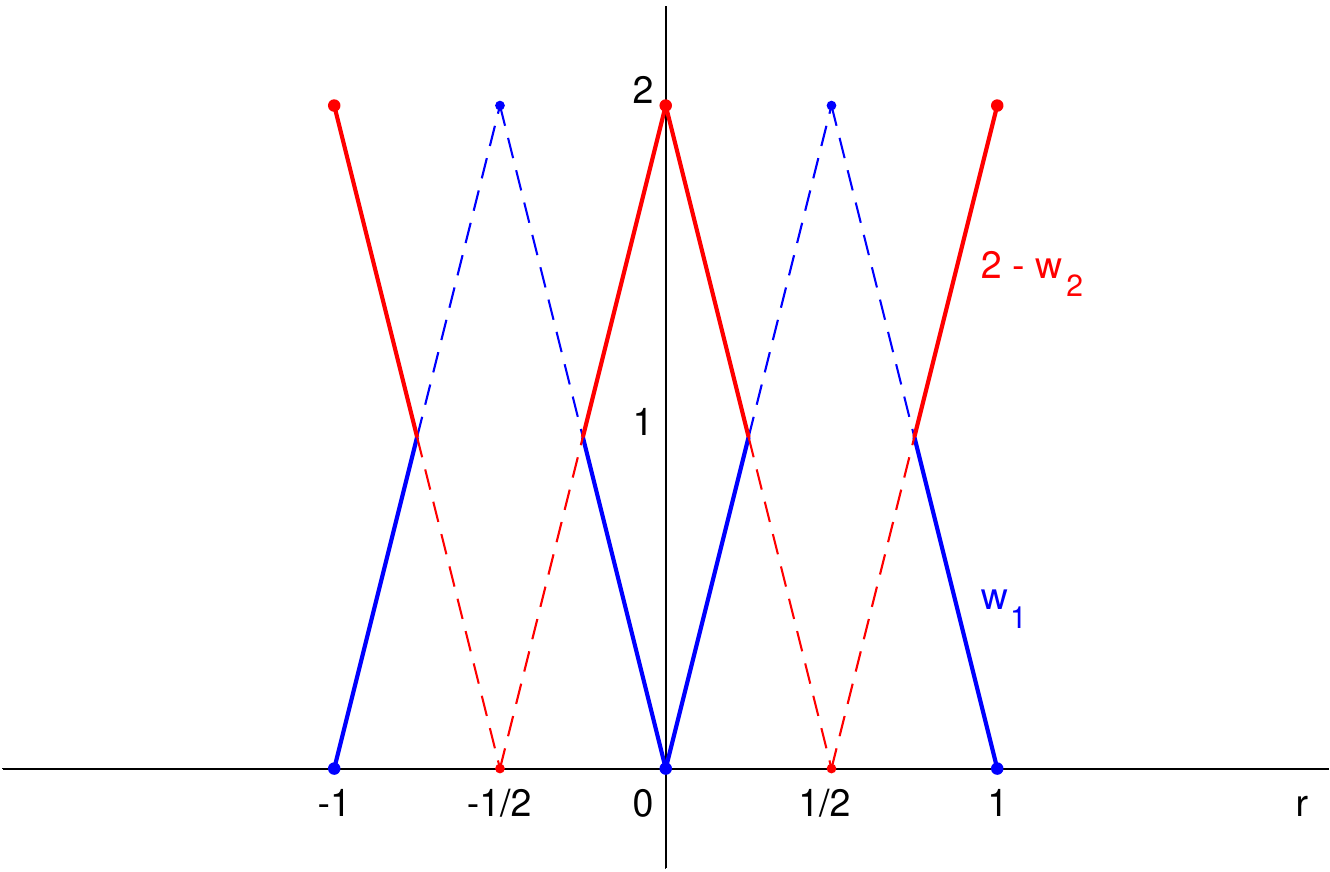}
\caption{The width functions of $\underline{\OO}^1$.}\label{fig:omega1}
\end{figure}

\begin{examples}\label{ex:sconnected}
Consider again the speared domain $\underline{\OO}$ constructed in Examples~\ref{ex:speared}. Assume $w_1,w_2$ to be concave and assume there exists $r_1\in(-1,1)$ such that $\max_{[-1,1]}w_s=w_s(r_1)$ for $s=1,2$. We claim that, for each $\alpha\in[0,2]$, the union of the superlevel sets $w_1^{-1}((\alpha,2])$ and $w_2^{-1}((2-\alpha,2])$ is an open subinterval of $[-1,1]$. Indeed: the concavity assumption guarantees that the superlevel sets $w_1^{-1}((\alpha,2])$ and $w_2^{-1}((2-\alpha,2])$ are open subintervals of $[-1,1]$; the assumption on the maximum guarantees that, if $w_1^{-1}((\alpha,2])$ is not empty, then it includes $r_1$ because $w_1(r_1)=\max_{[-1,1]}w_1>\alpha$ and that, if $w_2^{-1}((2-\alpha,2])$ is not empty, then it includes $r_1$ because $w_2(r_1)=\max_{[-1,1]}w_2>2-\alpha$.

Let us prove that $\underline{\OO}$ is $\s$-connected by fixing $\alpha+\beta\ui\in C$ and arguing that the intersection between the sphere $S:=\alpha+\beta\s$ and $\underline{\OO}$ is connected. This is obviously true if $\alpha+\beta\ui\in R$, since in such a case $S$ is entirely included in $\underline{\OO}$. We therefore assume $\alpha+\beta\ui\in[0,2]\times(3,4)$. In such a case, we see that
\[S\cap\underline{\OO}=\{\alpha+\beta(x_1i+x_2j+x_3k)\, |\, x_1i+x_2j+x_3k\in\s, w_1(x_3)>\alpha\mathrm{\ or\ }w_2(x_3)>2-\alpha\}\]
is connected, as an immediate consequence of the fact that $w_1^{-1}((\alpha,2])\cup w_2^{-1}((2-\alpha,2])$ is an open subinterval of $[-1,1]$.

For instance, the width functions
\begin{figure}[htbp]
\centering
\includegraphics[height=6.5cm]{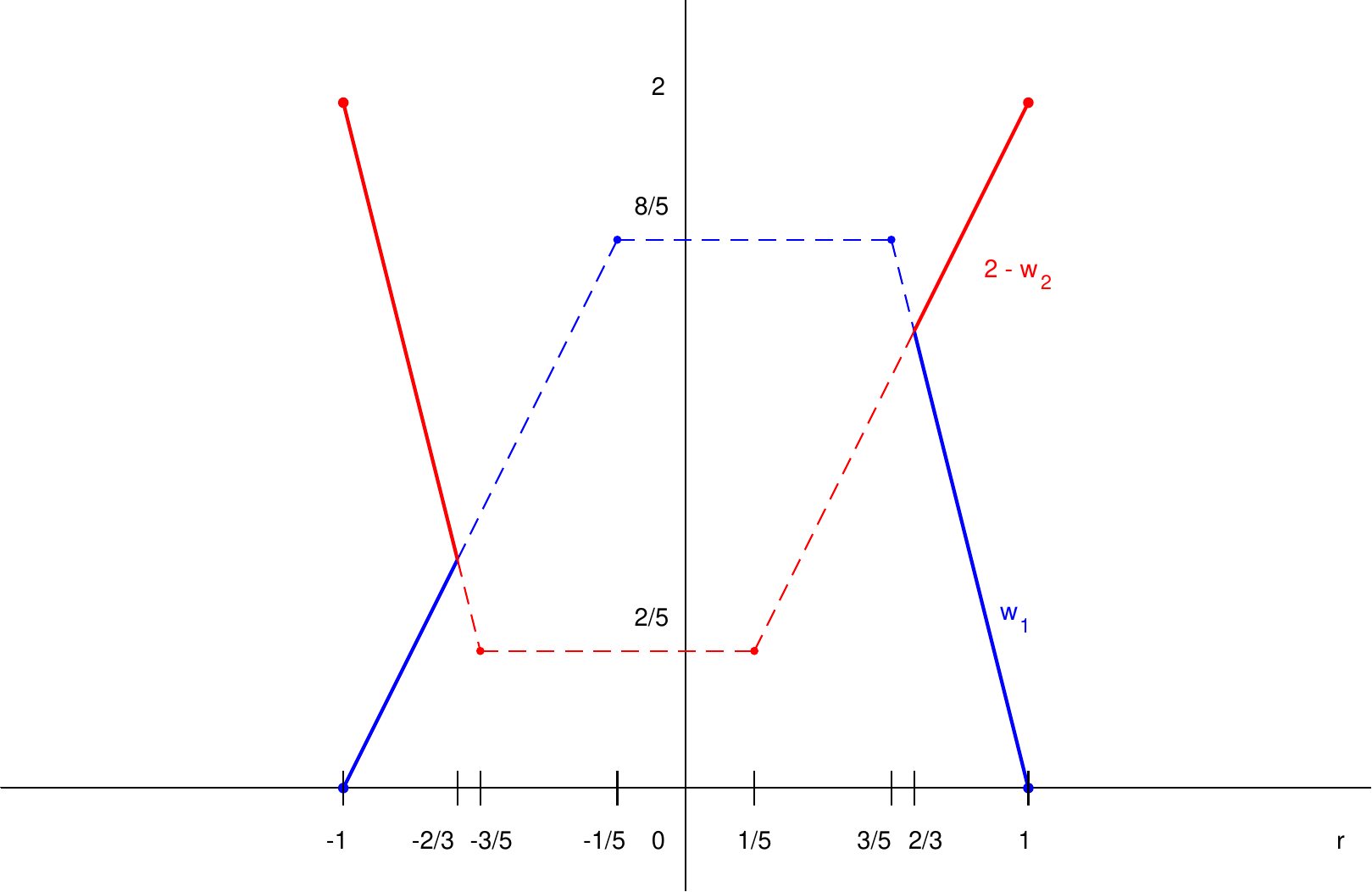}
\caption{The width functions of $\underline{\OO}^2$.}\label{fig:omega2}
\end{figure}
\begin{align*}
&w_1(r):=\left\{
\begin{array}{lll}
2r+2 & \mathrm{\ if\ } & r\in[-1,-1/5]\\
8/5 & \mathrm{\ if\ } & r\in[-1/5,3/5]\\
-4r+4 & \mathrm{\ if\ } & r\in[3/5,1]
\end{array}
\right.\,,\\
&w_2(r):=w_1(-r) \mathrm{\quad for\ all\quad} r\in[-1,1]\,,
\end{align*}
portrayed in Figure~\ref{fig:omega2}, are continuous and concave; they both vanish at $\pm1$ and take their maximum value $8/5$ at $0$; they have sum $w_1+w_2$ greater than $2$ exactly in $(-2/3,2/3)$ (corresponding to the dashed part of the figure). Let us denote the resulting $\underline{\OO}$ as $\underline{\OO}^2$. By our previous discussion, $\underline{\OO}^2$ is $\s$-connected. However, $\underline{\OO}^2$ is not spear-simple. For instance, if we set $J:=\frac1{\sqrt{2}}j+\frac1{\sqrt{2}}k$ and $K:=-J$, then the intersection
\[A:=\phi_J^{-1}\left(\left(\underline{\OO}^2\right)_J^\geqslant\right)\cap\phi_K^{-1}\left(\left(\underline{\OO}^2\right)_K^\geqslant\right)\]
has a connected component that does not meet the real axis.
\begin{figure}
\centering
\includegraphics[height=11.5cm]{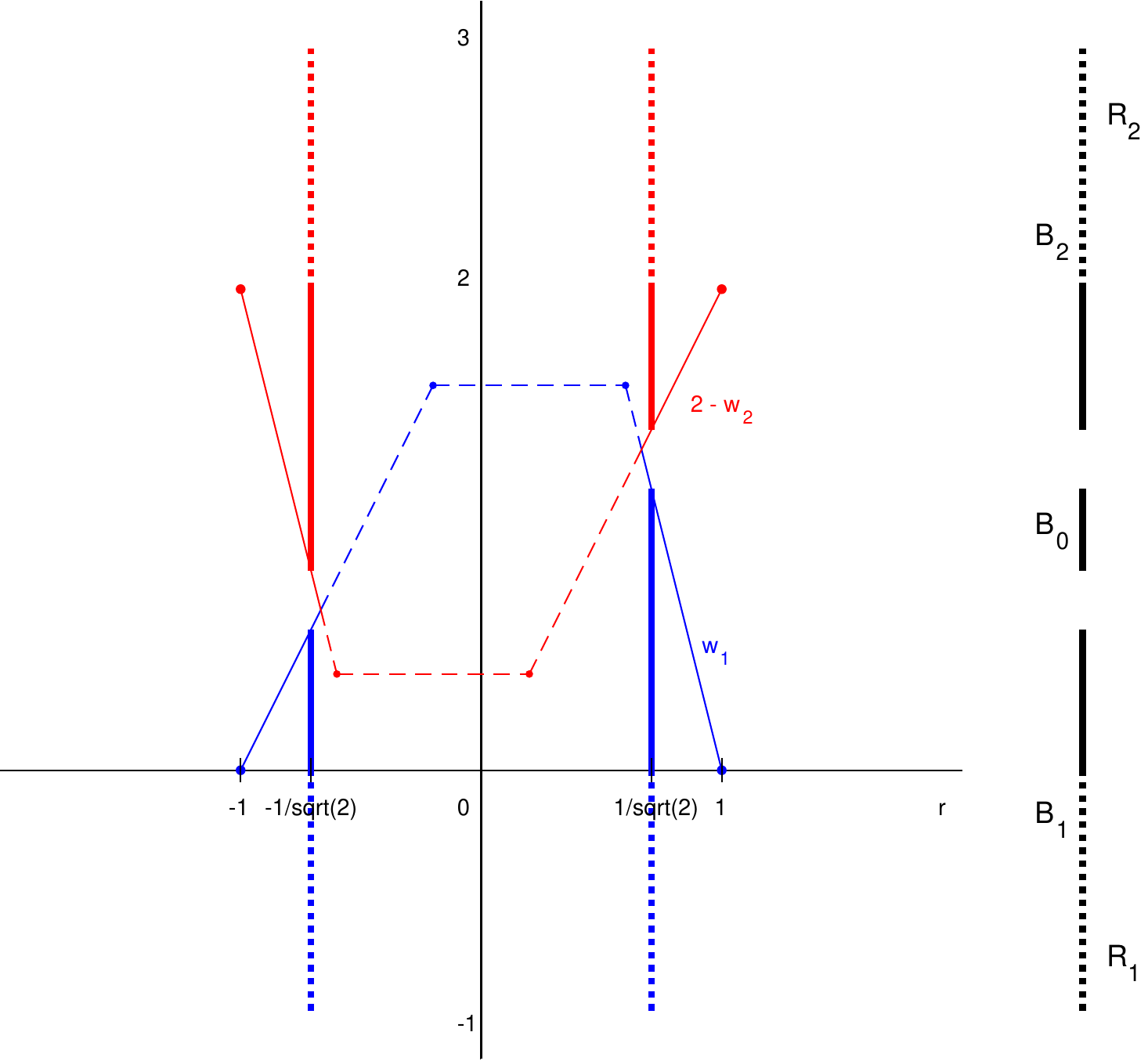}
\caption{A (one-dimensional) portrait of $A=(R_1\cup B_1)\cup B_0\cup(R_2\cup B_2)$.}\label{fig:omega2bis}
\end{figure}
This fact is visible in Figure~\ref{fig:omega2bis} and can be proven by direct computation, as follows. We compute
\[w_1(1/\sqrt{2})=4-2\sqrt{2},\quad w_1(-1/\sqrt{2})=2-\sqrt{2}\,,\]
\[2-w_2(1/\sqrt{2})=2-w_1(-1/\sqrt{2})=\sqrt{2},\quad 2-w_2(-1/\sqrt{2})=2-w_1(1/\sqrt{2})=2\sqrt{2}-2\,.\]
Since $2-\sqrt{2}<2\sqrt{2}-2<4-2\sqrt{2}<\sqrt{2}$, the intersection $A$ is the union between $R=R_1\cup R_2$, the rectangle $B_1:=\left(-1,2-\sqrt{2}\right)\times(3,4)$, the rectangle $B_0:=\left(2\sqrt{2}-2,4-2\sqrt{2}\right)\times(3,4)$ and the rectangle $B_2:=\left(\sqrt{2},3\right)\times(3,4)$. We remark that $B_0$ is a connected component of the intersection $A$ that does not intersect the real axis, as claimed; the other two connected components being the $\Gamma$-shaped figures $R_1\cup B_1$ and $R_2\cup B_2$.
\end{examples}

\begin{example}\label{ex:mainsail2}
Let us choose the width functions $w_1(r):=2-2|r|=:w_2(r)$ (for $r\in[-1,1]$), portrayed in Figure~\ref{fig:omega0}, and denote the resulting speared domain $\underline{\OO}$ as $\underline{\OO}^0$. We remark that $w_1,w_2$ are coinciding, continuous and concave functions on $[-1,1]$, vanishing at $\pm1$ and taking their maximum value $2$ at $0$. Their sum $w_1+w_2$ is greater than $2$ exactly in $(-1/2,1/2)$ (corresponding to the dashed part of the figure). Based on the discussions in Examples~\ref{ex:mainsail},~\ref{ex:spearsimple} and~\ref{ex:sconnected}, the domain $\underline{\OO}^0$ is spear-simple and $\s$-connected; moreover, it has a main sail.
\end{example}

\begin{figure}[htbp]
\centering
\includegraphics[height=6.5cm]{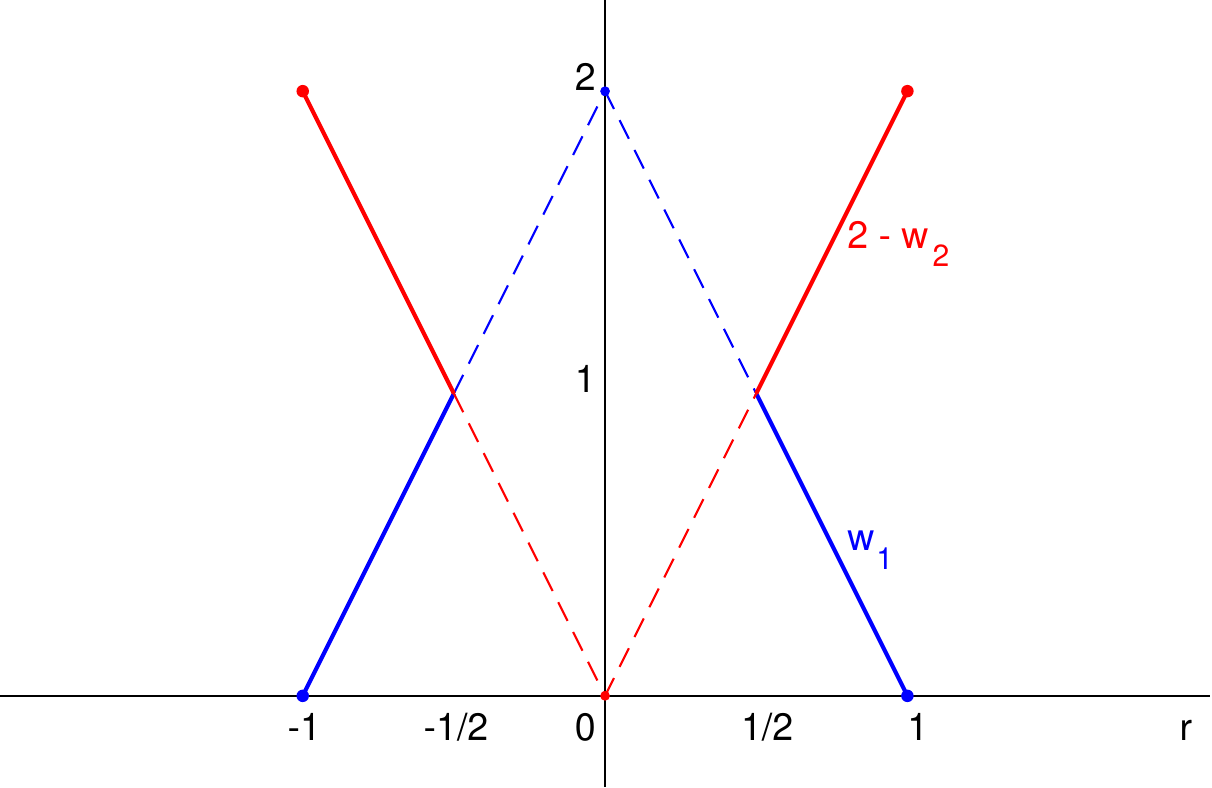}
\caption{The width functions of $\underline{\OO}^0$.}\label{fig:omega0}
\end{figure}

In the next example, we adopt the customary notation $\chi_{(a,b)}$ for the characteristic function of the interval $(a,b)$.

\begin{figure}[htbp]
\centering
\includegraphics[height=6.8cm]{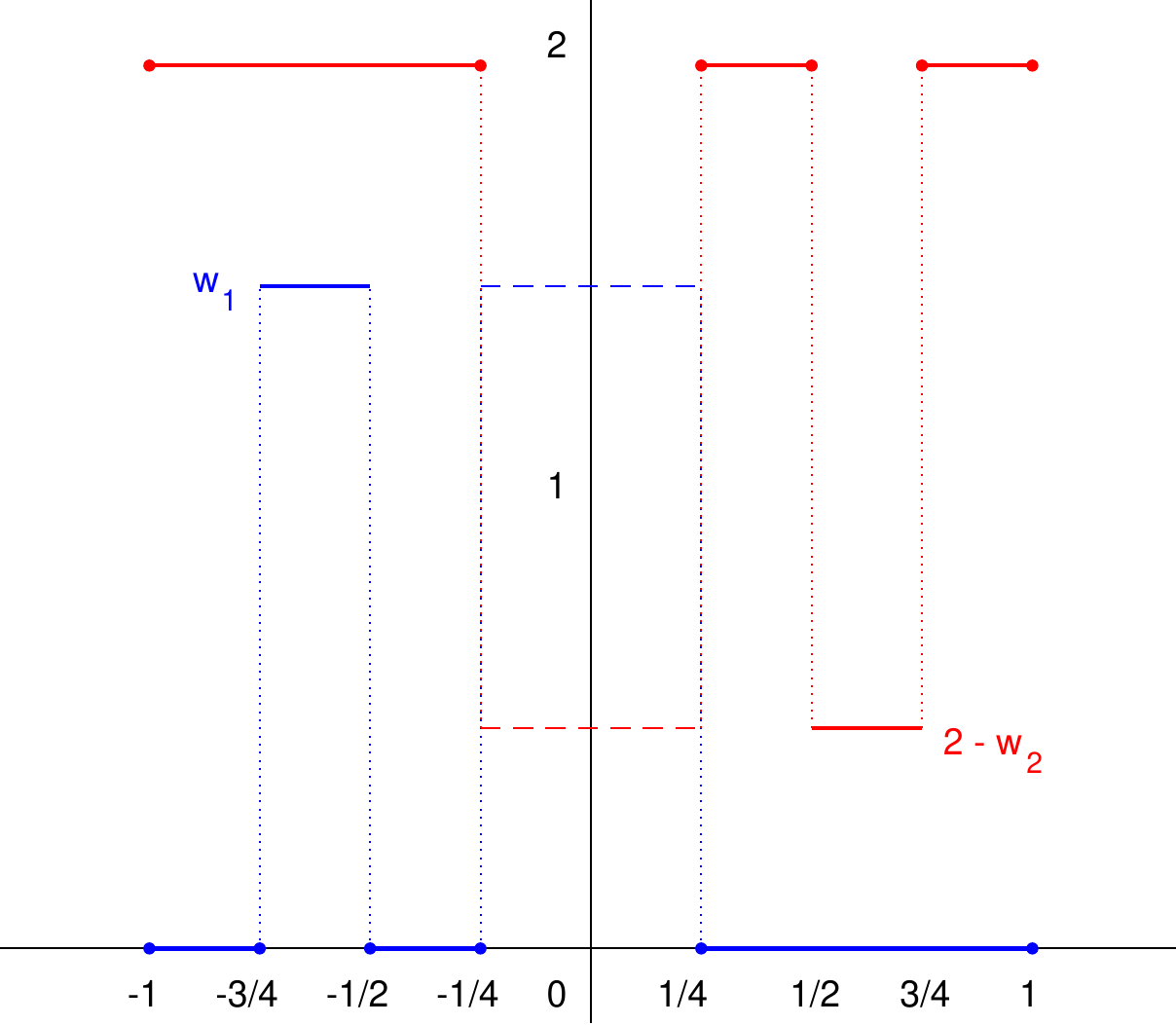}
\caption{The width functions of $\underline{\OO}^3$.}\label{fig:omega3}
\end{figure}

\begin{example}\label{ex:mainsail3}
According to Examples~\ref{ex:speared} and~\ref{ex:mainsail}, we construct a hinged domain $\underline{\OO}^3$ with a main sail, picking the width functions
\begin{align*}
&w_1:=\frac32\chi_{(-3/4,-1/2)}+\frac32\chi_{(-1/4,1/4)}\,,\\
&w_2:=\frac32\chi_{(-1/4,1/4)}+\frac32\chi_{(1/2,3/4)}\,,
\end{align*}
portrayed in Figure~\ref{fig:omega3}. By direct inspection, $w_1,w_2$ are lower semicontinuous, zero at $\pm1$ and fulfill the inequality $w_1+w_2>2$ exactly in the interval $(-1/4,1/4)$ (corresponding to the dashed part of the figure). We claim that $\underline{\OO}^3$ is not spear-simple nor $\s$-connected.

To prove our first claim, we remark that, for all $r\in(-3/4,-1/2)$, the set $D_r$ is the union between the $\Gamma$-shaped figure $\Gamma_1=R_1\cup \big((-1,3/2)\times(3,4)\big)$ and the rectangle $R_2$; while for all $r\in(1/2,3/4)$, the set $D_r$ is the union between the rectangle $R_1$ and the $\Gamma$-shaped figure $\Gamma_2=R_2\cup \big((1/2,3)\times(3,4)\big)$. Since $-\sqrt{2}/2\in(-3/4,-1/2)$ and $\sqrt{2}/2\in(1/2,3/4)$, if we set $J_-:=\sqrt{2}/2j-\sqrt{2}/2k$ and $J_+:=\sqrt{2}/2j+\sqrt{2}/2k$, then we get
\[\phi_{J_-}^{-1}\left(\left(\underline{\OO}^3\right)_{J_-}^\geqslant\right)\cap\phi_{J_+}^{-1}\left(\left(\underline{\OO}^3\right)_{J_+}^\geqslant\right) = R_1\cup\big((1/2,3/2)\times(3,4)\big)\cup R_2\,,\]
where the second connected component $(1/2,3/2)\times(3,4)$ does not meet the real axis. The previous equality can be visualized cutting Figure~\ref{fig:omega3} along the vertical lines $r=-\sqrt{2}/2$ and $r=\sqrt{2}/2$ and observing that the common shade of the two cuts corresponds to the interval from level $1/2$ to level $3/2$.

To prove our second claim, we consider the sphere $S:=1+\frac72\s$. Its intersection with $\underline{\OO}^3$ has three connected components. These components can be visualized cutting Figure~\ref{fig:omega3} with a horizontal line at level $1=\re\big(1+\frac72\ui\big)$ and are
\begin{align*}
&\{1+7/2(x_1i+x_2j+x_3k)\,|\,x_1i+x_2j+x_3k\in\s, x_3\in(-3/4,-1/2)\}\,,\\
&\{1+7/2(x_1i+x_2j+x_3k)\,|\,x_1i+x_2j+x_3k\in\s, x_3\in(-1/4,1/4)\}\,,\\
&\{1+7/2(x_1i+x_2j+x_3k)\,|\,x_1i+x_2j+x_3k\in\s, x_3\in(1/2,3/4)\}\,.
\end{align*}
Our claim is thus proven.
\end{example}

The next remark suggests simple ways to construct further examples of speared domains that are $\s$-connected domains or have a main sail.

\begin{remark}
Let $U$ be a symmetric open subset of $\hh$, let $H$ be a closed subset of $\hh$, let $\OO:=U\setminus H$ and assume $\OO$ to be a speared domain.
\begin{enumerate}
\item If $H$ is a closed half-space, a closed ball, or the complement of an open ball in $\hh$, then $\OO$ is $\s$-connected.
\item If there exists $J_0\in\s$ such that $H\cap\cc_{J_0}^\geqslant=\emptyset$, then $\OO$ has $\OO_{J_0}^\geqslant$ as a main sail.
\end{enumerate}
In both cases, $\OO$ is a hinged domain.
\end{remark}

For each class we introduced thus far in the present section, we have only made use of simple steps when proving it was included in the class of hinged domains. The next result allows to build new examples using double steps, too.

\begin{definition}\label{def:addingidenticalsails}
Let $\OO$ be a hinged domain. Consider a family $\{U_\lambda\}_{\lambda\in\Lambda}$ of open subsets of $\hh$ with the following properties:
\begin{enumerate}
\item for all $\lambda\in\Lambda$, the set $U_\lambda$ is $\s$-connected and has a main sail $(U_\lambda)_{J_\lambda}^\geqslant$;
\item there exists an open subset $D'$ of $\rr_\cc$ such that $(U_\lambda)_{J_\lambda}^\geqslant=\phi_{J_\lambda}(D')$ for all $\lambda\in\Lambda$;
\item there exists a set $D$, which intersects every connected component of $D'$, such that $\OO_{J_\lambda}^\geqslant\supseteq\phi_{J_\lambda}(D)$ for all $\lambda\in\Lambda$;
\item $\OO_{D'\setminus D}$ does not intersect $\OO$.
\end{enumerate}
In such a case, we say that the union
\[\OO':=\OO\cup\bigcup_{\lambda\in\Lambda}U_\lambda\,,\]
is \emph{obtained from $\OO$ by adding identical sails}. If $\Lambda$ is a singleton, we also say that $\OO'$ is \emph{obtained from $\OO$ by adding a sail}.
\end{definition}

\begin{remark}\label{rmk:addingidenticalsails}
In the situation described in Definition~\ref{def:addingidenticalsails}, the inclusion
\[\OO'\setminus\OO\subseteq\OO_{D'}\]
holds. This follows from the inclusion $U_\lambda\subseteq\OO_{D'}$ valid for each $\lambda\in\Lambda$, which, in turn, is a consequence of the chain of inclusions
\[\phi_{J}^{-1}\left((U_\lambda)_{J}^\geqslant\right)\subseteq\phi_{J_\lambda}^{-1}\left((U_\lambda)_{J_\lambda}^\geqslant\right)=D'\]
valid for all $J\in\s$.
\end{remark}

\begin{proposition}\label{prop:addingidenticalsails}
Let $\OO\subseteq\hh$ be a hinged domain. If $\OO'$ is a speared domain obtained from $\OO$ by adding identical sails, then $\OO'$ is a hinged domain.
\end{proposition}

\begin{proof}
We adopt the notations of Definition~\ref{def:addingidenticalsails} and assume, additionally, $\OO'$ to be a speared domain. We have to prove that, for all $\alpha,\beta\in\rr$ with $\beta\geq0$ and for every two points $x=\alpha+\beta J,y=\alpha+\beta K$ belonging to the intersection between the sphere $S:=\alpha+\beta\s$ and $\OO'$, it holds $x\sim y$. This is obviously true if $x,y\in\OO$, because $\OO$ is a hinged domain by hypothesis. We therefore assume that at least one among $x,y$ does not belong to $\OO$, without loss of generality $x$. In particular, there exists $\lambda\in\Lambda$ such that $x\in U_\lambda\setminus\OO$. For future use, we define $x_\lambda:=\alpha+\beta J_\lambda$: since $U_\lambda$ is $\s$-connected, the points $x,x_\lambda\in S\cap U_\lambda$ must belong to the same connected component of $S\cap\OO'$, whence $x\sim x_\lambda$ in $\OO'$. Since $x\in\OO'\setminus\OO$, Remark~\ref{rmk:addingidenticalsails} implies that $x\in\OO_{D'}$, whence $\alpha+\beta\ui\in D'$. If, moreover, $\alpha+\beta\ui\in D$, then
\[x_\lambda=\alpha+\beta J_\lambda\in\phi_{J_\lambda}(D)\subseteq\OO_{J_\lambda}^\geqslant\subset\OO\,.\]
We separate two cases:
\begin{itemize}
\item Suppose $y\in\OO$. Since $\OO_{D'\setminus D}$ does not intersect $\OO$, it follows that $\alpha+\beta\ui\in D$. Thus, the point $x_\lambda$ belongs to $\OO$. Since $\OO$ is a hinged domain, it follows that $x_\lambda\sim y$ in $\OO$, whence in $\OO'$. Since we already established that $x\sim x_\lambda$ in $\OO'$, it follows that $x\sim y$ in $\OO'$, as desired.
\item Suppose $y\not\in\OO$, so that there exists $\mu\in\Lambda$ with $y\in U_\mu\setminus\OO$. If we set $y_\mu:=\alpha+\beta J_\mu$, then, reasoning as before, we find that $y_\mu\sim y$ in $\OO'$. We claim that $x_\lambda\sim y_\mu$ in $\OO'$. The chain of equivalences $x\sim x_\lambda\sim y_\mu\sim y$ in $\OO'$ yields the thesis.\\
We prove our claim as follows. If $\alpha+\beta\ui\in D$, then $x_\lambda,y_\mu\in\OO$, whence $x_\lambda\sim y_\mu$ in $\OO$ and in $\OO'$, as claimed. Now assume $\alpha+\beta\ui\in D'\setminus D$: the connected component of $D'$ including the point $\alpha+\beta\ui$ intersects $D$ at some point $\alpha'+\beta'\ui\neq\alpha+\beta\ui$: we can thus pick a path in $D'$ from $\alpha+\beta\ui$ to $\alpha'+\beta'\ui$. The points
\begin{align*}
&x'_\lambda:=\alpha'+\beta' J_\lambda\in\phi_{J_\lambda}(D)\subseteq\OO_{J_\lambda}^\geqslant\,,\\
&y'_\mu:=\alpha'+\beta' J_\mu\in\phi_{J_\mu}(D)\subseteq\OO_{J_\mu}^\geqslant
\end{align*}
both belong to the hinged domain $\OO$. Thus, $x'_\lambda\sim y'_\mu$ in $\OO$, i.e., there exists a chain $\{p_s\}_{s=1}^{t-1}$ from $x'_\lambda$ to $y'_\mu$ in $\OO\subseteq\OO'$. Now, consider again the path we have picked in $D'$: since $\phi_{J_\lambda}(D')\subseteq(\OO')_{J_\lambda}^\geqslant,\phi_{J_\mu}(D')\subseteq(\OO')_{J_\mu}^\geqslant$, it follows that $(y_\mu,y'_\mu)$ shadows $(x_\lambda,x'_\lambda)$ in $\OO'$. Thus, if we set $p_0:=x_\lambda$ and $p_t:=y_\mu$, then $\{p_s\}_{s=0}^{t}$ is a chain from $x_\lambda$ to $y_\mu$ in $\OO'$ with a double step at $(0,t-1)$. It follows that $x_\lambda\sim y_\mu$, as claimed.\qedhere
\end{itemize}
\end{proof}

We can now provide examples of hinged domains that do not have a main sail. We do so by adding identical sails to previously constructed examples.

\begin{figure}[htbp]
\centering
\includegraphics[height=7cm]{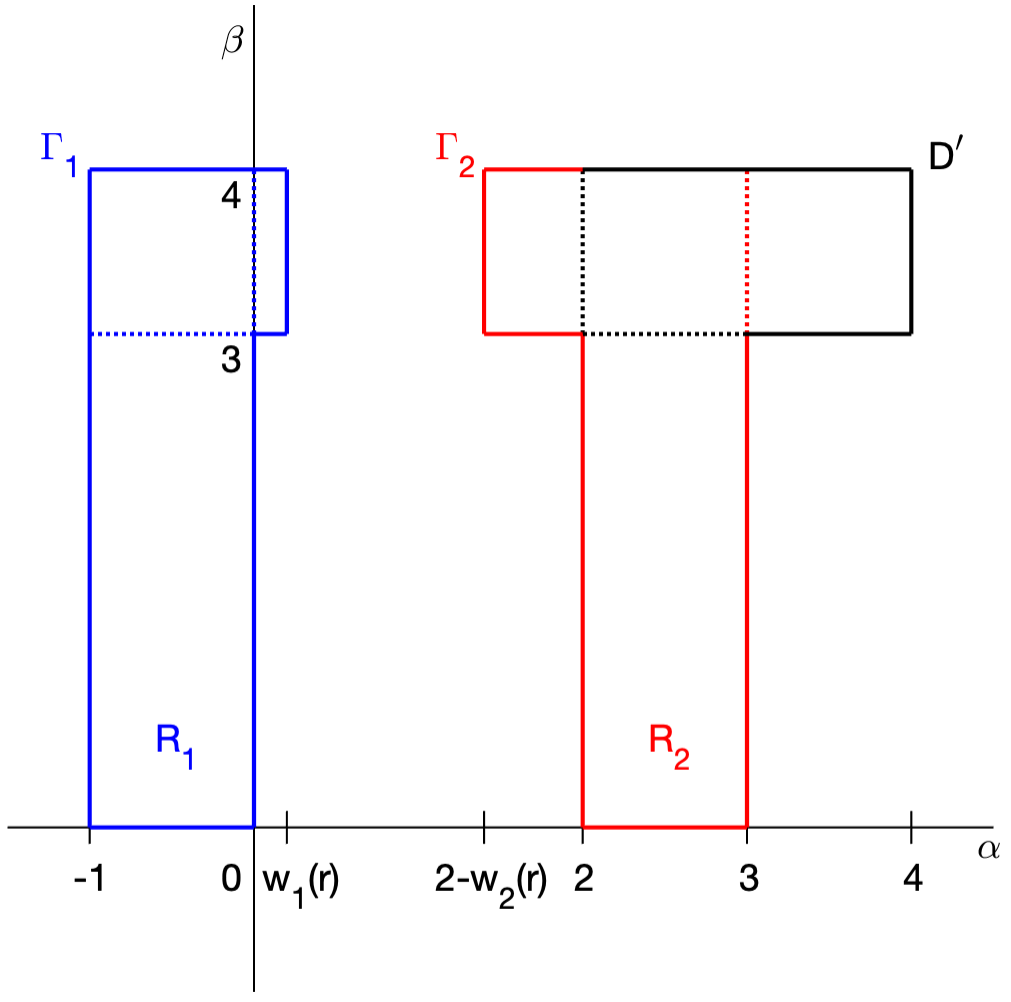}
\caption{The planar domain $D_r\cup D'$ of Example~\ref{ex:nomainsail}.}\label{fig:DrandDprime}
\end{figure}

\begin{examples}\label{ex:nomainsail}
Each speared domain $\underline{\OO}$ constructed in Examples~\ref{ex:speared} is a hinged domain because it has a main sail, as proven in Examples~\ref{ex:mainsail}. We additionally assume we can pick $\rho$ such that $w_1(r)+w_2(r)\leq2$ when $r\geq\rho$ (this is always the case if $w_1+w_2$ is continuous, since $w_1(1)+w_2(1)=0$). We let $D'$ denote the the rectangle $(2,4)\times(3,4)$ which is portrayed in black in Figure~\ref{fig:DrandDprime} and contains the square $D:=(2,3)\times(3,4)\subset R_2$ (dotted, in black and red). We set
\begin{align*}
U_\rho:=\bigcup_{x_1i+x_2j+x_3k\in\s,\,x_3>\rho}\phi_{x_1i+x_2j+x_3k}(D')
\end{align*}
and make the following remarks.
\begin{itemize}
\item The union $\underline{\OO}':=\underline{\OO}\cup U_\rho$ is still a speared domain and is obtained from $\underline{\OO}$ by adding a sail. Indeed, $U_\rho$ is an $\s$-connected open subset of $\hh$ with a main sail $(U_\rho)_{k}^\geqslant=\phi_{k}(D')$. Moreover, $\phi_{k}(D)\subseteq\underline{\OO}_{k}$; actually, $\OO_D\subset\underline{\OO}$. Finally, the circularization $\OO_{D'\setminus D}$ of $D'\setminus D$ (which is the part of $D'$ to the right of the red dotted line) does not intersect $\underline{\OO}$.
\item By Proposition~\ref{prop:addingidenticalsails}, the speared domain $\underline{\OO}'$ is a hinged domain.
\item There exists no $J=x_1i+x_2j+x_3k\in\s$ such that $(\underline{\OO}')_{J}^\geqslant$ is a main sail for $\underline{\OO}'$. If such a main sail existed, it would have to include both a copy $\phi_J(D')$ of the added sail and the $C$-shaped figure $\phi_J(C)$. The former inclusion would imply $x_3>\rho$; the latter inclusion would imply $w_1(x_3)+w_2(x_3)>2$; this would contradict our construction.
\item The hinged domain $\underline{\OO}'$ is spear-simple if, and only if, $\underline{\OO}$ is. Indeed, take any distinct $J=x_1i+x_2j+x_3k,K=y_1i+y_2j+y_3k\in\s$. If $x_3\leq\rho$ or $y_3\leq\rho$, then the intersections
\begin{align*}
&A:=\phi_J^{-1}\big(\big(\underline{\OO}\big)_J^\geqslant\big)\cap\phi_K^{-1}\big(\big(\underline{\OO}\big)_K^\geqslant\big)\,,\\
&B:=\phi_J^{-1}\big(\big(\underline{\OO}'\big)_J^\geqslant\big)\cap\phi_K^{-1}\big(\big(\underline{\OO}'\big)_K^\geqslant\big)
\end{align*}
coincide. Only when $x_3,y_3>\rho$: one of the connected components of $A$, namely the $\Gamma$-shaped component including $R_2$ and intersecting the real axis in the interval $(2,3)$, changes into a larger $T$-shaped component of $B$ including $R_2\cup D'$, still intersecting the real axis in $(2,3)$; each further connected component of $A$ equals one of the remaining connected components of $B$.
\item The hinged domain $\underline{\OO}'$ is $\s$-connected if, and only if, $\underline{\OO}$ is. Indeed, for every sphere $S:=\alpha+\beta\s$ intersecting $U_\rho$: either $S\subset\OO_D\subset\underline{\OO}\subset\underline{\OO}'$; or $S$ is contained in $\OO_{D'\setminus D}$, which does not intersect $\underline{\OO}$, whence $S\cap\underline{\OO}'=S\cap U_\rho$ (which is connected).
\end{itemize}
In particular, if we consider the spear-simple domain $\underline{\OO}^0$ constructed in Example~\ref{ex:mainsail2} (which was $\s$-connected) and we set $(\underline{\OO}^0)':=\underline{\OO}^0\cup U_{1/2}$, then $(\underline{\OO}^0)'$ is an example of spear-simple domain that is $\s$-connected but has no main sail. If we consider the spear-simple domain $\underline{\OO}^1$ constructed in Example~\ref{ex:spearsimple} (which was not $\s$-connected) and we set $(\underline{\OO}^1)':=\underline{\OO}^1\cup U_{3/4}$, then $(\underline{\OO}^1)'$ is an example of spear-simple domain that is not $\s$-connected nor has a main sail. If we consider the $\s$-connected speared domain $\underline{\OO}^2$ constructed in Example~\ref{ex:sconnected} (which was not spear-simple) and we set $(\underline{\OO}^2)':=\underline{\OO}^2\cup U_{2/3}$, then $(\underline{\OO}^2)'$ is an example of $\s$-connected speared domain that is not spear-simple nor has a main sail.
\end{examples}

In our last example of hinged domain, double steps are necessary. In particular, the hinged domain we construct has no main sail and is not spear-simple nor $\s$-connected.

\begin{figure}[htbp]
\centering
\begin{minipage}{5.5cm}
  \centering
  \includegraphics[width=5cm]{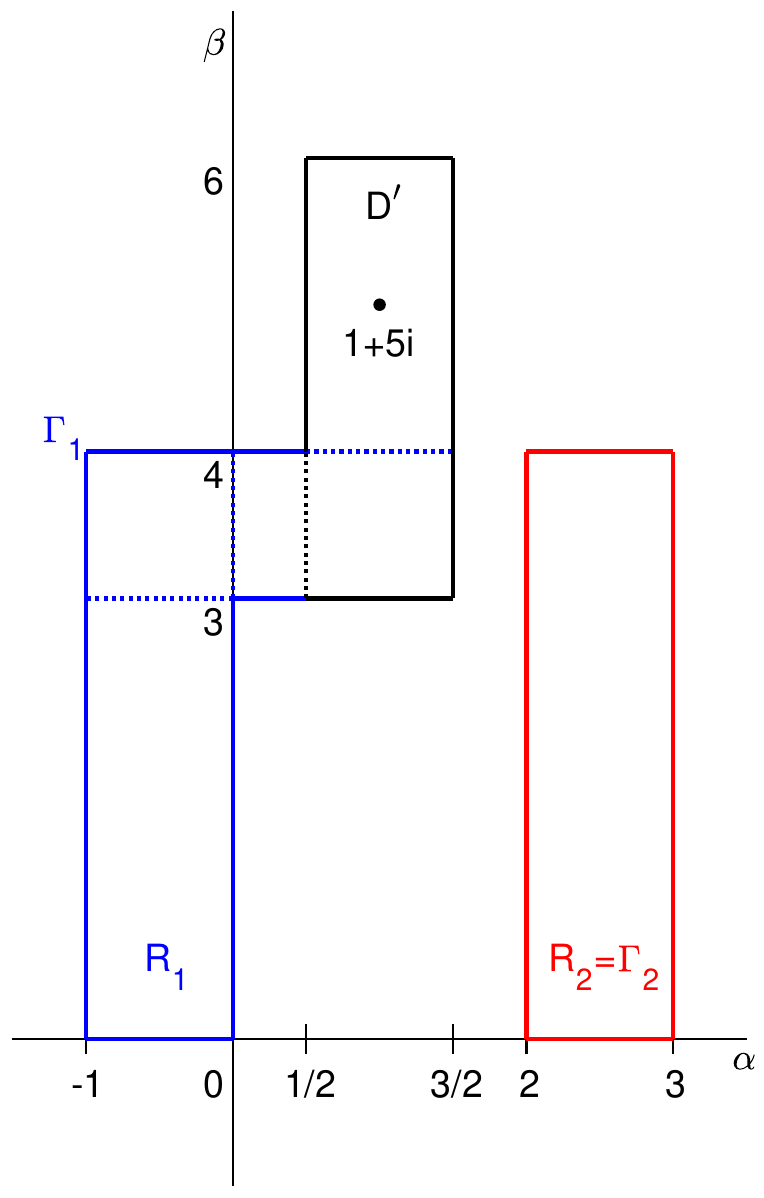}
\end{minipage}%
\begin{minipage}{5.5cm}
  \centering
  \includegraphics[width=5cm]{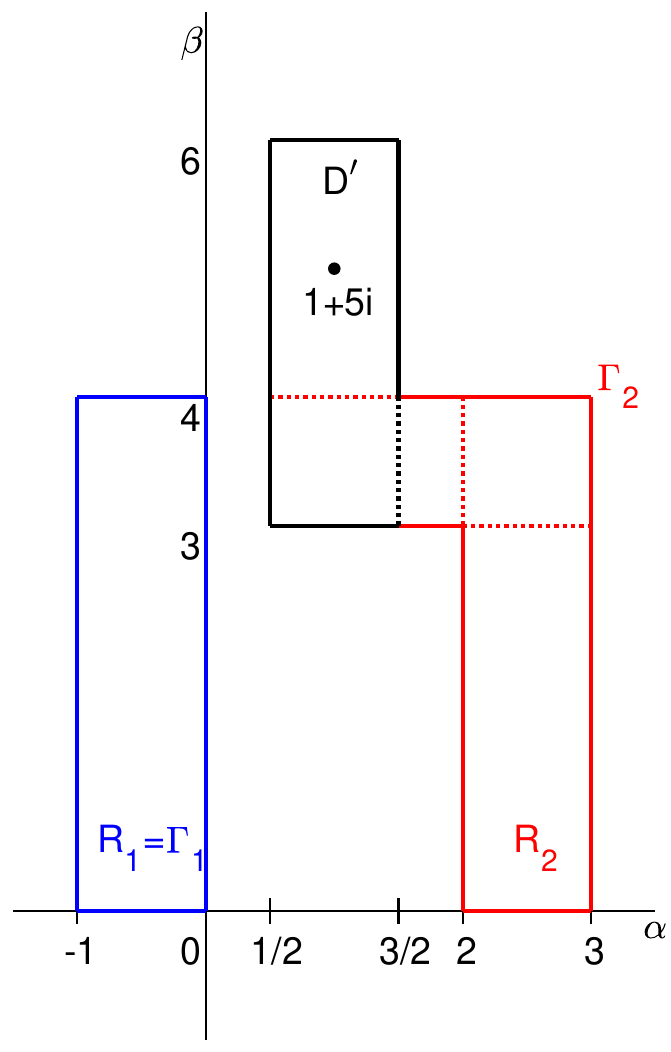}
\end{minipage}
\caption{The planar domains $D_{y_3}\cup D'$ and $D_{y'_3}\cup D'$ of Example~\ref{ex:doublestep}, with $y_3\in(-3/4,-1/2)$ and for $y_3'\in(1/2,3/4)$.}\label{fig:DminusplusandDprime}
\end{figure}

\begin{example}\label{ex:doublestep}
Consider again the speared domain $\underline{\OO}^3$ with a main sail constructed in Example~\ref{ex:mainsail3}. Let $D'$ denote the rectangle $(1/2,3/2)\times(3,6)$, which is portrayed twice, in black, in Figure~\ref{fig:DminusplusandDprime} and contains the square $D:=(1/2,3/2)\times(3,4)$ (dotted). Set
\begin{align*}
&U_-:=\bigcup\limits_{\substack{x_1i+x_2j+x_3k\in\s,\\ x_3\in(-3/4,-1/2)}}\phi_{x_1i+x_2j+x_3k}(D')\,,\\
&U_+:=\bigcup\limits_{\substack{x_1i+x_2j+x_3k\in\s,\\ x_3\in(1/2,3/4)}}\phi_{x_1i+x_2j+x_3k}(D')\,.
\end{align*}
We recall that we have set $J_-:=\sqrt{2}/2j-\sqrt{2}/2k$ and $J_+:=\sqrt{2}/2j+\sqrt{2}/2k$ and make the following remarks.
\begin{itemize}
\item The union $(\underline{\OO}^3)':=\underline{\OO}^3\cup U_- \cup U_+$ is still a speared domain and is obtained from $\underline{\OO}^3$ by adding identical sails. Indeed, $U_\pm$ is an $\s$-connected open subset of $\hh$. The half slice $(U_\pm)_{J_\pm}^{\geqslant}=\phi_{J_\pm}(D')$ is a main sail for $U_\pm$. The equality $w_1(-\sqrt{2}/2)=3/2$ implies that $\phi_{J_-}(D)\subseteq\phi_{J_-}(D_{-\sqrt{2}/2})=(\underline{\OO}^3)_{J_-}^\geqslant$, while the equality $2-w_2(\sqrt{2}/2)=1/2$ implies that $\phi_{J_+}(D)\subseteq\phi_{J_+}(D_{\sqrt{2}/2})=(\underline{\OO}^3)_{J_+}^\geqslant$. Moreover, the circularization $\OO_{D'\setminus D}$ of $D'\setminus D$ (which is the part of $D'$ above the colored dotted line) does not intersect $\underline{\OO}^3$.
\item By Proposition~\ref{prop:addingidenticalsails}, the speared domain $(\underline{\OO}^3)'$ is a hinged domain.
\item We can exhibit two hinged points in $(\underline{\OO}^3)'$, namely $1+5J_-\in U_-$ and $1+5J_+\in U_+$, such that every chain connecting them must include a double step. This is because the intersection of the sphere $S:=1+5\s$ with $(\underline{\OO}^3)'$ has two connected components, namely $S\cap U_-$ and $S\cap U_+$ and because every chain comprising only simple steps must be entirely contained in either component. Indeed, there can only be a simple step between two points belonging to separate connected components if these two points are strongly hinged, but we can show that this is impossible in our case. If $1+5J\in S\cap U_-$, then $J=y_1i+y_2j+y_3k\in\s$ with $y_3\in(-3/4,-1/2)$, whence
\[\phi_J^{-1}(((\underline{\OO}^3)')_J^\geqslant)=D_{y_3}\cup D'=\Gamma_1\cup D'\cup R_2\]
(see the left part of Figure~\ref{fig:DminusplusandDprime}). If $1+5J'\in S\cap U_+$, then $J'=y'_1i+y'_2j+y'_3k\in\s$ with $y_3'\in(1/2,3/4)$, whence
\[\phi_{J'}^{-1}(((\underline{\OO}^3)')_{J'}^\geqslant)=D_{y'_3}\cup D'=R_1\cup D'\cup\Gamma_2\]
(see the right part of Figure~\ref{fig:DminusplusandDprime}). Within the intersection
\[\phi_J^{-1}(((\underline{\OO}^3)')_J^\geqslant)\cap\phi_{J'}^{-1}(((\underline{\OO}^3)')_{J'}^\geqslant)=R_1\cup D'\cup R_2\,,\]
the connected component of $1+5\ui$ is $D'$, which does not meet the real axis. Our claim that $1+5J'$ cannot be strongly hinged to $1+5J$ now follows from part {\it 2.} of Lemma~\ref{lem:stronglyhinged}.
\item As a consequence of the previous discussion, $(\underline{\OO}^3)'$ is not spear-simple nor $\s$-connected and does not have a main sail.
\end{itemize}
\end{example}

We are in a position to provide explicit examples of chains of length $t>1$.

\begin{figure}[htbp]
\centering
\begin{minipage}{4.9cm}
  \centering
  \includegraphics[width=4cm]{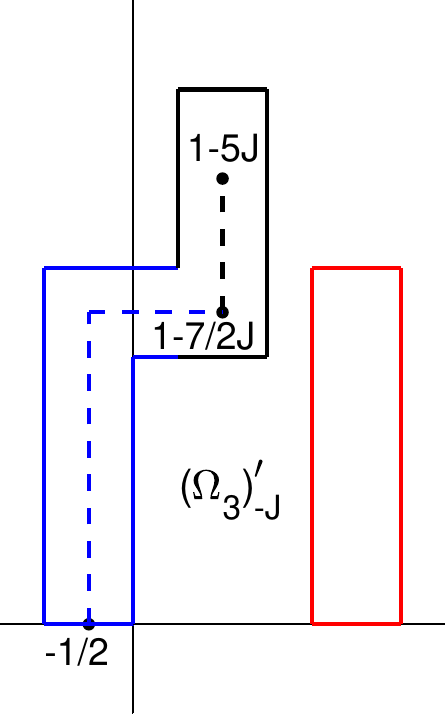}
\end{minipage}%
\begin{minipage}{4.9cm}
  \centering
  \includegraphics[width=4cm]{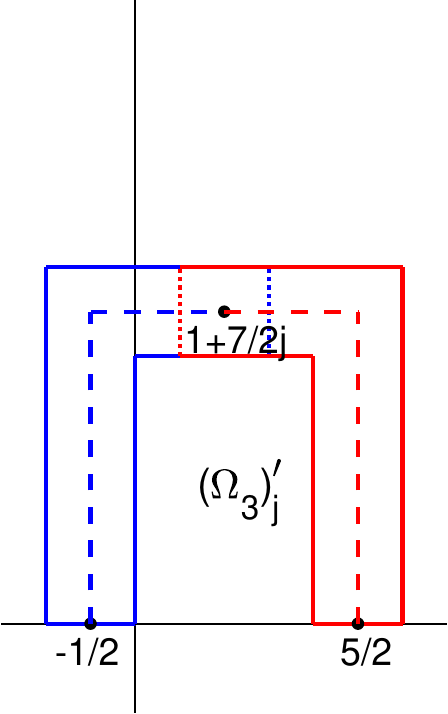}
\end{minipage}%
\begin{minipage}{4.9cm}
  \centering
  \includegraphics[width=4cm]{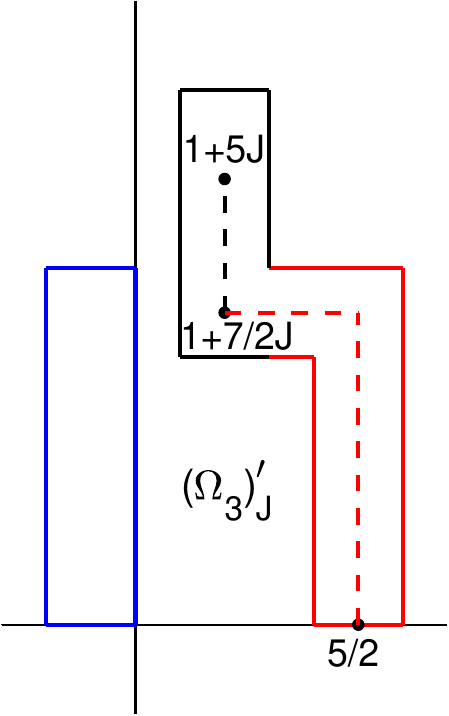}
\end{minipage}
\caption{Within $(\underline{\Omega}_3)'$: there is a simple step from $x_1=1-7/2J$ to $x_2=1+7/2j$ because they are strongly hinged, as shown by the dashed blue paths; there is a simple step from $x_2=1+7/2j$ to $x_3=1+7/2J$ because they are strongly hinged, as shown by the dashed red paths; while the dashed black segments from $x_0=1-5J$ to $x_1=1-7/2J$ and from $x_3=1+7/2J$ to $x_4=1+5J$ correspond to a double step.}\label{fig:chain}
\end{figure}

\begin{example}\label{ex:chain}
Let $J:=\sqrt{2}/2j+\sqrt{2}/2k$ and let us construct a chain connecting $1-5J$ to $1+5J$ in $(\underline{\Omega}_3)'$. Since $\sqrt{2}/2\in(1/2,3/4)$, we already know that these two points cannot be connected by a simple step. If we set
\[x_0:=1-5J,\ x_1:=1-7/2J,\ x_2:=1+7/2j,\ x_3:=1+7/2J,\ x_4:=1+5J\,,\]
then $\{x_t\}_{t=0}^4$ is a chain of length $4$ connecting $x_0$ to $x_4$ in $(\underline{\Omega}_3)'$ (portrayed in Figure~\ref{fig:chain}). Indeed: there is a simple step at $1$ because there is a path from $1+7/2\ui$ to $-1/2$ in $E^{-J,j}\subset\rr_\cc$; there is a simple step at $2$ because there is a path from $1+7/2\ui$ to $5/2$ in $E^{j,J}\subset\rr_\cc$; and there is a double step at $(0,3)$ because the line segment from $1+5\ui$ to $1+7/2\ui$ is entirely contained in $E^{-J,J}\subset\rr_\cc$.

We remark that $\{x_t\}_{t=1}^3$ is a chain of length $2$ connecting $x_1$ to $x_3$ in $\underline{\Omega}_3\subset(\underline{\Omega}_3)'$. Moreover, the points $x_1=1-7/2J$ and $x_3=1+7/2J$ cannot be connected by a chain of length $1$ in $(\underline{\Omega}_3)'\supset\underline{\Omega}_3$ because: they belong to distinct connected components of $(1+7/2\s)\cap(\underline{\Omega}_3)'$; they are not strongly hinged because the point $1+7/2\ui$ belongs to the connected component $D'$ of $E^{-J,J}\subset\rr_\cc$, which has $D'\cap\rr=\emptyset$.
\end{example}

We conclude with Table~\ref{table:cases}, which recaps the properties of the speared domains $\underline{\OO}^s$ and $(\underline{\OO}^s)'$ for $s=0,1,2,3$.

\begin{table}[htbp]
\centering
\begin{tabular}{|c|c|c|c|c|}
\hline
&{\rm spear-simple} & {\rm $\s$-connected} & {\rm with a main sail} & {\rm hinged domain} \\
\hline
$\underline{\OO}^0$ & \checkmark & \checkmark & \checkmark & \checkmark\\
\hline
$\underline{\OO}^1$ & \checkmark & $\times$ & \checkmark & \checkmark\\
\hline
$\underline{\OO}^2$ & $\times$ & \checkmark & \checkmark & \checkmark\\
\hline
$\underline{\OO}^3$ & $\times$ & $\times$ & \checkmark & \checkmark\\
\hline
$(\underline{\OO}^0)'$ & \checkmark & \checkmark & $\times$ & \checkmark\\
\hline
$(\underline{\OO}^1)'$ & \checkmark & $\times$ & $\times$ & \checkmark\\
\hline
$(\underline{\OO}^2)'$ & $\times$ & \checkmark & $\times$ & \checkmark\\
\hline
$(\underline{\OO}^3)'$ & $\times$ & $\times$ & $\times$ & \checkmark\\
\hline
\end{tabular}
\caption{Properties of the examples of speared domains constructed in the present subsection.}\label{table:cases}
\end{table}

\vfill
\section*{Acknowledgements}

This work was partly supported by INdAM, through: GNSAGA; INdAM project ``Hypercomplex function theory and applications''. The first author was also partly supported by MIUR, through: PRIN 2017 ``Moduli theory and birational classification''. The second author was also partly supported by MIUR, through: Finanziamento Premiale FOE 2014 ``Splines for accUrate NumeRics: adaptIve models for Simulation Environments''; PRIN 2017 ``Real and complex manifolds: topology, geometry and holomorphic dynamics''; PRIN 2022 ``Real and complex manifolds: geometry and holomorphic dynamics''.

The authors warmly acknowledge the role of the anonymous referee, who provided careful and useful suggestions about the presentation.

\vfill
\section*{Statements required by the publisher}

On behalf of all authors, the corresponding author states that there is no conflict of interest.
Data sharing is not applicable to this article as no datasets were generated or analyzed during the current study.

\vfill
\newpage


\end{document}